%% file: main.tex
\documentclass{article}

\usepackage{microtype}
\usepackage{graphicx, xcolor}
\usepackage{subfigure}
\usepackage{booktabs} 
\usepackage{nicefrac}

\newcommand{\squeeze}{\textstyle} 

\newcommand{\alertR}[1]{{\color{red}#1}}

\usepackage{colortbl}
\definecolor{bgcolor}{rgb}{0.85, 1.00, 0.85}

\usepackage{hyperref}

\usepackage[accepted]{icml2021}

\icmltitlerunning{ADOM}

\input{commands}

\begin{document}

\twocolumn[
\icmltitle{ADOM: Accelerated Decentralized Optimization Method \\ for Time-Varying Networks}




\begin{icmlauthorlist}
\icmlauthor{Dmitry Kovalev}{kaust}
\icmlauthor{Egor Shulgin}{kaust}
\icmlauthor{Peter Richt\'{a}rik}{kaust}
\icmlauthor{Alexander Rogozin}{mipt}
\icmlauthor{Alexander Gasnikov}{mipt}
\end{icmlauthorlist}

\icmlaffiliation{kaust}{King Abdullah University of Science and Technology, Thuwal, Saudi Arabia}
\icmlaffiliation{mipt}{Moscow Institute of Physics and Technology, Dolgoprudny, Russia}

\icmlcorrespondingauthor{Dmitry Kovalev}{dakovalev1@gmail.com}


\vskip 0.3in
]



\printAffiliationsAndNotice{}  

\begin{abstract}
We propose {\sf ADOM} -- an accelerated method for smooth and strongly convex decentralized optimization over time-varying networks. {\sf ADOM} uses a dual oracle, i.e., we assume access to the gradient of the Fenchel conjugate of the individual loss functions. Up to a constant factor, which depends on the network structure only, its communication complexity is the same as that of accelerated Nesterov gradient method \citep{nesterov2003introductory}. To the best of our knowledge,  only the algorithm of \citet{rogozin2019optimal}  has a  convergence rate with similar properties. However, their algorithm converges under the very restrictive assumption that the number of network changes can not be greater than a tiny percentage of the number of iterations. This assumption is hard to satisfy in practice, as the network topology changes usually can not be controlled. In contrast, {\sf ADOM} merely requires the network to stay connected throughout time. 
\end{abstract}

\section{Introduction}\label{sec:intro}

We study the {\em decentralized optimization} problem
\begin{equation}\label{eq:main}
\squeeze
	\min \limits_{x\in\R^d} \sum \limits_{i=1}^n f_i(x),
\end{equation}
where each function $f_i\colon \R^d \rightarrow \R$ is stored on a compute node $i\in [n]\eqdef \{1,2,\ldots,n\}$. We assume that the nodes are connected through a {\em communication network} defined by an undirected connected graph. Each node can perform computations based on its local state and data, and can directly communicate  with its neighbors only. Further, we assume  the functions $f_i$ to be smooth and strongly convex. Such decentralized optimization  problems have been studied heavily  \cite{Gorbunov-Decentralized-Survey2020}, and  arise in many applications, incluyding  estimation by sensor networks \citep{rabbat2004distributed}, network resource allocation \citep{beck20141}, cooperative control \citep{giselsson2013accelerated}, distributed spectrum sensing \citep{bazerque2009distributed}, power system control \citep{gan2012optimal} and federated learning \citep{FL_survey_2020, D-DIANA}. 
When the network is not allowed to change in time, a lower communication complexity bound has been established by \citet{scaman2017optimal}. This bound is tight as there is a matching upper bound both in the case when a {\em dual oracle} is assumed \citep{scaman2017optimal}, which means that we have access to the gradient of the Fenchel conjugate of the functions $f_i(x)$, and also in the case when a {\em primal oracle} is assumed \citep{kovalev2020optimal}, which means that we have access to the gradient of the functions $f_i(x)$ themselves.

\subsection{Time-varying networks}

\begin{figure}[t]
	\centering
	\begin{tikzpicture}
		\node[] at (-0.32\linewidth, 0) (G1) {\includegraphics[width=0.3\linewidth]{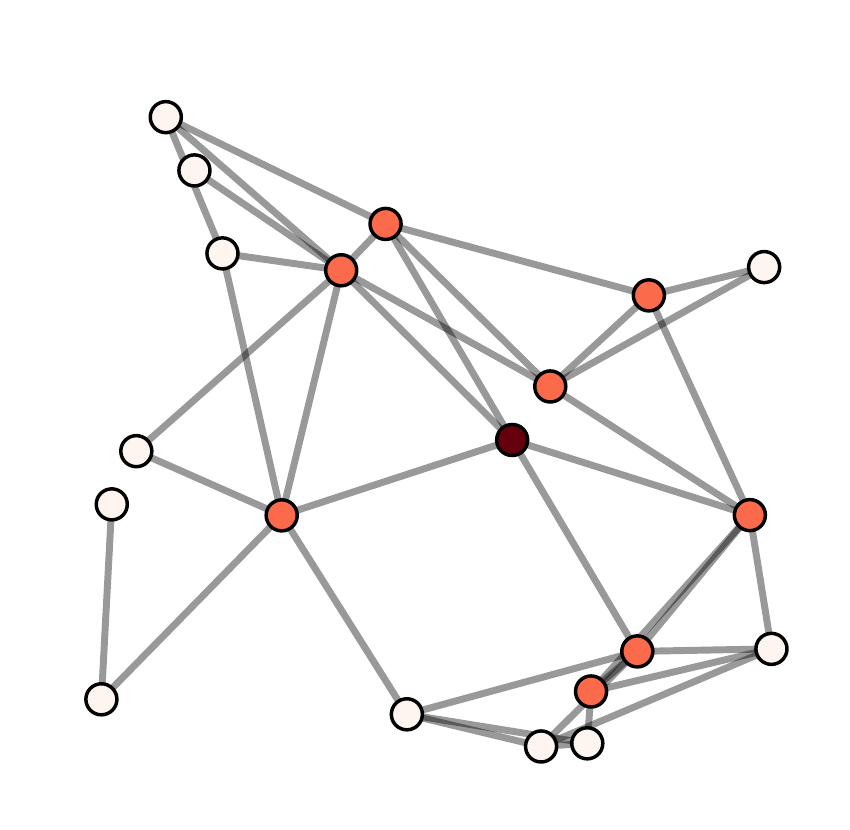}};
		\node[] at (0, 0) (G2) {\includegraphics[width=0.3\linewidth]{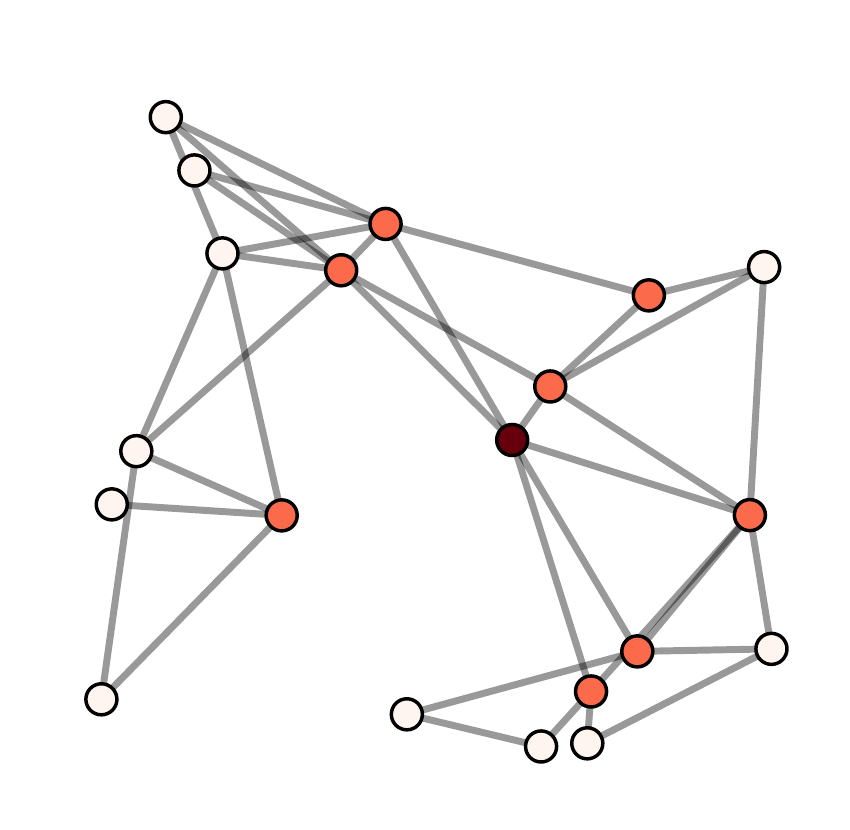}};
		\node[] at (0.32\linewidth, 0)(G3) {\includegraphics[width=0.3\linewidth]{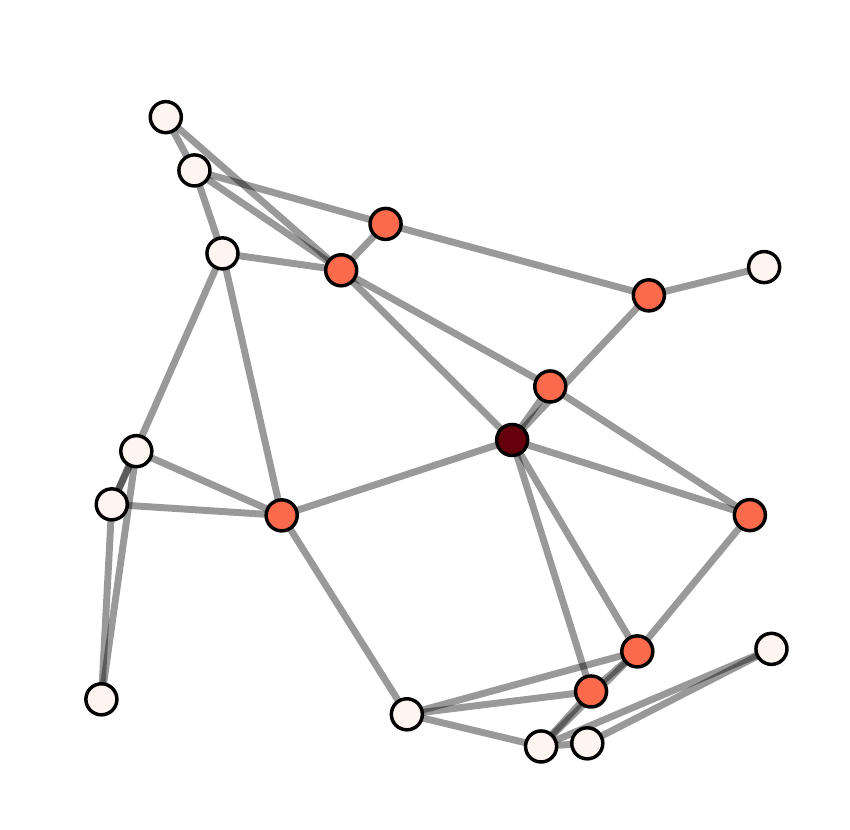}};
		\node[] at (-0.16\linewidth,0) {$\Rightarrow$};
		\node[] at (0.16\linewidth,0) {$\Rightarrow$};
	\end{tikzpicture}
	\caption{A sample time-varying network with $n=20$ nodes.}
	\label{fig:time_varying}
	\vspace{-1em}
\end{figure}

In this work, we study the situation when the links in the communication network are allowed to change over time (for an illustration, see Figure~\ref{fig:time_varying}). Such {\em time-varying networks} \cite{Zadeh1961, Kolar2010} are ubiquitous in many complex systems and practical applications. In sensor networks,  for example, changes in the link structure occur when the sensors are in motion, and due to other disturbances in the wireless signal connecting pairs of nodes.  We envisage that a similar regime will be supported in future-generation federated learning systems \cite{FEDLEARN, FL2017-AISTATS, D-DIANA}, where the communication pattern among pairs of mobile devices or mobile devices and edge servers will be dictated by their physical proximity, which naturally changes over time. Our work can be partially understood as an attempt to contribute to the algorithmic foundations of this nascent field.

\subsection{On the smooth and strongly convex regime}

As mentioned earlier, throughout this paper we  restrict each function $f_i(x)$ to be $L$-smooth and $\mu$-strongly convex. That is, we require that the inequalities
$	f_i(x) \geq f_i(y) + \<\g f_i(y), x-y> + \frac{\mu}{2}\sqn{x- y}$ and
$f_i(x)  \leq f_i(y) + \<\g f_i(y), x-y> + \frac{L}{2}\sqn{x- y}$
hold for all nodes $i\in [n]$ and all $x,y \in \R^d$. This naturally leads to the quantity \begin{equation}\label{eq:kappa}\kappa \eqdef L/\mu\end{equation} known as the {\em condition number} of   function $f$.  As we shall see, current understanding of decentralized optimization over time-varying networks is insufficient even in this setting, and we believe that the key  technical issues we face at present do not come from the difficulty of the function class, but from the algorithmic and modeling aspect of dealing with the decentralized and time-varying nature of the problem. Thus, focusing on smooth and strongly convex problems should not be seen as a weakness, but as a necessary step in the quest to  make a significant advance in our understanding of how efficient decentralized methods should be designed in the time-varying network regime.

\subsection{Methods for time-varying networks}

To the best of our knowledge, there is only a handful of  algorithms for solving the decentralized optimization problem \eqref{eq:main} that enjoy a {\em linear convergence rate} in the time-varying regime under smoothness and strong convexity assumptions. These include DIGing \citep{nedic2017achieving} and Push-Pull Gradient Method \citep{pu2020push}, which use the primal oracle, and PANDA \citep{maros2018panda}, which uses the dual oracle. While linear, their  rates are  slow in comparison to the best methods ``on the market'' at present (see Table~\ref{tbl:contributions-rates}).

A well known  mechanism for improving the convergence rates of standard gradient type methods is to apply or adapt Nesterov {\em acceleration} \citep{nesterov2003introductory}, whose goal is to reduce the dependence of the method on the  condition number $\kappa$ associated with the problem, the  condition number $\chi$ associated with the network structure (see \eqref{eq:chi} in Section~\ref{sec:chi}), or both. However, doing this is nontrivial in the decentralized time-varying setting. 
    

\section{Summary of Contributions}

We now briefly outline the main contributions: 

\subsection{New algorithm} In this paper we propose an accelerated algorithm---{\sf ADOM} (Algorithm~\ref{alg:dual})---for smooth and strongly convex decentralized optimization over time-varying networks. This algorithm uses the dual oracle, 
and is based on a careful generalization of the Projected Nesterov Gradient Descent method (PNGD; Algorithm~\ref{alg:nesterov}). 

\subsection{Convergence analysis} We prove that {\sf ADOM} enjoys the rate $\cO(\chi \kappa^{1/2} \log \frac{1}{\varepsilon})$ (see Thm~\ref{thm:main}), which matches the $\cO(\kappa^{1/2} \log \frac{1}{\varepsilon})$   rate of  PNGD in the special case of a fully connected time-invariant network.

\subsection{Innovations in the analysis} Our analysis requires several new insights and tools. First, we rely on the {\em new observation} that {\em  decentralized communication can be seen as the application of a certain  contractive compression operator} (see Section~\ref{sec:compression}). This  operator is linear, but may be {\em biased}, which raises significant challenges.  While the use of {\em unbiased} compression operators, such as sparsification  and quantization, is increasingly popular in  modern literature  on distributed optimization in the 
parameter server framework\footnote{Distributed optimization in a 
parameter server framework  is mathematically  equivalent to the setting where  communication happens over a  fully connected time-invariant network.}, we only know of a handful of results  combining compression with acceleration \citep{ADIANA, EC-Katyusha}. Of these, the first handles {\em unbiased} compressors only, and the second  is  the only work we know of successfully combining biased communication compression and acceleration. However, their work makes use of a different acceleration mechanism from ours, and it is not clear how to extend it to decentralized optimization.  We are not aware of any results combining  the use of biased compressors, acceleration and decentralized communication, even if we allow for the networks to be time-invariant. The observation that decentralized communication can be modeled as the application of a certain contractive compressor allows us to design a bespoke {\em error-feedback} mechanism,  previously studied in other settings by \citet{stich2019error, karimireddy2019error, biased2020, gorbunov2020linearly},  for achieving acceleration despite dealing with a biased  compressor.

\subsection{Comparison to accelerated methods designed for time-varying networks} 

While there were attempts to design  accelerated algorithms that could deal with time-varying networks, only several methods provide sub-quadratic dependence on $\chi$: 
    Acc-DNGD \citep{qu2019accelerated}, Mudag \citep{ye2020multi},  and the Accelerated Penalty Method (APM) \citep{rogozin2020towards, li2018sharp}. Acc-DNGD has $\cO(\chi^{3/2})$ dependence on $\chi$, which is worse than the linear dependence on $\chi$ shared by Mudag, APM and our method {\sf ADOM}. Moreover, Acc-DNGD has $\cO(\kappa^{5/7})$ dependence on $\kappa$ and Mudag has $\cO(\kappa^{1/2} \log \kappa)$ dependence, which is worse than the $\cO(\kappa^{1/2})$   dependence of APM and our method {\sf ADOM}. Lastly, APM has a square-logarithmic dependence on $1/\varepsilon$, which is worse than the dependence of all the other methods  on this quantity. These results are summarized in Table~\ref{tbl:contributions-rates}. In summary, {\em {\sf ADOM} achieves the new state-of-the-art rate for decentralized optimization over time-varying networks.}

\begin{table}[t]
	\centering
	\footnotesize
	\caption{
		A review of decentralized optimization algorithms capable of working in the time-varying network regime, with guarantees. 		Complexity terms  \alertR{highlighted in red} represent the best known dependencies.
		Our method is the only method with best known dependencies in all terms.
}
	\vspace{1em}
	\begin{tabular}{|c|c|}
		\hline
		\bf Algorithm & \bf Communication complexity\\
		\hline
		\hline
		\makecell{DIGing \\ \citet{nedic2017achieving}}& $\cO\left(n^{1/2}\chi^{2} \kappa^{3/2} \alertR{\log \frac{1}{\epsilon}}  \right)$\\
		\hline
		\makecell{PANDA\\ \citet{maros2018panda}}& $\cO\left(\chi^{2} \kappa^{3/2} \alertR{\log \frac{1}{\epsilon}}  \right)$\\
		\hline
		\makecell{Acc-DNGD\\ \citet{qu2019accelerated}}& $\cO\left(\chi^{3/2} \kappa^{5/7} \alertR{\log \frac{1}{\epsilon}} \right)$\\
		\hline
		\makecell{APM\\ \citet{li2018sharp}}& $\cO\left( \alertR{\chi \kappa^{1/2}} \log^2 \frac{1}{\epsilon}\right)$\\
		\hline
		\makecell{Mudag\\ \citet{ye2020multi}}& $\cO\left( \alertR{\chi \kappa^{1/2}}\log(\kappa) \alertR{\log \frac{1}{\epsilon} }\right)$\\
		\hline
		\hline
		\rowcolor{bgcolor}
		\makecell{\bf ADOM (Algorithm~\ref{alg:dual}) \\ \bf THIS PAPER}& $\cO\left(\alertR{\chi\kappa^{1/2} \log \frac{1}{\epsilon}}\right)$\\
		\hline
	\end{tabular}
	\label{tbl:contributions-rates}
	\vspace{-1em}
\end{table}

\subsection{Comparison to DNM of \citet{rogozin2019optimal}} 

We left one relevant method our from the above comparison -- the Distributed Nesterov Method (DNM) of \citet{rogozin2019optimal}. This method has $\cO(\chi^{1/2})$ dependence on $\chi$. However, DNM converges under the very restrictive assumption  requiring the {\em number of network changes to not exceed a tiny percentage of the number of iterations}. This assumption is hard to satisfy in practice, as the changes in the network topology usually can not be controlled and happen independently of the algorithm run.  In contrast, our algorithm just requires the network to be connected all the time. In Figure~\ref{fig:slow2} we give a representative comparison of the workings of our method ADOM and DNM in a regime where the number of network changes exceeds the theoretical limit. While ADOM converges,  DNM often diverges, which shows that DNM is not robust to the network dynamics, and that the restrictive assumption is crucial to their analysis.

\section{Problem Formulation and Projected Nesterov Gradient Descent}\label{sec:problem}

The design of our method is based on a particular reformulation of problem \eqref{eq:main}, which we now describe.

\subsection{Reformulation via Lifting}
Consider function $F\colon (\R^d)^\cV \rightarrow \R$ defined by
\begin{equation}\label{eq:F}
\squeeze F(x) = \sum \limits_{i\in\cV} f_i(x_i),
\end{equation}
where $x = (x_1,\ldots,x_n) \in (\R^d)^\cV$ and $\cV \eqdef [n]$ denotes the set of compute nodes. Then $F$ is $L$-smooth $\mu$-strongly convex since the individual functions $f_i$ are. Consider also the so called {\em consensus space} $\cL \subset (\R^d)^\cV$ defined by
\begin{equation}\label{eq:Lspace}
\cL \eqdef \{(x_1,\ldots,x_n )\in(\R^{d})^\cV: x_1 = \cdots = x_n\}.
\end{equation}
Using this notation, we arrive at an equivalent formulation of  problem \eqref{eq:main}, which we call the {\em primal formulation}:
\begin{equation}\label{eq:primal}
\min_{x \in \cL} F(x).
\end{equation}
Since the function $F(x)$ is strongly convex, this problem has a unique solution, which we denote as $x^* \in \cL$.

\subsection{Dual Problem}
It is a well known fact that problem \eqref{eq:primal} has an equivalent {\em dual formulation} of the form
\begin{equation}\label{eq:dual}
	\min_{z \in \cL^\perp} F^*(z),
\end{equation}
where $F^*$ is the Fenchel transform of $F$ and $\cL^\perp \subset (\R^d)^\cV$ is the orthogonal complement to the space $\cL$, given as follows:
\begin{equation}
\squeeze \cL^\perp = \left\{(z_1,\ldots,z_n )\in(\R^{d})^\cV: \sum_{i=1}^n z_i= 0\right\}.
\end{equation}
Note that the function $F^*(z)$ is $\frac{1}{\mu}$-smooth and $\frac{1}{L}$-strongly convex \citep{rockafellar1970convex}. Hence, problem \eqref{eq:dual} also has a unique solution, which we denote as $z^* \in \cL^\perp$.

\subsection{Projected Nesterov Gradient Descent}
A natural way to tackle problem \eqref{eq:dual} is to use a projected version of Nesterov's accelerated gradient method: Projected Nesterov Gradient Descent (PNGD) \citep{nesterov2003introductory}.
This algorithm requires us to calculate projection onto the set $\cL^\perp$, which can be written in the closed form
\begin{equation}\label{eq:proj}
	\proj_{\cL}(g) \eqdef \argmin_{z \in \cL^\perp} \sqn{g - z} = \mP g,
\end{equation}
where $g \in (\R^d)^\cV$ and $\mP$ is an orthogonal projection matrix onto the subspace $\cL^\perp$. Matrix $\mP$ is given as follows:
\begin{equation}\label{eq:P}
\squeeze \mP = \left(\mI_n - \frac{1}{n}\ones_ n\ones_n^\top\right)\otimes \mI_d,
\end{equation}
where $\mI_p$ denotes $p\times p$ identity matrix, $\ones_n = (1,\ldots,1)\in \R^n$, $\otimes$ is a Kronecker product. Note that
\begin{equation}\label{eq:PP}
\mP^2 = \mP.
\end{equation}
With this notation, PNGD is presented as Algorithm~\ref{alg:nesterov}.

\begin{algorithm}
	\caption{PNGD: Projected Nesterov Gradient Descent}
	\label{alg:nesterov}
	\begin{algorithmic}[1]
		\State {\bf input:} $z^0 \in \cL^\perp$, $\alpha,\eta,\theta > 0$, $\tau \in (0,1)$
		\State set $z_f^0 = z^0$
		\For{$k=0,1,2\ldots$}
		\State $z_g^k = \tau z^k + (1-\tau)z_f^k$ \label{nesterov:line:z1}
		\State $z^{k+1} = z^k + \eta\alpha (z_g^k - z^k)- \eta \mP\g F^*(z_g^k)$\label{nesterov:line:z2}
		\State $z_f^{k+1} = z_g^k - \theta \mP \g F^*(z_g^k)$\label{nesterov:line:z3}
		\EndFor
	\end{algorithmic}
\end{algorithm}

A key property of Algorithm~\ref{alg:nesterov} is that it converges with the accelerated rate $\cO \left(\sqrt{\nicefrac{L}{\mu}}\log \frac{1}{\epsilon}\right)$. However, PNGD in each iteration calculates the matrix-vector multiplication $\mP \g F^*(z_g^k)$, which requires {\em full averaging}, i.e., consensus, over all nodes of the communication network. In particular, this {\em can not be done in decentralized fashion}. In Section~\ref{sec:algorithm} we describe our algorithm {\sf ADOM}, which in a certain sense mimics the behavior of Algorithm~\ref{alg:nesterov}, but {\em can be} implemented in a decentralized fashion.

\section{Decentralized Communication}\label{sec:communication}

We now introduce the necessary notation, definitions and formalism to be able to describe our method.
Compute nodes $\cV=[n]$ are connected through a communication network represented as a graph $\cG^k = (\cV, \cE^k)$, where $k \in \{0,1,2,\ldots\}$ encodes time, and $\cE^k \subseteq \{(i,j) \in \cV\times\cV : i\neq j\}$ is the set of edges at time $k$. In this work we assume that the graph $\cG^k$ is undirected, that is, $(i,j) \in \cE^k$ implies $(j,i) \in \cE^k$. We also assume that $\cG^k$ is connected. For each node $i \in \cV$ we consider a set of its neighbors at time step $k$: $\cN_i^k = \{j \in \cV : (i, j) \in \cE^k\}$. At time step $k$, each node $i \in \cV$ can communicate  with the nodes from set $\cN_i^k$ only. This type of communication is known as {\em decentralized communication} in the literature.

\subsection{Gossip Matrices}
Decentralized  communication between nodes is typically represented via a matrix-vector multiplication with a {\em gossip matrix}. For time-invariant networks such representations can be found in, e.g., \citep{kovalev2020optimal}. For each time step $k \in \{0,1,2,\ldots\}$ consider a matrix $\hat{\mW}(k) \in \R^{n \times n}$ with the following properties:
\begin{enumerate}
	\item $\hat{\mW}(k)$ is symmetric and positive semi-definite,
	\item $\hat{\mW}(k)_{i,j} \neq 0$ if and only if $(i,j) \in \cE^k$ or $i=j$,
	\item $\ker \hat{\mW}(k) = \{(x_1,\ldots,x_n) \in \R^n : x_1 = \ldots = x_n\}$.
\end{enumerate}
Matrix $\hat{\mW}(k)$ is often called a {\em gossip matrix}. A typical example is the Laplacian of the graph $\cG^k$. Consider also a linear map $\mW(k)\colon (\R^d)^\cV \rightarrow (\R^d)^\cV$, i.e., $nd\times nd$ matrix defined by $\mW(k) = \hat{\mW}(k) \otimes \mI_d$. This matrix can be represented as a block matrix $(\mW(k)_{i,j})_{(i,j)\in \cV^2}$, where each block $\mW(k)_{i,j} = \hat{\mW}(k)_{i,j}\mI_d$ is  a $d\times d$ matrix  proportional to $\mI_d$. Matrix $\mW(k)$ satisfies similar properties to $\hat{\mW}(k)$:
\begin{enumerate}
	\item $\mW(k)$ is symmetric and positive semi-definite,
	\item $\mW(k)_{i,j} \neq 0$ if and only if $(i,j) \in \cE^k$ or $i=j$,
	\item $\ker \mW(k) = \cL$ or equivalently $\range \mW(k) = \cL^\perp$.
\end{enumerate}
With a slight abuse of language, in the rest of this paper we will refer to  $\mW(k)$ as a {\em gossip matrix} as well.

\subsection{Decentralized Communication as Multiplication with the Gossip Matrix}
Decentralized communication of  vectors $x_1, \dots, x_n\in \R^d$ stored on the nodes  among neighboring nodes at time step $k$ can be represented as a multiplication of the  $nd$-dimensional vector by matrix $\mW(k)$. Indeed, consider $x = (x_1,\ldots,x_n) \in (\R^d)^\cV, y = (y_1,\ldots, y_n) \in (\R^d)^\cV$, where each $x_i$ is stored by node $i\in \cV$, and let $y = \mW(k) x$. One can observe that
\begin{equation*}
\squeeze 	y_i = \sum \limits_{j=1}^n \hat{\mW}(k)_{i,j} x_j = \sum \limits_{j \in \cN_i} \hat{\mW}(k)_{i,j} x_j.
\end{equation*}
Hence, for each node $i$, vector $y_i$ is a linear combination of vectors $x_j$, stored at the neighboring nodes $j \in \cN_i$. This means that matrix-vector multiplications by matrix $\mW(k)$ can be computed in a decentralized fashion. 

\subsection{Condition Number of Time-Varying Networks}\label{sec:chi}
A condition number of the matrix $\hat{\mW}(k)$ is given as $\frac{\lmax(\hat{\mW}(k))}{\lminp(\hat{\mW}(k))}$, where $\lambda_{\max}$ refers to the largest and $\lambda_{\min}^+$ to the smallest positive eigenvalue. This quantity is known to be a measure of the  connectivity of graph $\cG^k$, and appears in convergence rates of many decentralized algorithms. In this work we assume that this condition number is bounded for all $k \in \{0,1,2\ldots\}$. In particular, we assume that there exist constants $0 < \lminp < \lmax$ such that
\begin{equation}\label{eq:09h0df9hjfd}
	\lminp \leq \lminp(\hat{\mW}(k)) \leq \lmax(\hat{\mW}(k)) \leq \lmax.
\end{equation}
So, we assume that the worst case spectral behavior of the gossip matrices is bounded, and these bounds will later appear in our convergence rate for {\sf ADOM}.

Relation \eqref{eq:09h0df9hjfd} can be equivalently written in the form of a linear matrix inequality involving the gossip matrix $\mW(k)$:
\begin{equation}\label{eq:spectrum}
\lminp \mP \preceq \mW(k) \preceq \lmax \mP,
\end{equation}
where $\mP$ is orthogonal projector onto subspace $\range \mW^k = \cL^\perp$ given by \eqref{eq:P}. Note that
\begin{equation}\label{eq:PW}
\mP\mW(k) = \mW(k)\mP = \mW(k).
\end{equation}
By $\chi$ we  denote a bound on the condition number of matrices $\mW(k), k =0,1,2\ldots$, given by \begin{equation} \label{eq:chi}\chi \eqdef \lmax/\lminp.\end{equation}

\subsection{Decentralized Communication as a Compression Operator}\label{sec:compression}

We have just shown that decentralized communication at time step $k$ can be represented as  multiplication by the gossip matrix $\mW(k)$. We will now show, and this is a key insight which was the starting point of our work, that {\em decentralized communication can also be seen as the application of a contractive compression operator.} 

Let $\cQ$ be a linear space. A mapping $\cC\colon \cQ\rightarrow \cQ$ is called a {\em compression operator} if there exists $\delta \in (0,1]$ such that
\begin{equation}\label{eq:contractor}
	\sqn{\cC(z) - z} \leq (1-\delta)\sqn{z} \text{ for all } z \in \cQ.
\end{equation}
The following lemma shows that matrix-vector multiplication by gossip matrix $\mW(k)$ is a contractive compression operator acting on the subspace $\cL^\perp$.
\begin{lemma}
	Let $\sigma \in (0,1/\lmax)$, $k \in \{0,1,2\ldots \}$. Then the following inequality holds for all $z \in \cL^\perp$:
	\begin{equation*}
		\sqn{\sigma\mW(k) z - z} \leq (1-\sigma\lminp)\sqn{z}.
	\end{equation*}
\end{lemma}

There is a natural question regarding Algorithm~\ref{alg:nesterov}: can we replace the gradient $\mP \g F^*(z_g^k)$ on lines~\ref{nesterov:line:z2} and~\ref{nesterov:line:z3} with its compressed version $\mW(k)\mP \g F^*(z_g^k)$,  or some modification thereof, and still obtain a good convergence result? Note that \eqref{eq:PW} implies $\mW(k)\mP \g F^*(z_g^k) = \mW(k)\g F^*(z_g^k)$. In Section~\ref{sec:algorithm} we will provide a positive answer to this question.

Convergence of gradient-type methods with contractive compression operators satisfying \eqref{eq:contractor} was studied in several recent papers. In particular, \citet{stich2019error, karimireddy2019error, biased2020, gorbunov2020linearly} study a mechanism called {\em error feedback}, which allows to design variations with better convergence properties. However, these works do not study accelerated algorithms, nor provide any connections between compression and decentralized communication, which we do. 

The only exception we are aware of is \citep{EC-Katyusha}, which is the first work proposing an accelerated error compensated method. However, their ECLK method is more complicated than ours, uses the Katyusha momentum \citep{allen2017katyusha, L-SVRG} instead of the Nesterov momentum we employ, and does not apply to decentralized optimization.  Moreover, while for us  it is crucial that the contractive property is enforced on a subspace only, \citet{EC-Katyusha} require this property to hold globally.

\section{ADOM: Algorithm and its Analysis}\label{sec:algorithm}

Armed with the notions and ideas described in preceding sections, we are now ready to  present our  method {\sf ADOM} (Algorithm~\ref{alg:dual}). As alluded to in the introduction, {\sf ADOM} is a generalization of Algorithm~\ref{alg:nesterov}  that can be implemented in a decentralized fashion. Indeed, our algorithm does {\em not} make use of matrix-vector multiplication by $\mP$ in the way Algorithm~\ref{alg:nesterov} does, which requires full averaging over the network. Instead, {\sf ADOM} uses matrix-vector multiplication by the  gossip matrix $\mW(k)$, which represents a single decentralized communication round, as we discussed in Section~\ref{sec:communication}.

\begin{algorithm}[t]
	\caption{ADOM: Accelerated Decentralized Optimization Method}
	\label{alg:dual}
	\begin{algorithmic}[1]
		\State {\bf input:} $z^0 \in \cL^\perp, m^0 \in (\R^d)^\cV, \alpha, \eta, \theta,\sigma\!>\!0, \tau\!\in\!(0,1)$
		\State set $z_f^0 = z^0$
		\For{$k = 0,1,2,\ldots$}
		\State $z_g^k = \tau z^k + (1-\tau)z_f^k$\label{dual:line:z1}
		\State $\Delta^k = \sigma\mW(k)(m^k - \eta\g F^*(z_g^k))$\label{dual:line:delta}
		\State $m^{k+1} = m^k - \eta\g F^*(z_g^k) - \Delta^k$\label{dual:line:m}
		\State $z^{k+1} = z^k + \eta\alpha(z_g^k - z^k) + \Delta^k$\label{dual:line:z2}
		\State $z_f^{k+1} = z_g^k - \theta\mW(k)\g F^*(z_g^k)$\label{dual:line:z3}
		\EndFor
	\end{algorithmic}
\end{algorithm}

\subsection{Design and Analysis of the New Algorithm}
First, we mention two lemmas, which play an important role in the convergence  analysis of Algorithm~\ref{alg:dual}.
\begin{lemma}\label{dual:lem:descent}
	For $\theta \leq \frac{\mu}{\lmax}$ we have the inequality
	\begin{equation}\label{dual:eq:1}
	\squeeze 
	F^*(z_f^{k+1}) \leq F^*(z_g^k) - \frac{\theta\lminp}{2}\sqn{\g F^*(z_g^k)}_{\mP}.
	\end{equation}
\end{lemma}
\begin{lemma}\label{dual:lem:error}
	For $\sigma \leq \frac{1}{\lmax}$ we have the inequality 
	\begin{equation}\label{dual:eq:2}
	\begin{split}
	&\squeeze \sqn{m^k}_{\mP}
	\leq
	\left(1 - \frac{\sigma\lminp}{4}\right)\frac{4}{\sigma\lminp}\sqn{m^k}_{\mP}
	\\&\squeeze -
	\frac{4}{\sigma\lminp}\sqn{m^{k+1}}_\mP
	+
	\frac{8\eta^2}{(\sigma\lminp)^2}\sqn{\g F^*(z_g^k)}_\mP.
	\end{split}
	\end{equation}
\end{lemma}

We now make a few remarks about the main steps of Algorithm~\ref{alg:dual} compared to Algorithm~\ref{alg:nesterov} and comment on some aspects of our convergence analysis and the role of the above lemmas in it:
\begin{itemize}
	\item Line~\ref{dual:line:z1} of Algorithm~\ref{alg:dual} is unchanged compared to line~\ref{nesterov:line:z1} of Algorithm~\ref{alg:nesterov}.
	
	\item Line~\ref{dual:line:z3} of Algorithm~\ref{alg:dual} corresponds to line~\ref{nesterov:line:z3} of Algorithm~\ref{alg:nesterov}. Note that the analysis of Algorithm~\ref{alg:nesterov} requires an inequality of the type
	\begin{equation*}
		F^*(z_f^{k+1}) \leq F^*(z_g^k) - \mathrm{const}\cdot \sqn{\g F^*(z_g^k)}_\mP.
	\end{equation*}
	Lemma~\ref{dual:lem:descent} establishes a similar inequality for Algorithm~\ref{alg:dual}.
	
	\item Together, lines~\ref{dual:line:delta}, \ref{dual:line:m} and~\ref{dual:line:z2} of Algorithm~\ref{alg:dual} form an error feedback update, which we discussed in Section~\ref{sec:compression} when we interpreted a decentralized communication round as the application of a contractive compression operator. A key to the theoretical analysis of this update is to make use of the so-called {\em ghost iterate} $\hat{z}^k = z^k + \mP m^k$. One can observe that the ghost iterate is updated as
	\begin{equation*}
		\hat{z}^{k+1} = \hat{z}^k + \eta\alpha (z_g^k - z^k)- \eta \mP\g F^*(z_g^k),
	\end{equation*}
	which is similar to the update on line~\ref{nesterov:line:z2} of Algorithm~\ref{alg:nesterov}.
	
	\item Another key step in the analysis of Algorithm~\ref{alg:dual} is to bound the distance between the actual iterate $z^k$ and the ghost iterate $\hat{z}^k$, which is equal to $\sqn{m^k}_\mP$. This is done in Lemma~\ref{dual:lem:error} based on  line~\ref{dual:line:m} of Algorithm~\ref{alg:dual}.
\end{itemize}

\subsection{Main Convergence Theorem}

Now, we are ready to present our main theorem.
\begin{theorem}\label{thm:main}
	Set parameters $\alpha, \eta, \theta, \sigma,\tau$ of Algorithm~\ref{alg:dual} to $\alpha = \frac{1}{2L}$, $\eta = \frac{2\lminp\sqrt{\mu L}}{7\lmax}$, $\theta = \frac{\mu}{\lmax}$, $\sigma = \frac{1}{\lmax}$, and\\ $\tau = \frac{\lminp}{7\lmax}\sqrt{\frac{\mu}{L}}$. Then there exists $C>0$, such that
	\begin{equation}
	\squeeze 
		\sqn{\g F^*(z_g^k) - x^*} \leq C \left(1- \frac{\lminp}{7\lmax} \sqrt{\frac{\mu}{L}}\right)^k.
	\end{equation}
\end{theorem}
Note that the rate is $\cO(\chi \kappa^{1/2} \log \frac{1}{\varepsilon})$, as previously advertised.
Proofs of all our results are available in the appendix.

\subsection{Comparison with Existing Algorithms}
In this paper we compare our Algorithm~\ref{alg:dual} with current state-of-the-art algorithms for decentralized optimization over time-varying networks. While the Accelerated Penalty Method (APM) \citep{li2018sharp} and Mudag \citep{ye2020multi} were originally designed for time-invariant networks, they can be easily extended to the time-varying case. Also note that DIGing \citep{nedic2017achieving}, Push-Pull Gradient Method \citep{pu2020push} and PANDA \citep{maros2018panda} converge under slightly more general assumptions than those we used to analyze our method. However, these methods converge at a substantially slower rate due to the fact that they do not employ any acceleration mechanism. Moreover, to the best of our knowledge, no results  improving the  convergence rates of these algorithms under our assumptions exist in the literature.

\section{Numerical Experiments}\label{experiments}

\begin{figure*}[t]
	\centering
	\includegraphics[width=0.24\linewidth]{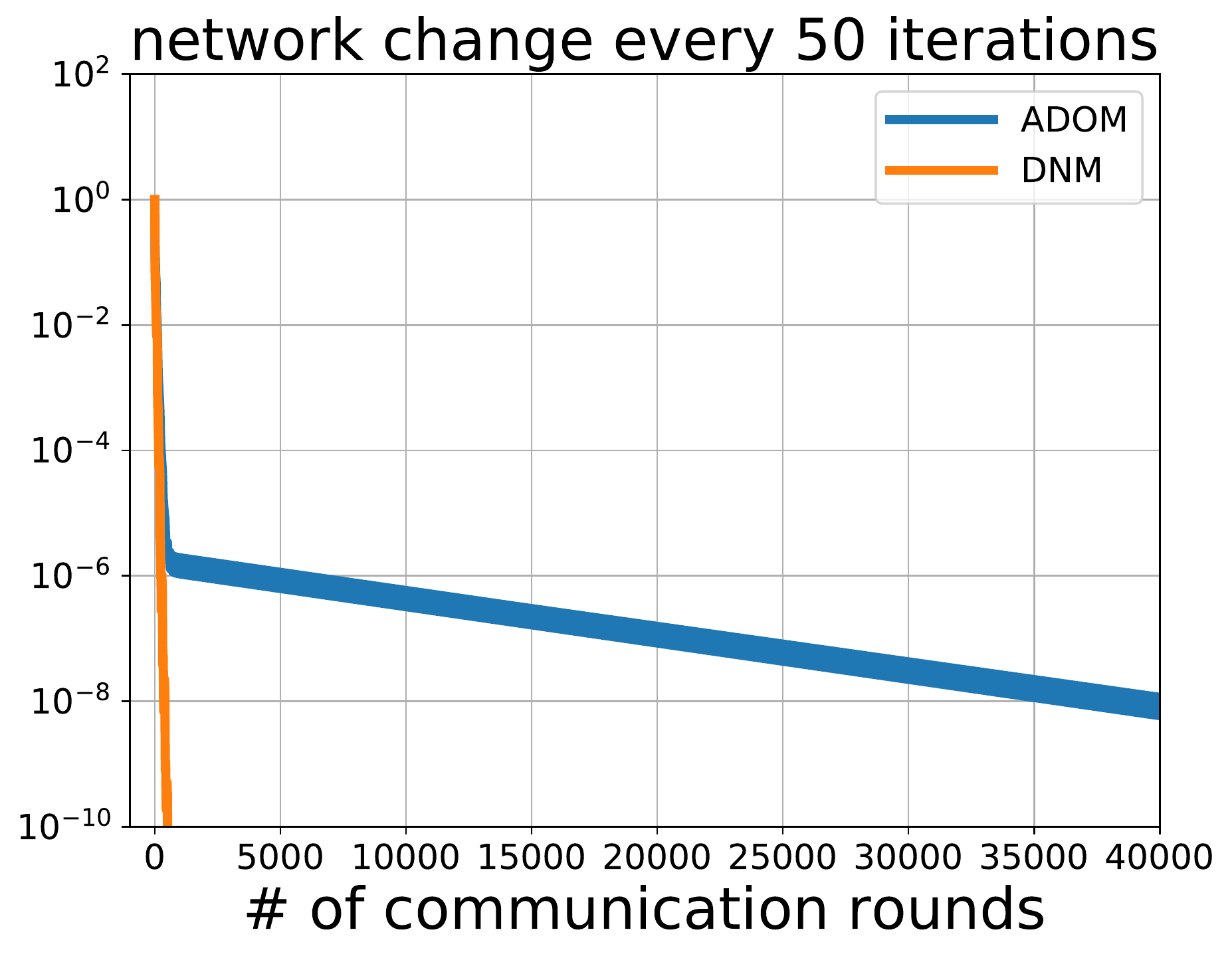}
	\includegraphics[width=0.24\linewidth]{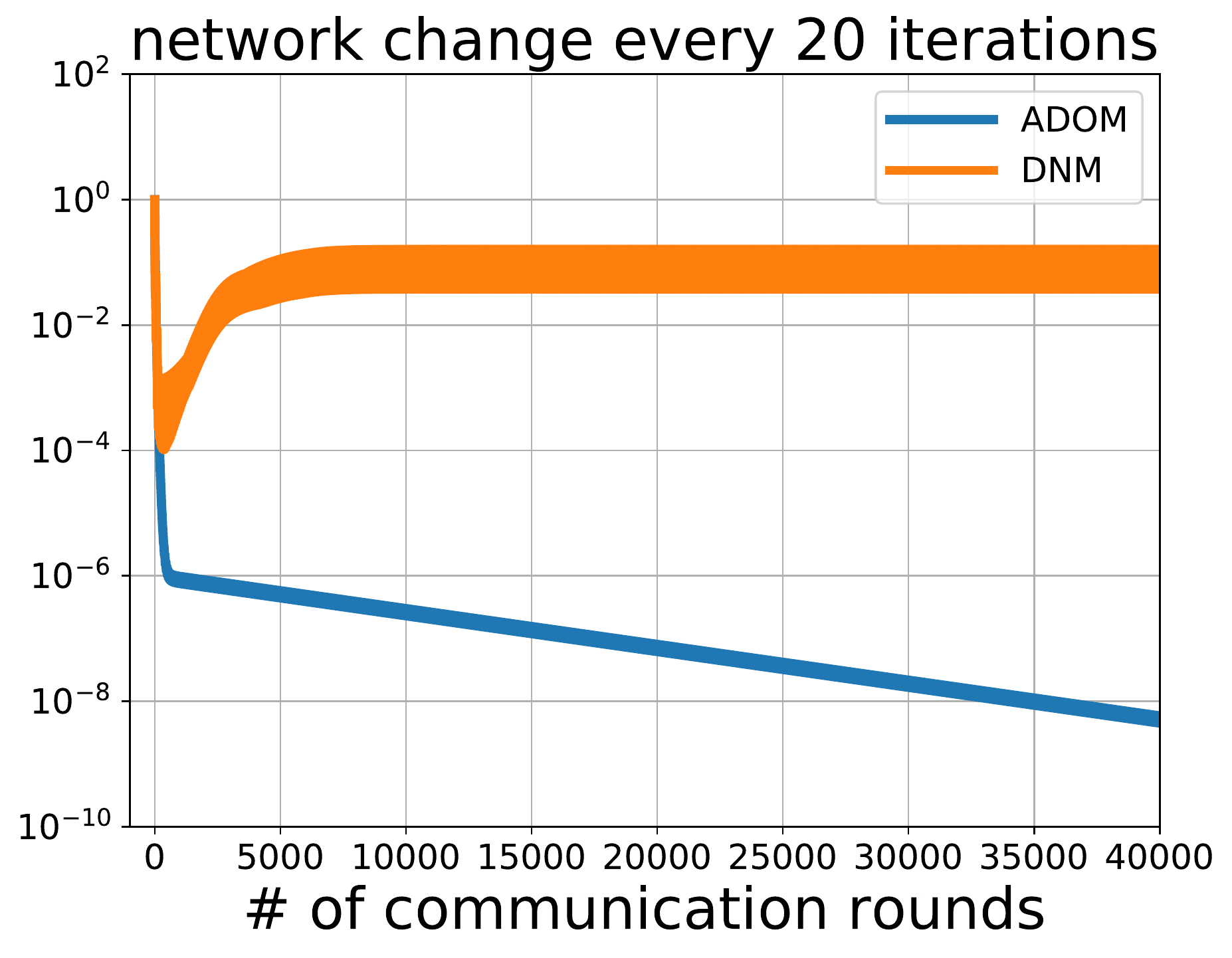}
	\includegraphics[width=0.24\linewidth]{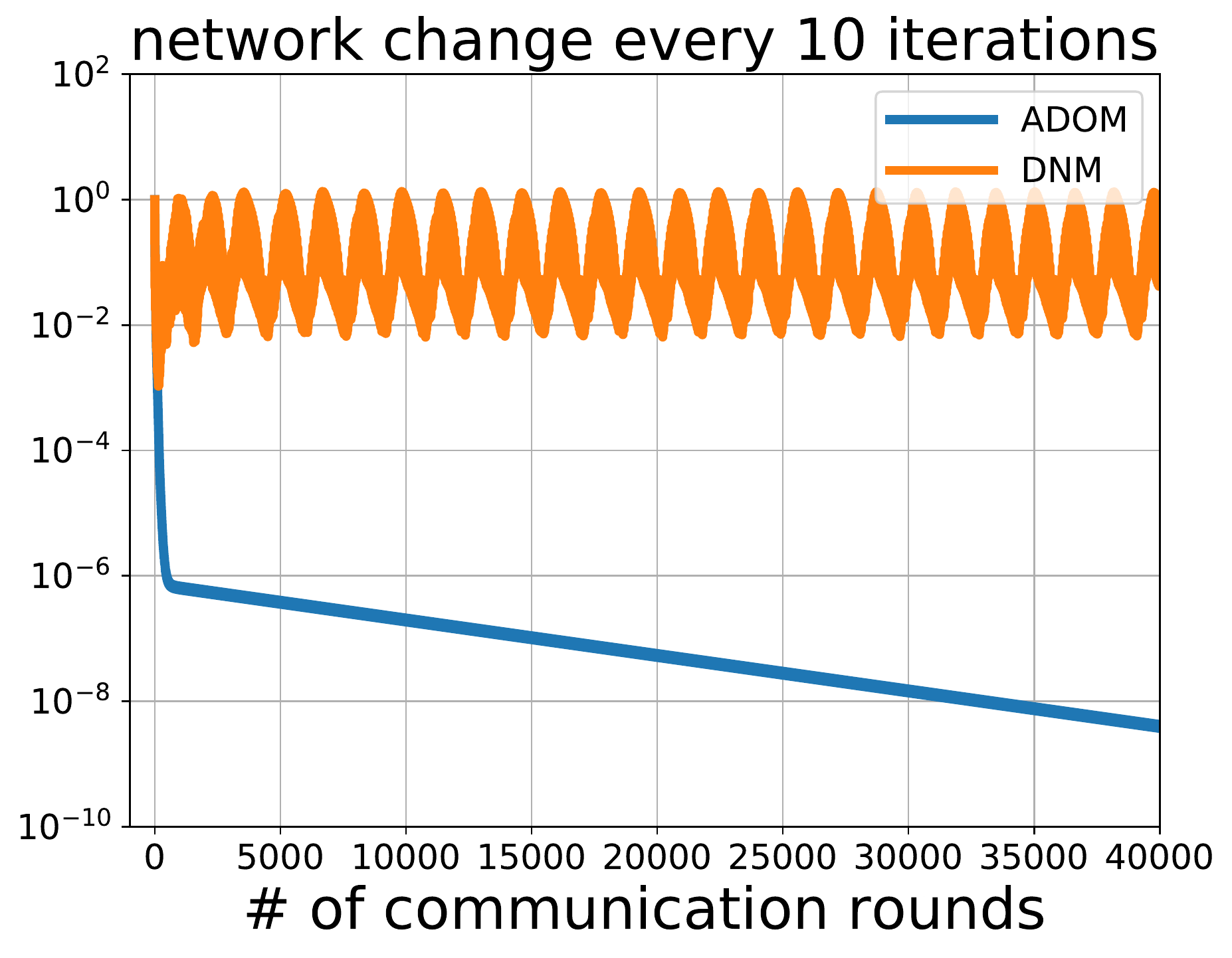}
	\includegraphics[width=0.24\linewidth]{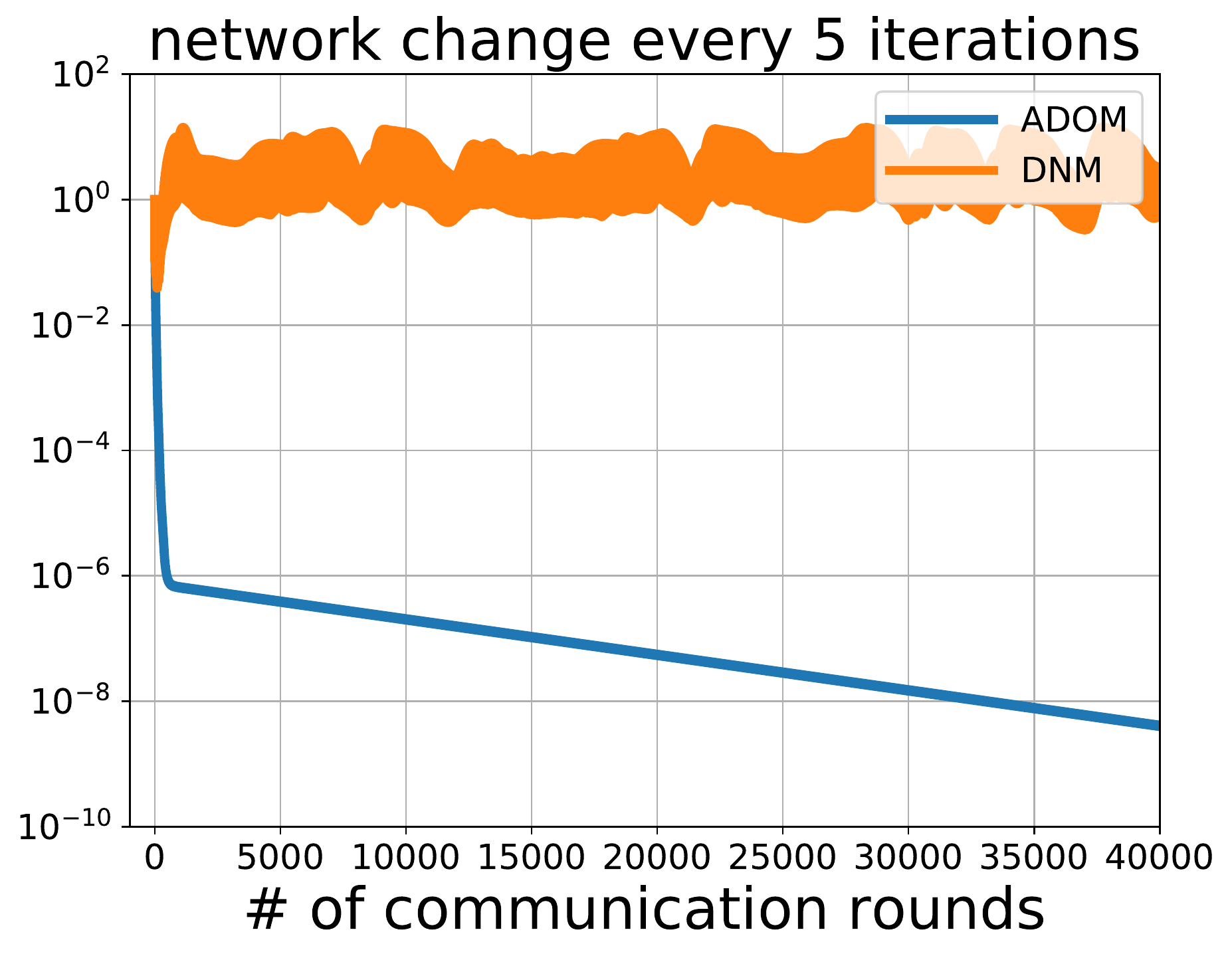}	
\\
	\includegraphics[width=0.24\linewidth]{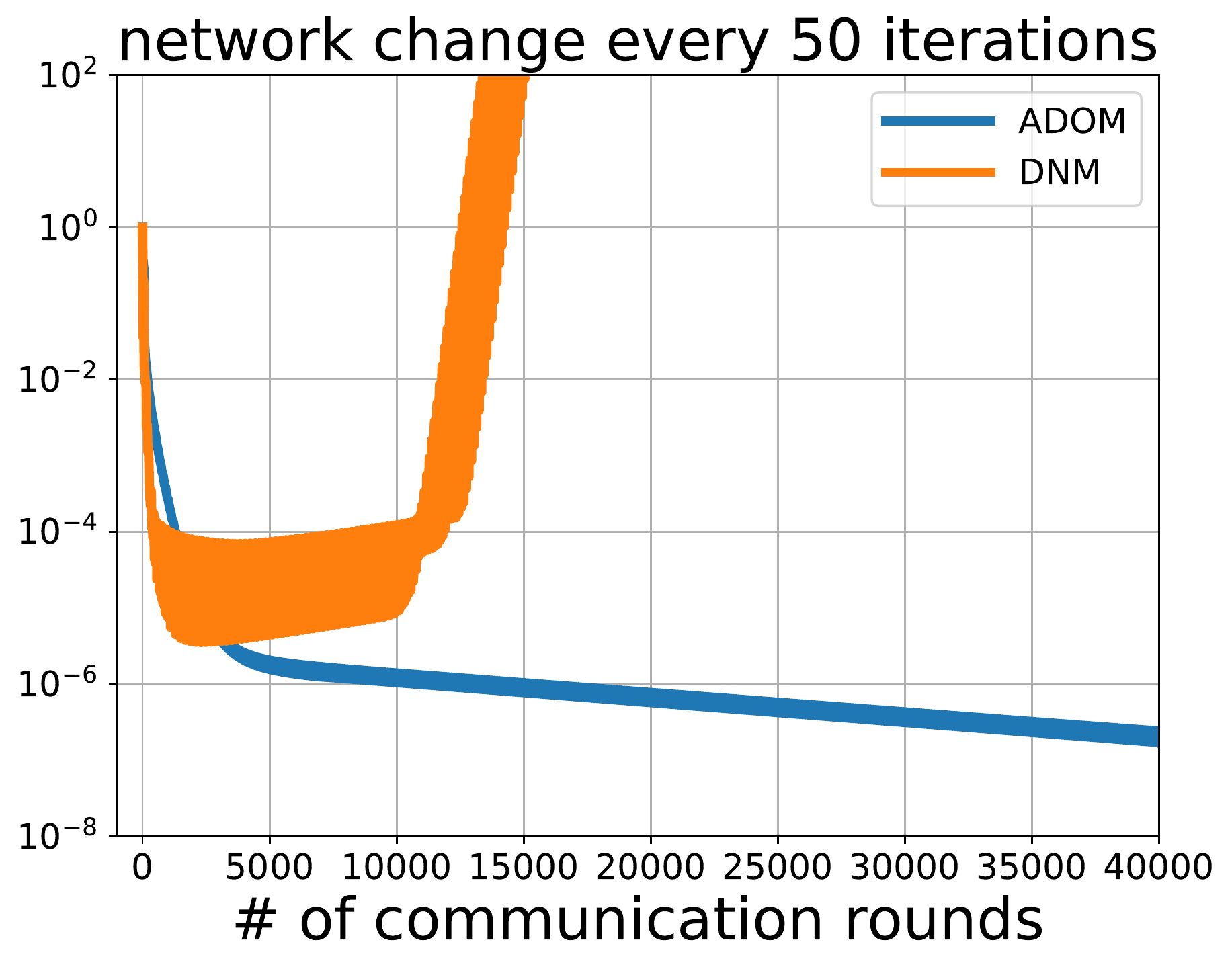}
	\includegraphics[width=0.24\linewidth]{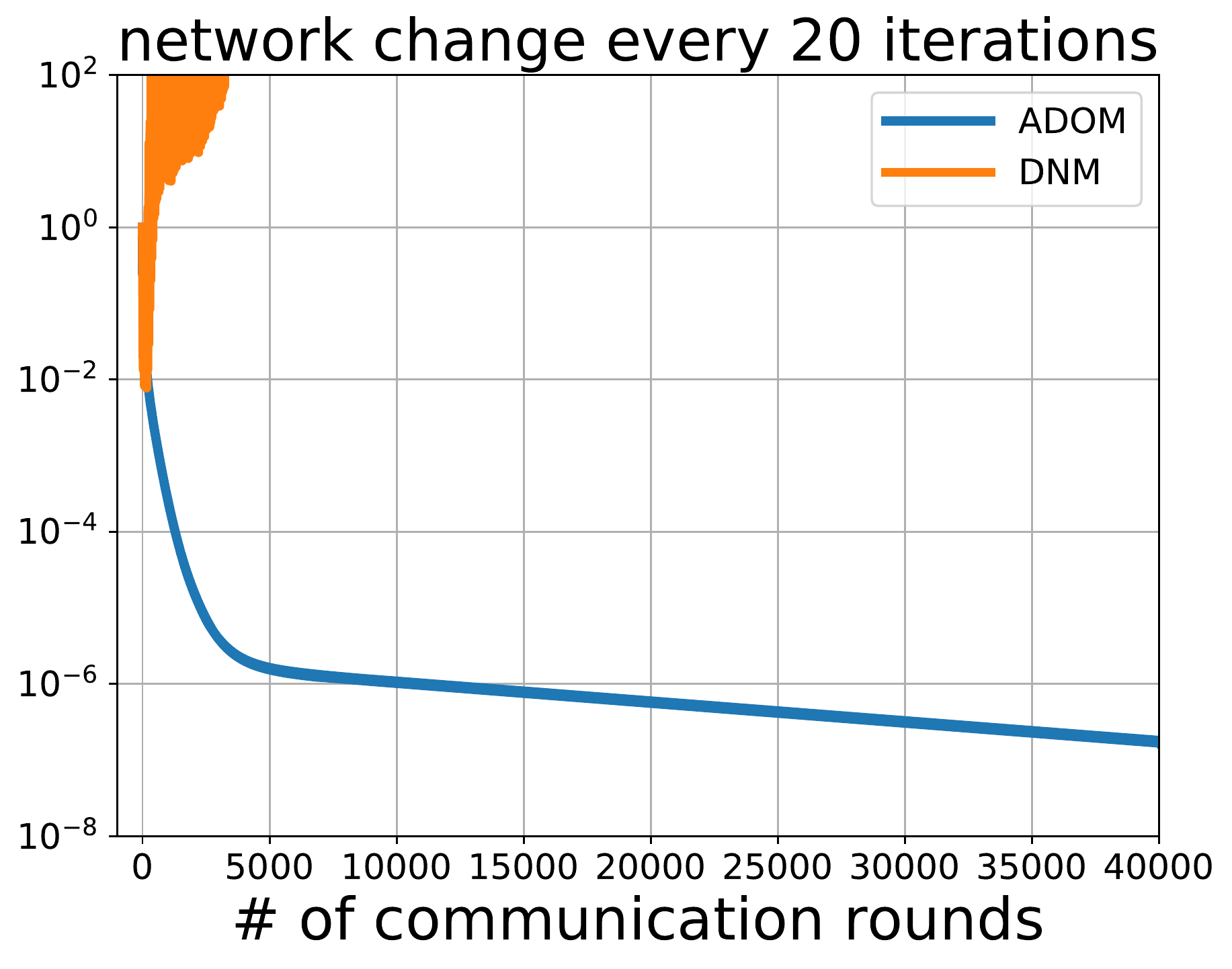}
	\includegraphics[width=0.24\linewidth]{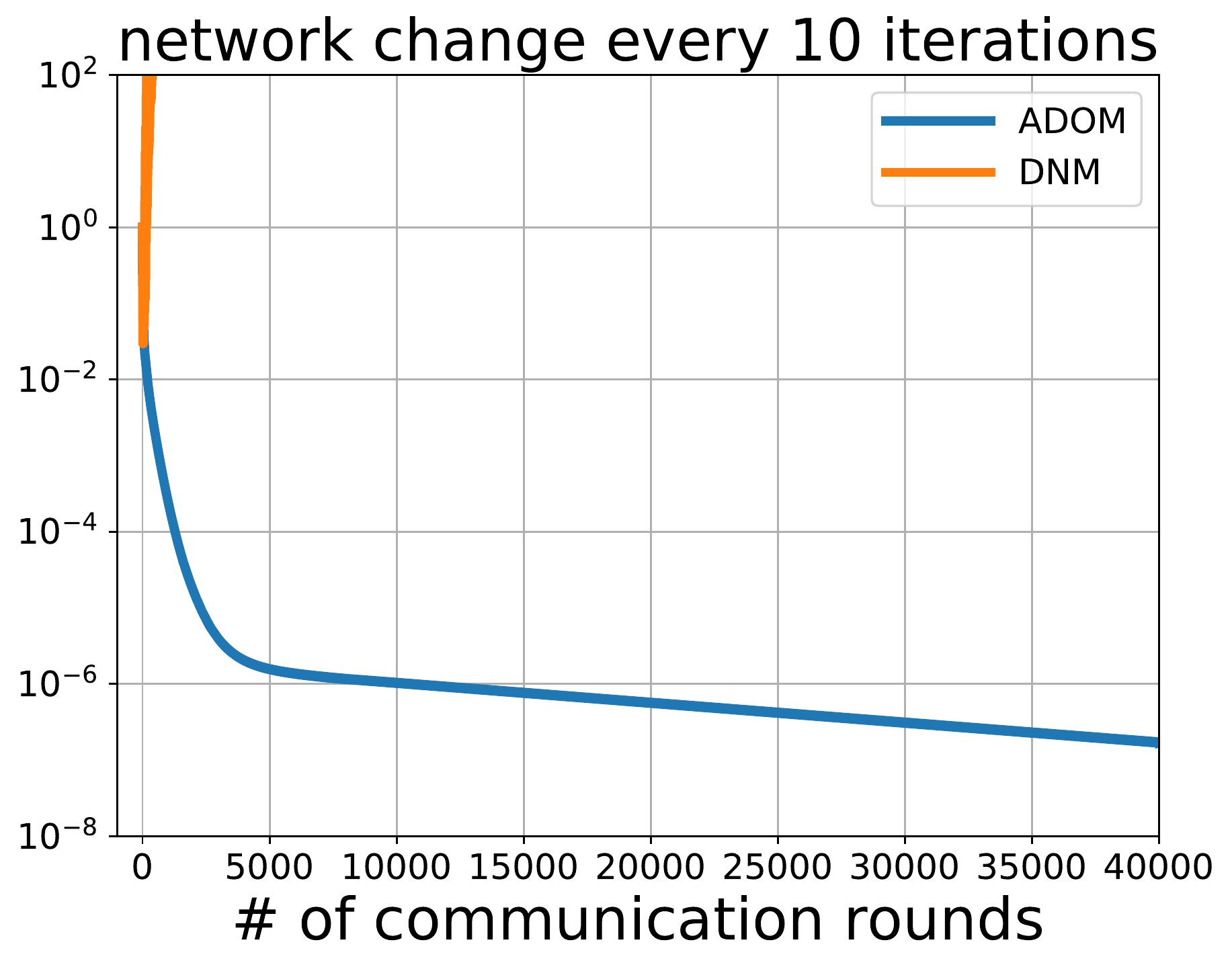}
	\includegraphics[width=0.24\linewidth]{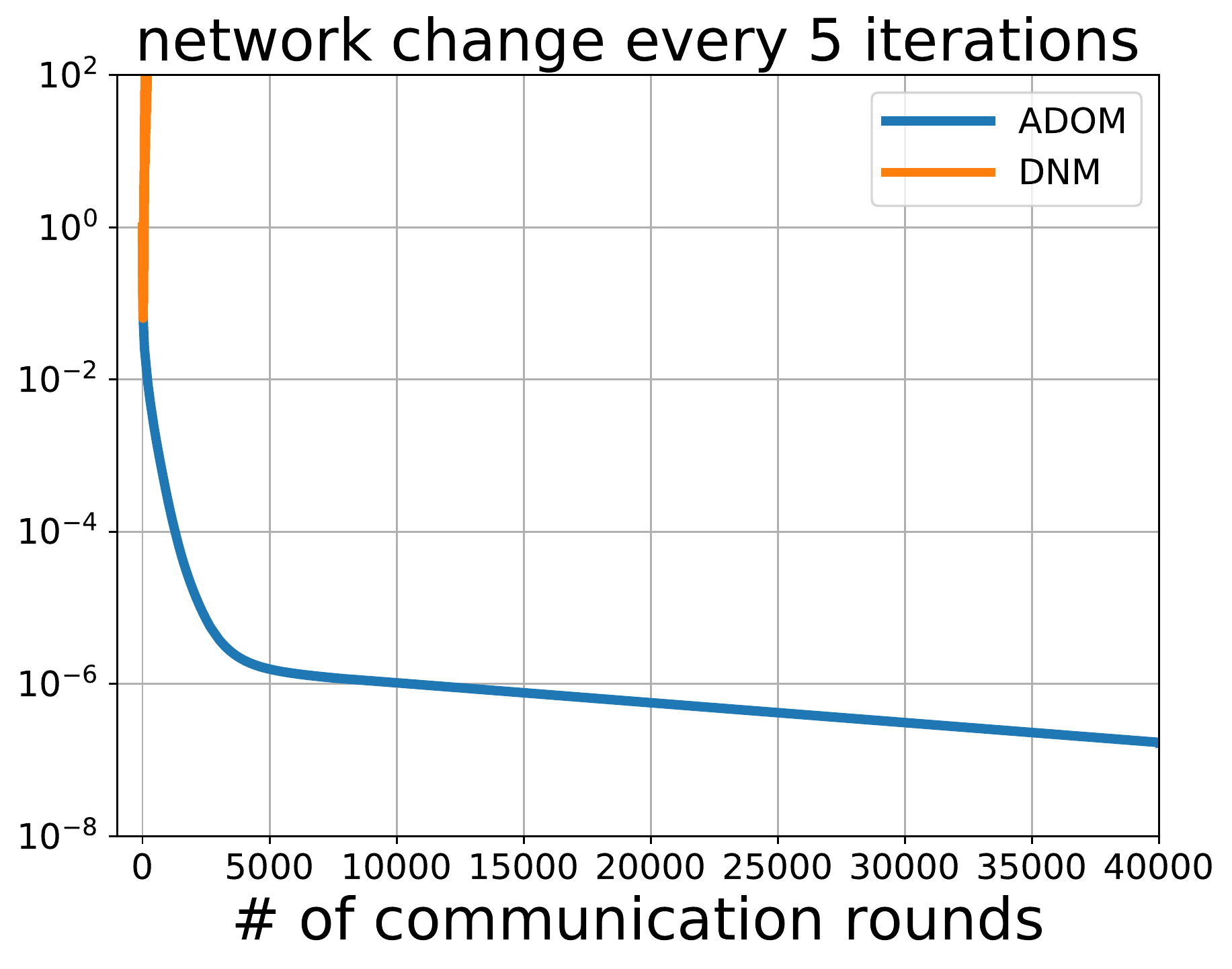}
	
	\caption{Comparison of DNM and {\sf ADOM} on a  problem with $\kappa = 30$ and $d =  40$. {\bf Top row:}  We alternate between two geometric  graphs ($\chi \approx 400$). {\bf Bottom row:} We alternate between two networks, one with a ring and the other with a star topology ($\chi \approx 1000$).}
	\label{fig:slow2}
\end{figure*}

\begin{figure*}[!htb]
	\centering
	\includegraphics[width=0.24\linewidth]{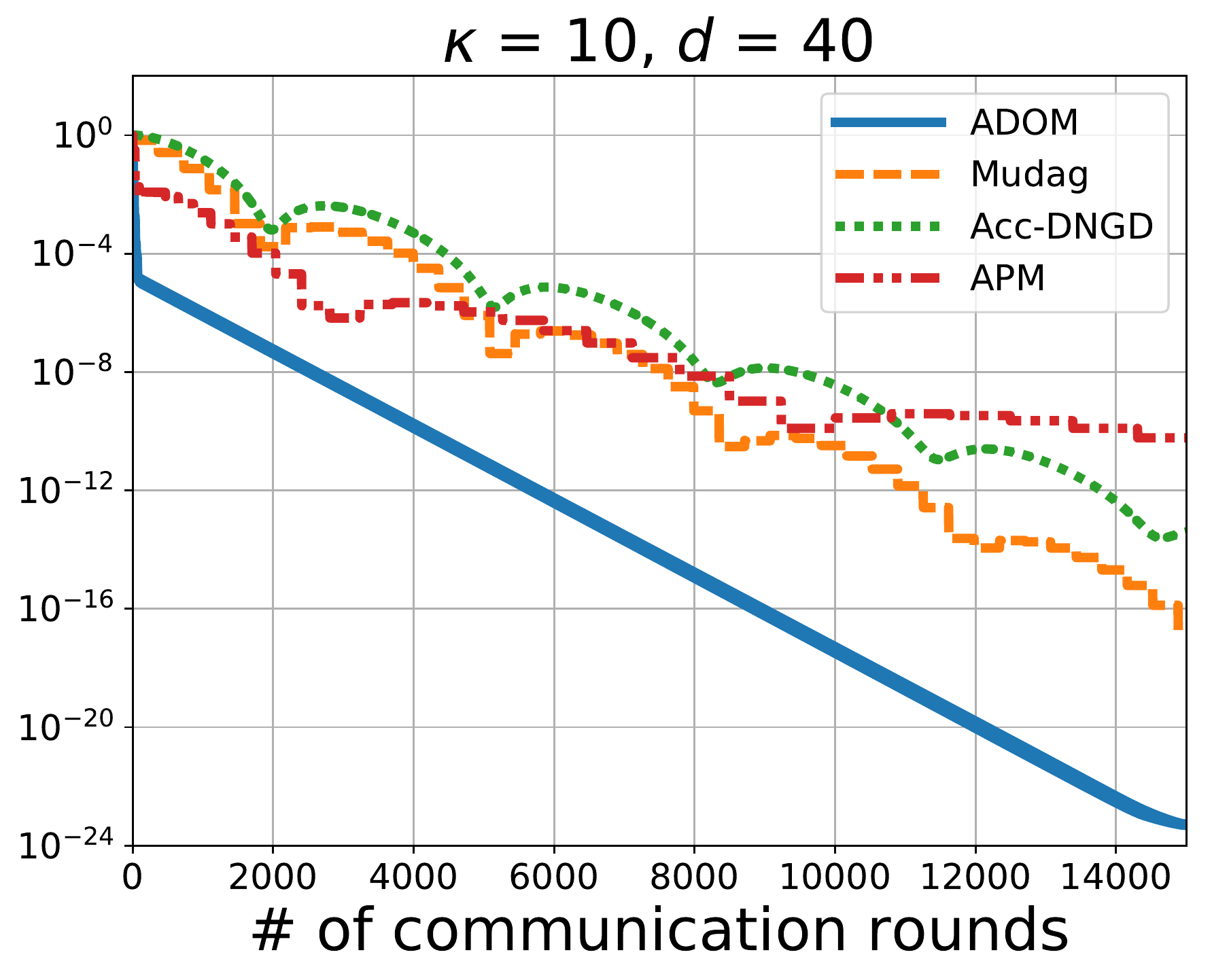}
	\includegraphics[width=0.24\linewidth]{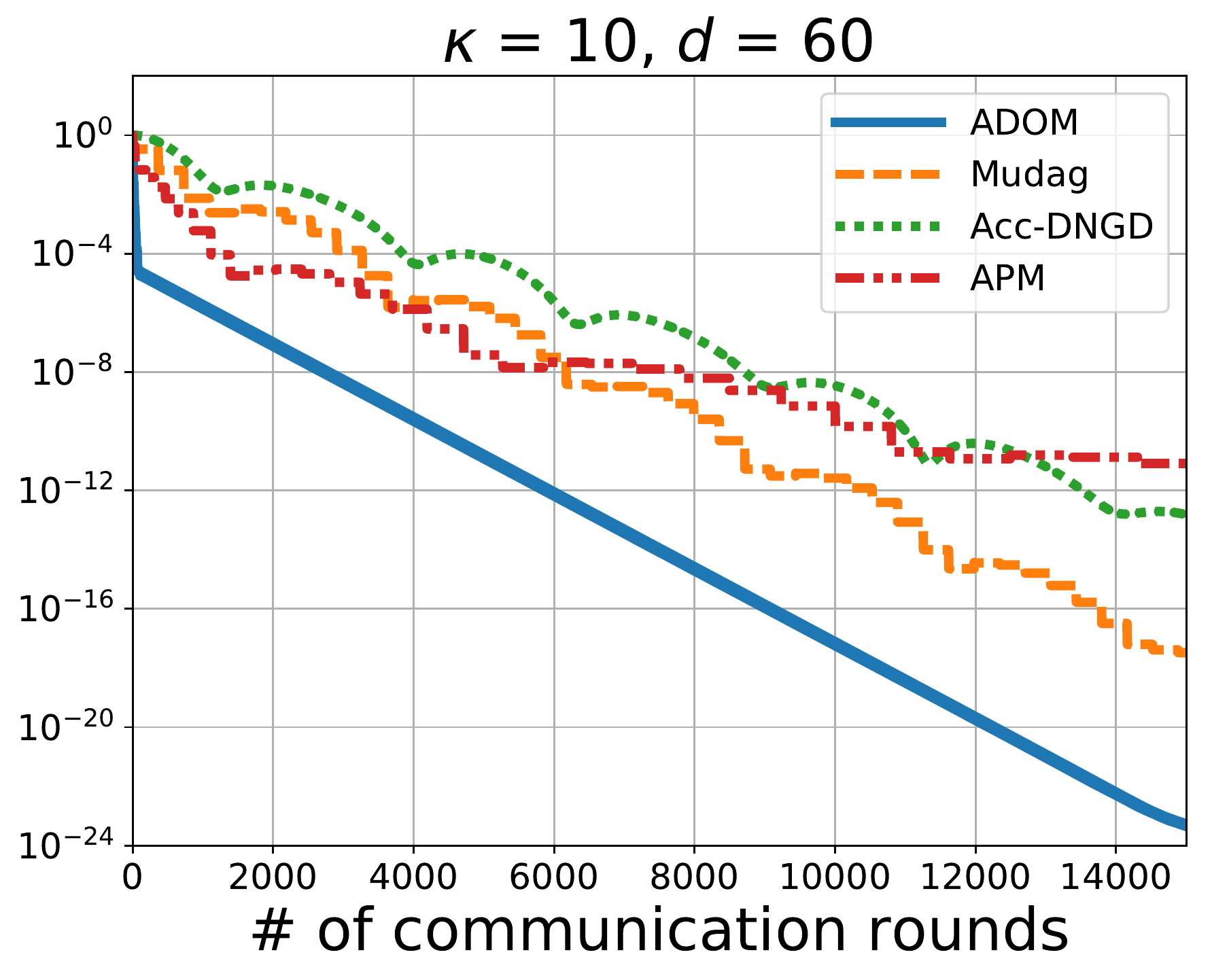}
	\includegraphics[width=0.24\linewidth]{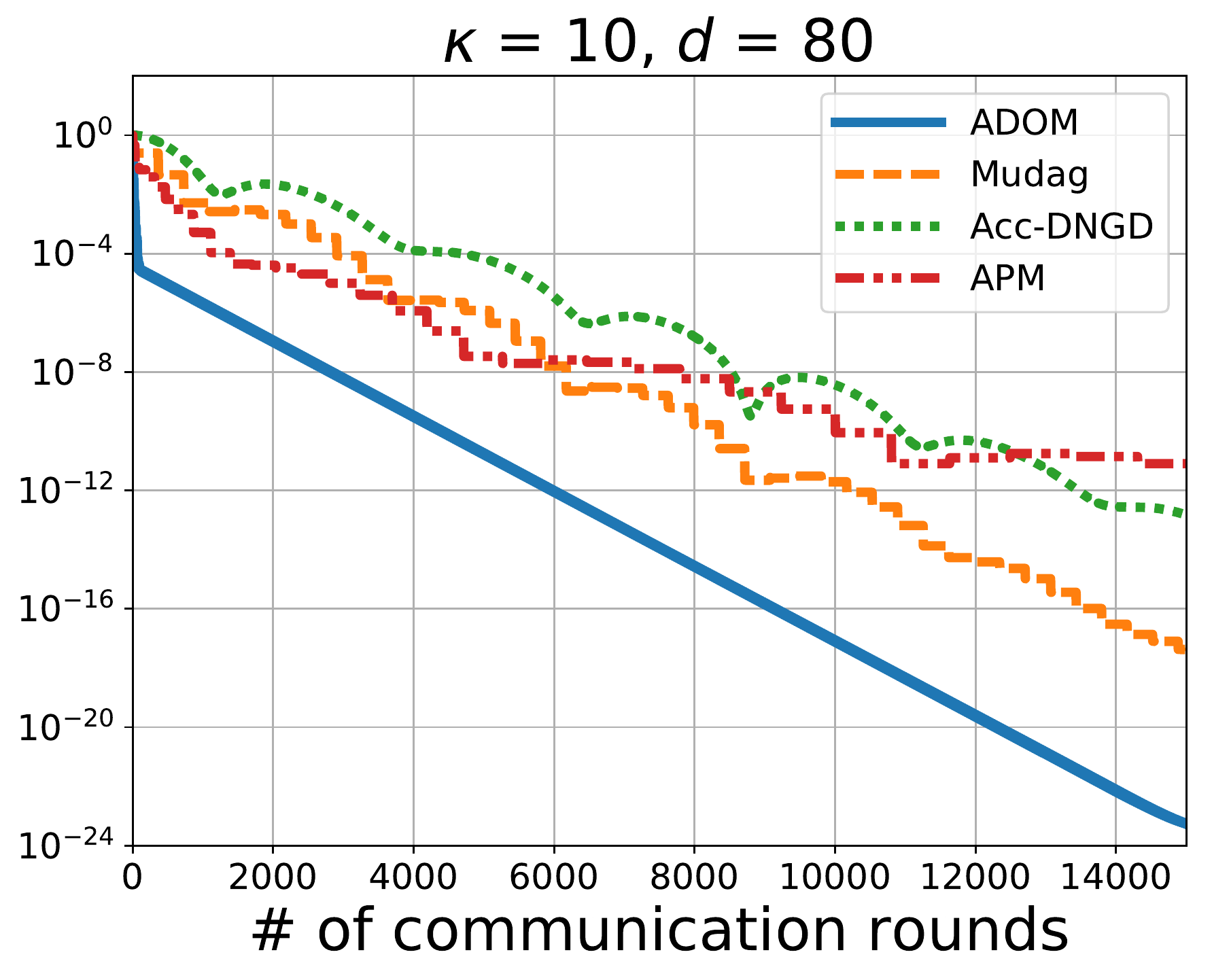}
	\includegraphics[width=0.24\linewidth]{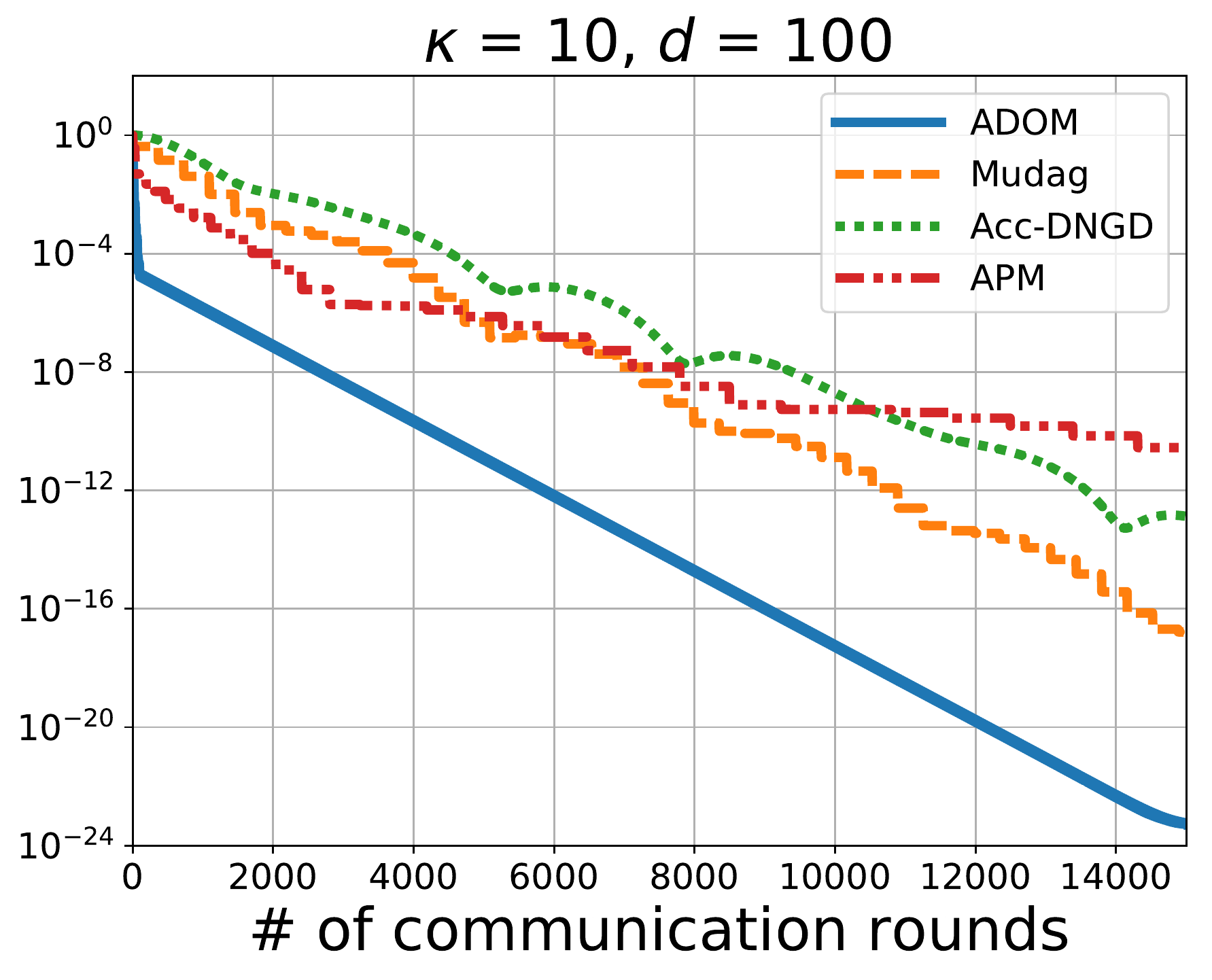}
	
	\includegraphics[width=0.24\linewidth]{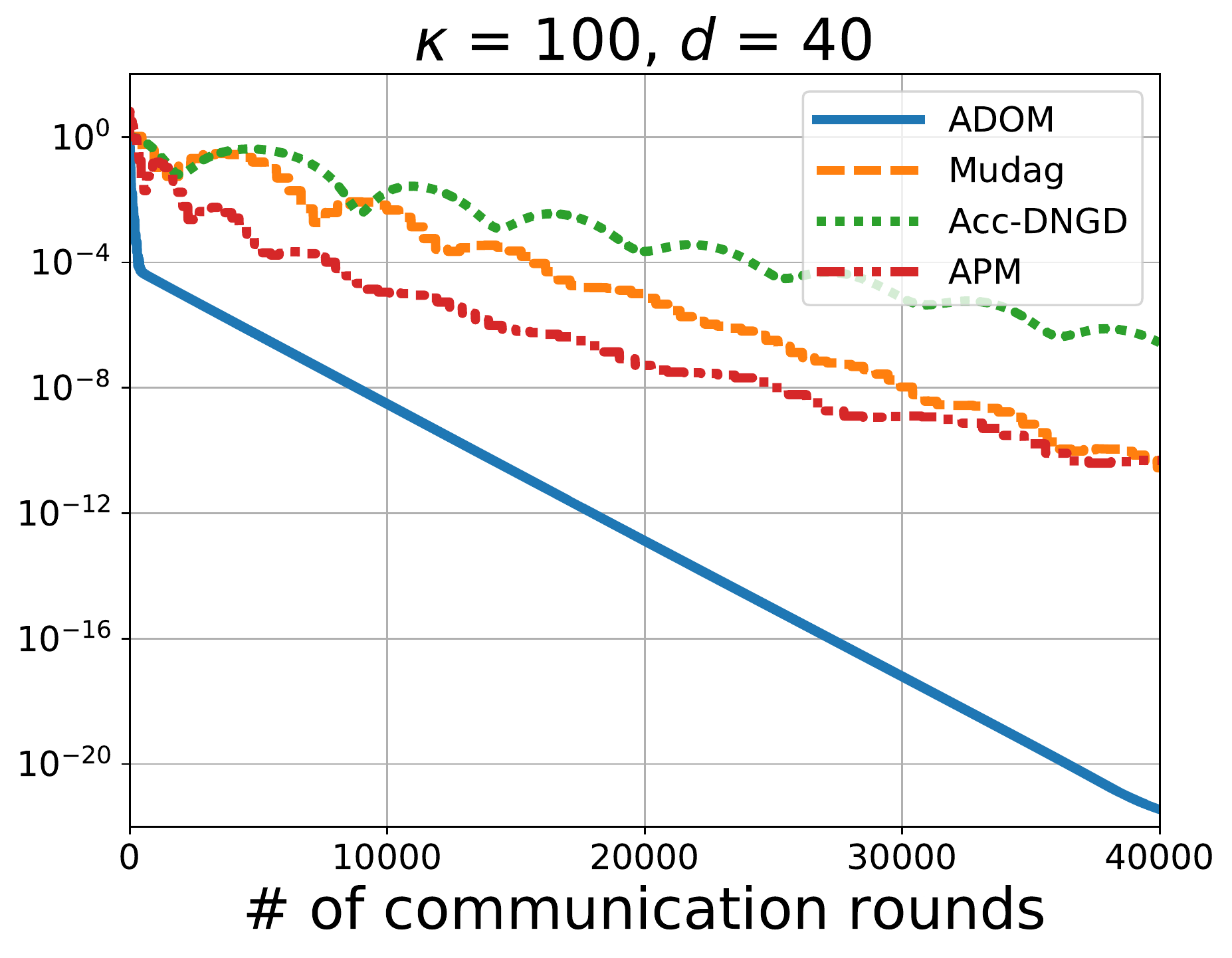}
	\includegraphics[width=0.24\linewidth]{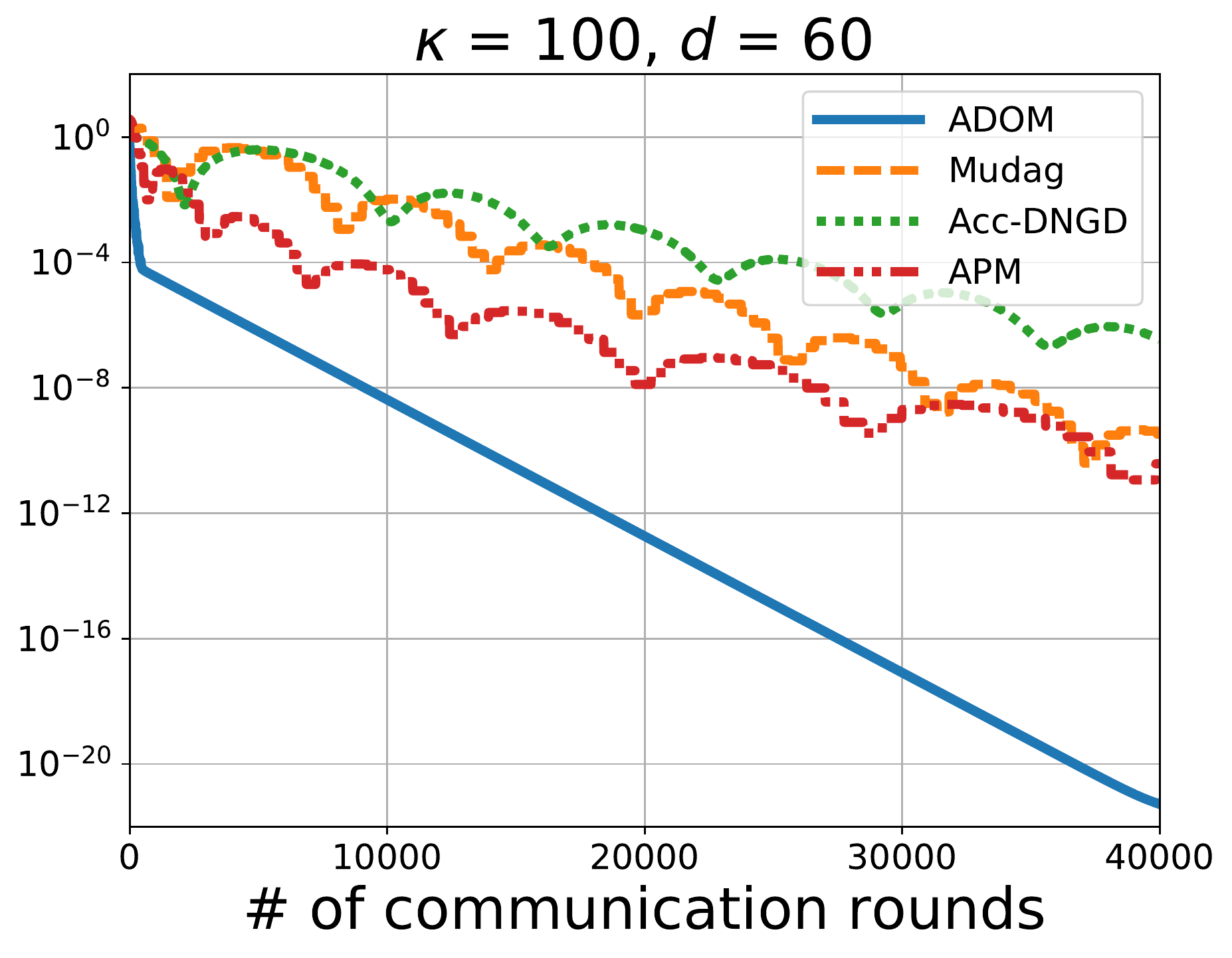}
	\includegraphics[width=0.24\linewidth]{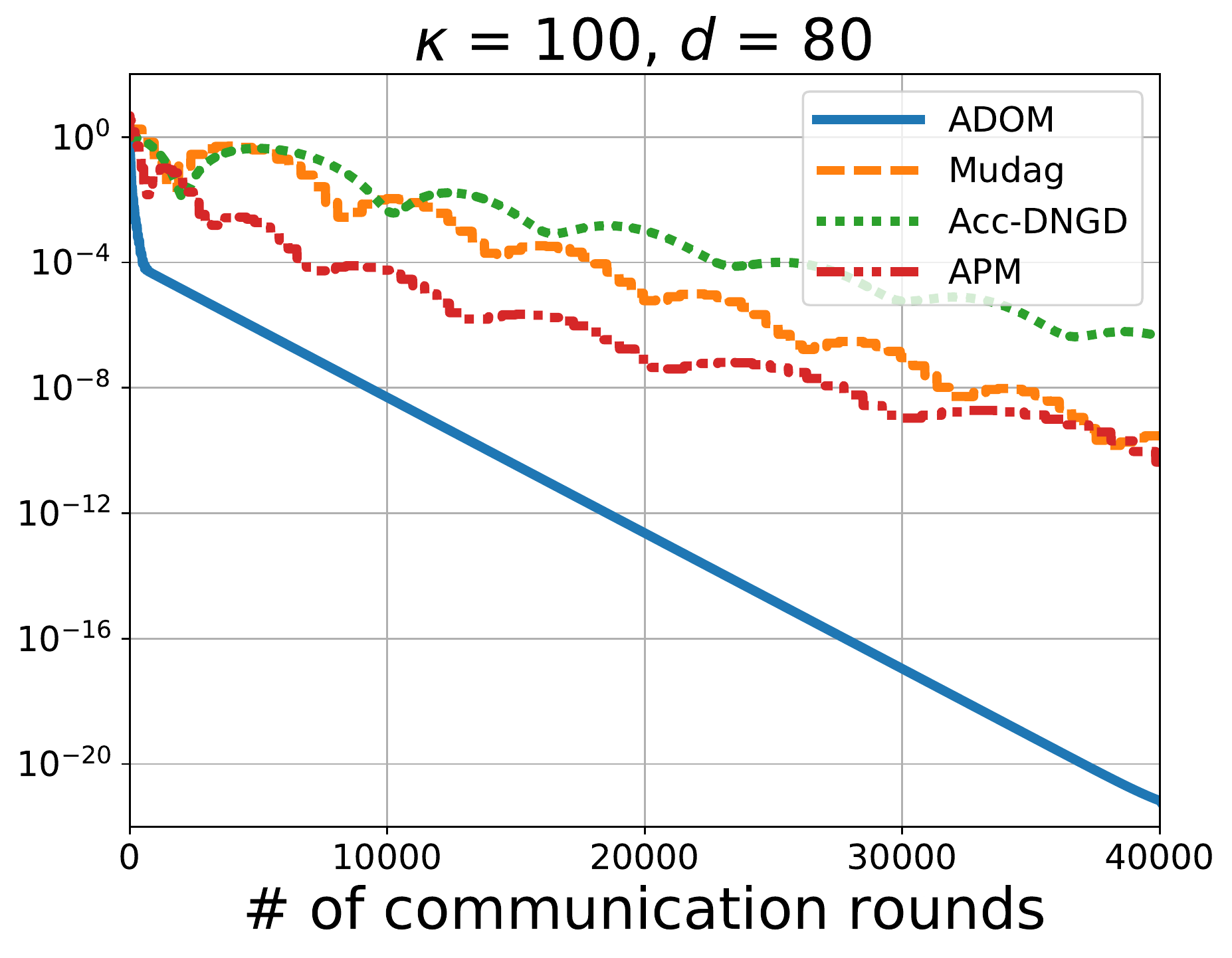}
	\includegraphics[width=0.24\linewidth]{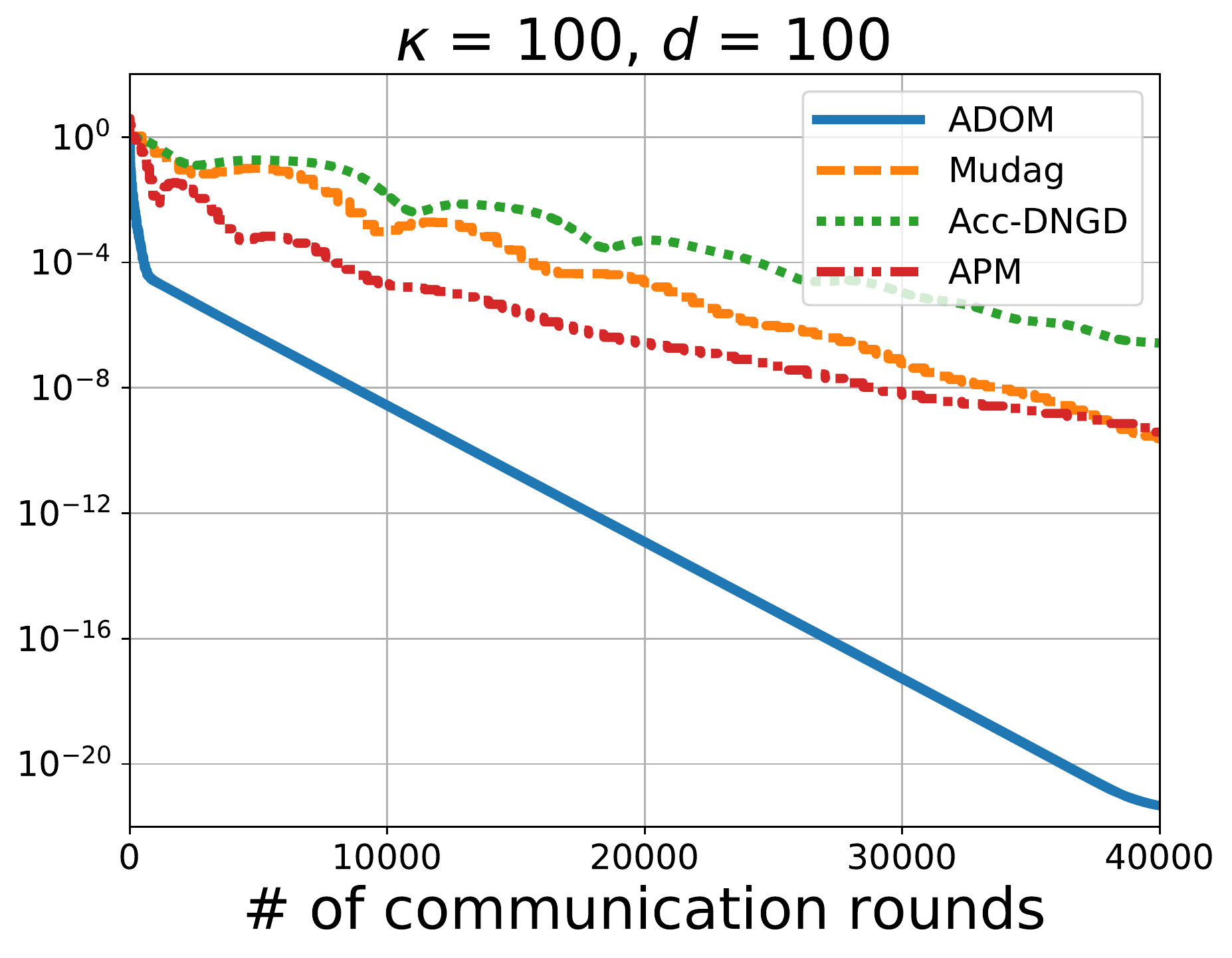}
	
	\includegraphics[width=0.24\linewidth]{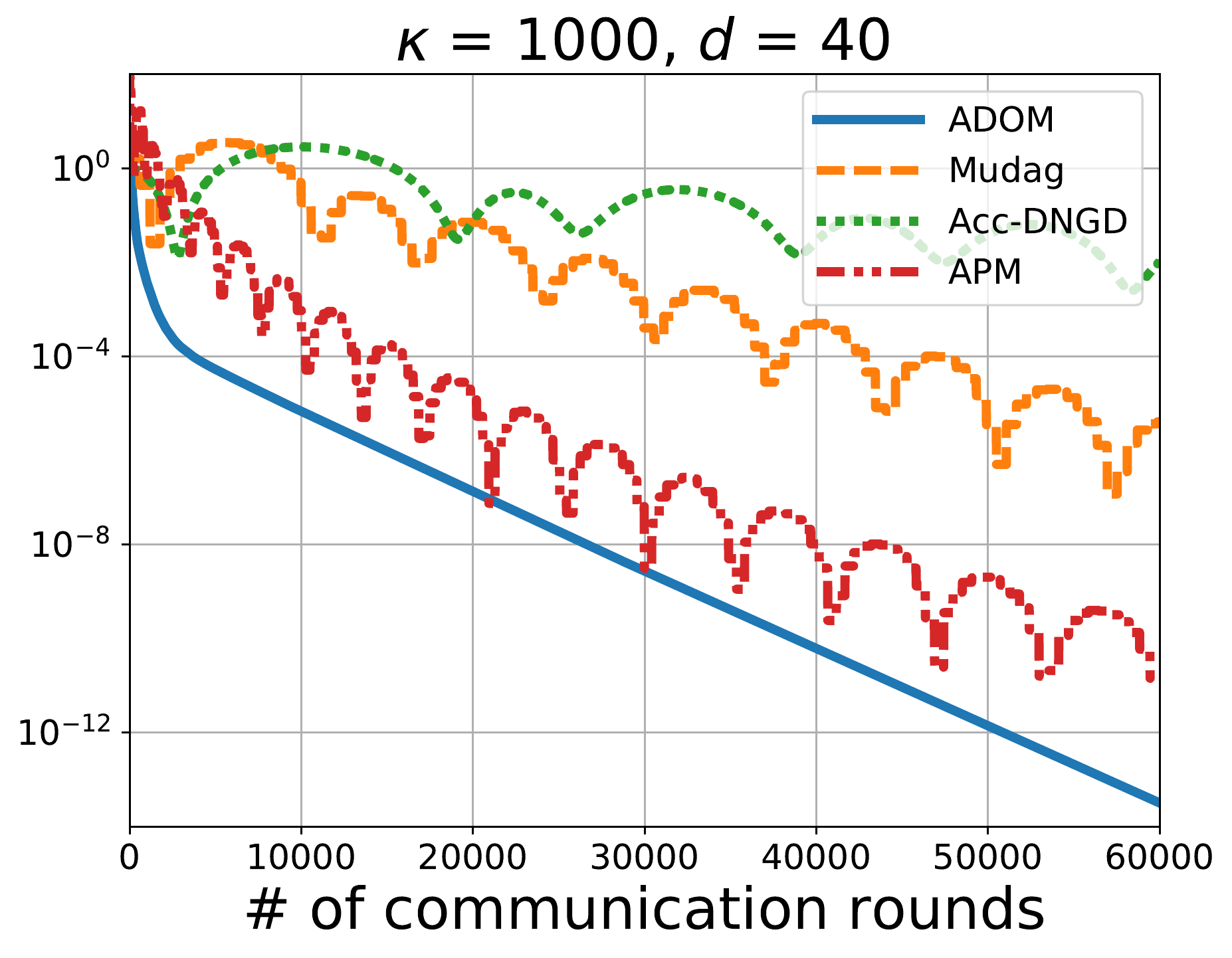}
	\includegraphics[width=0.24\linewidth]{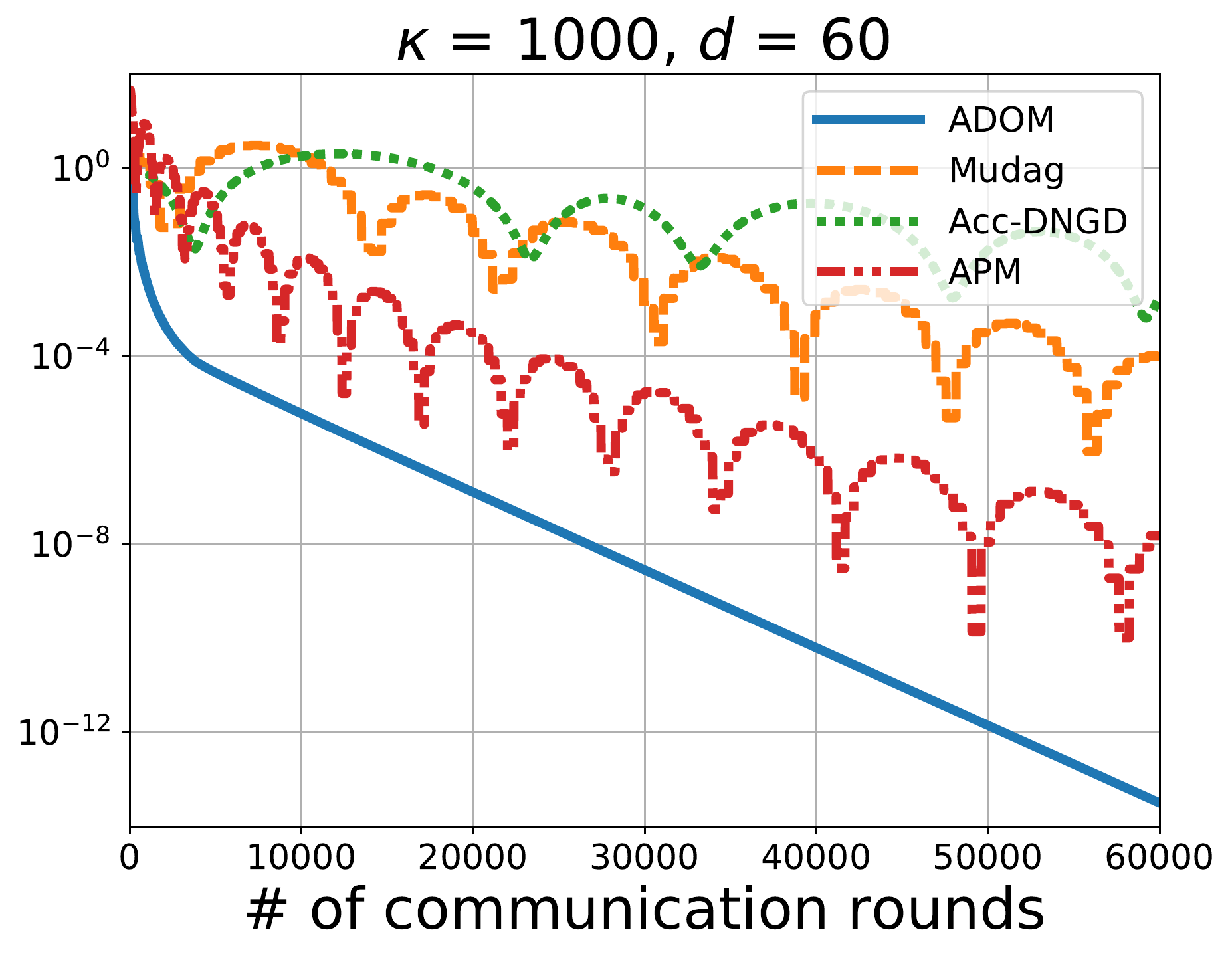}
	\includegraphics[width=0.24\linewidth]{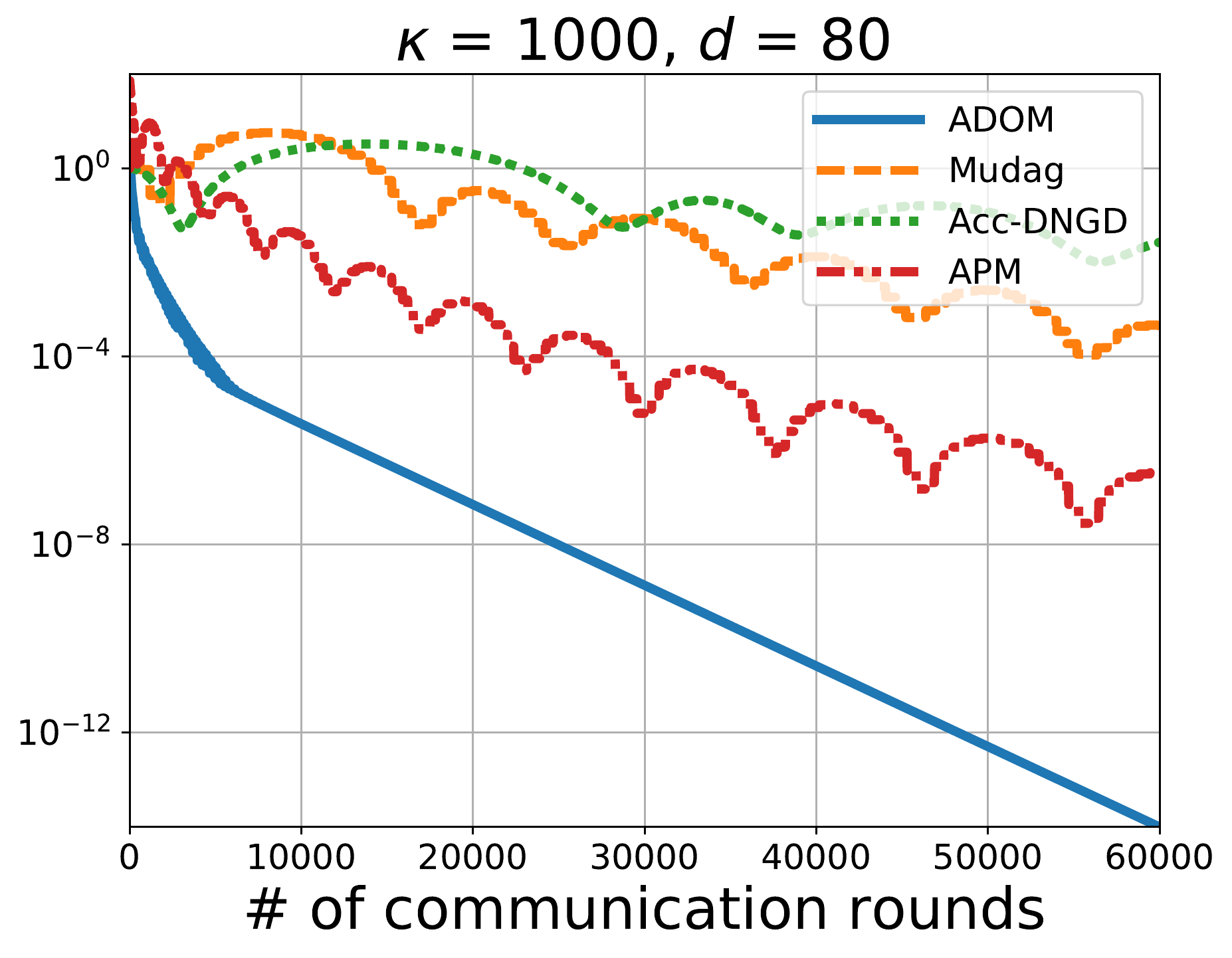}
	\includegraphics[width=0.24\linewidth]{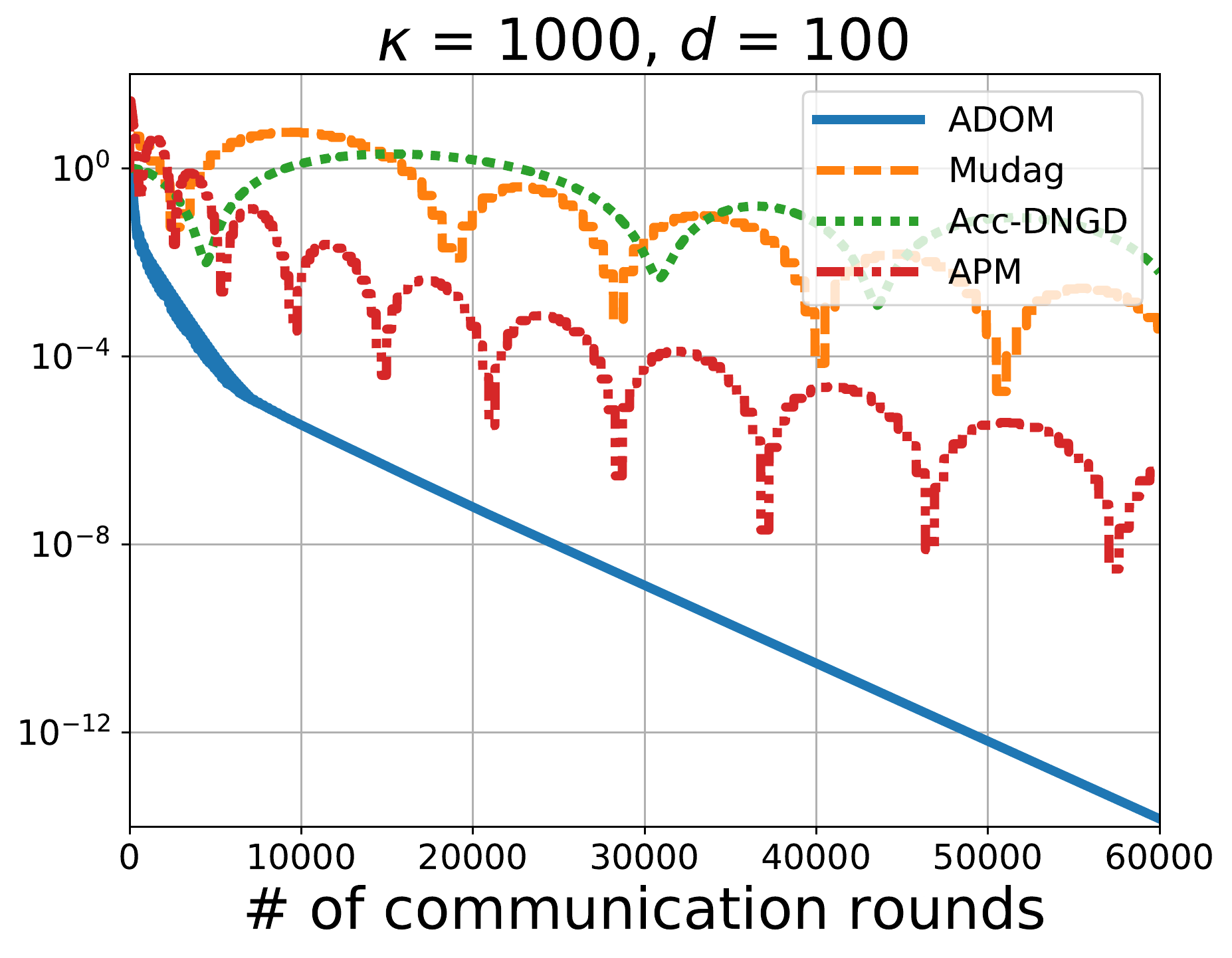}
	
	\includegraphics[width=0.24\linewidth]{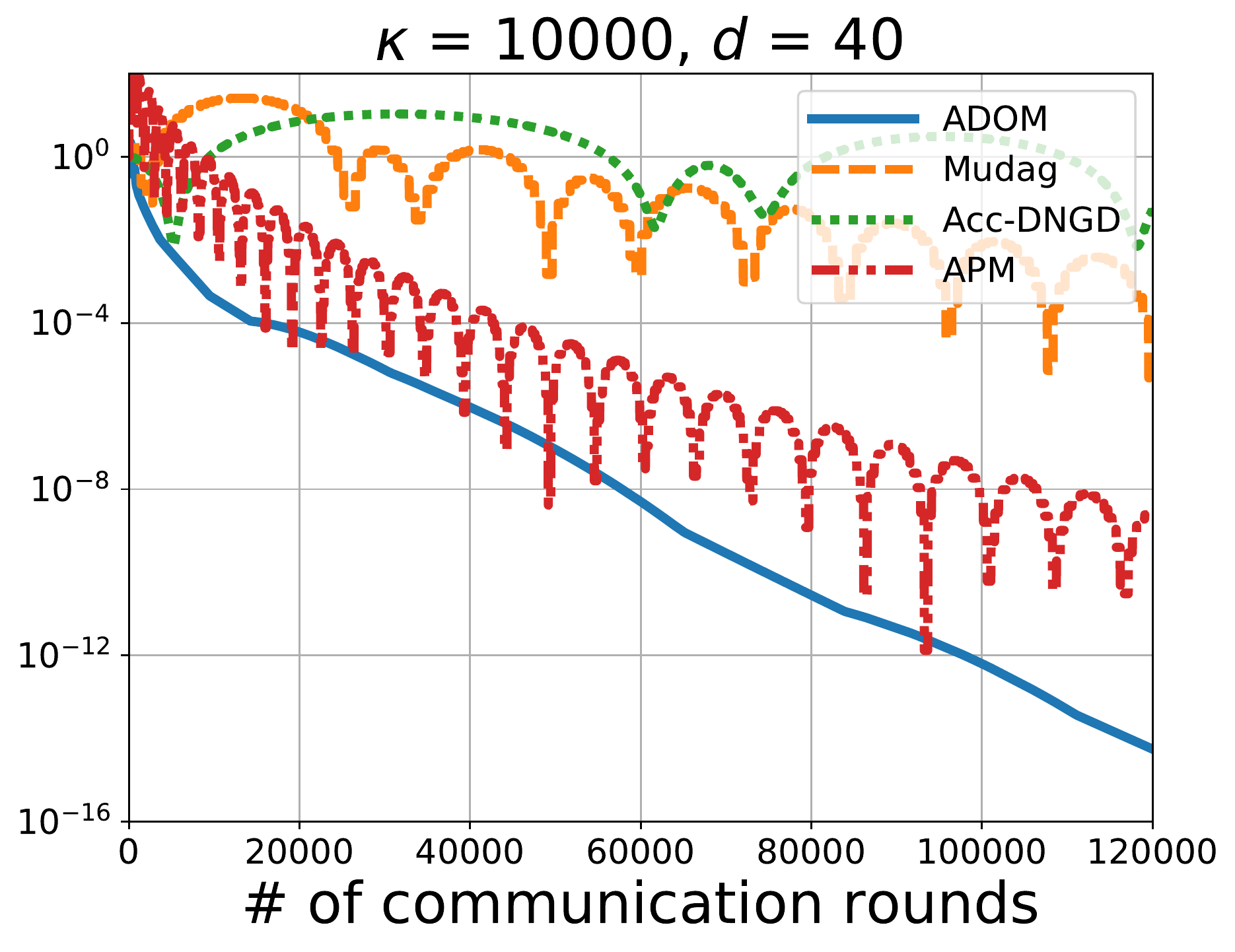}
	\includegraphics[width=0.24\linewidth]{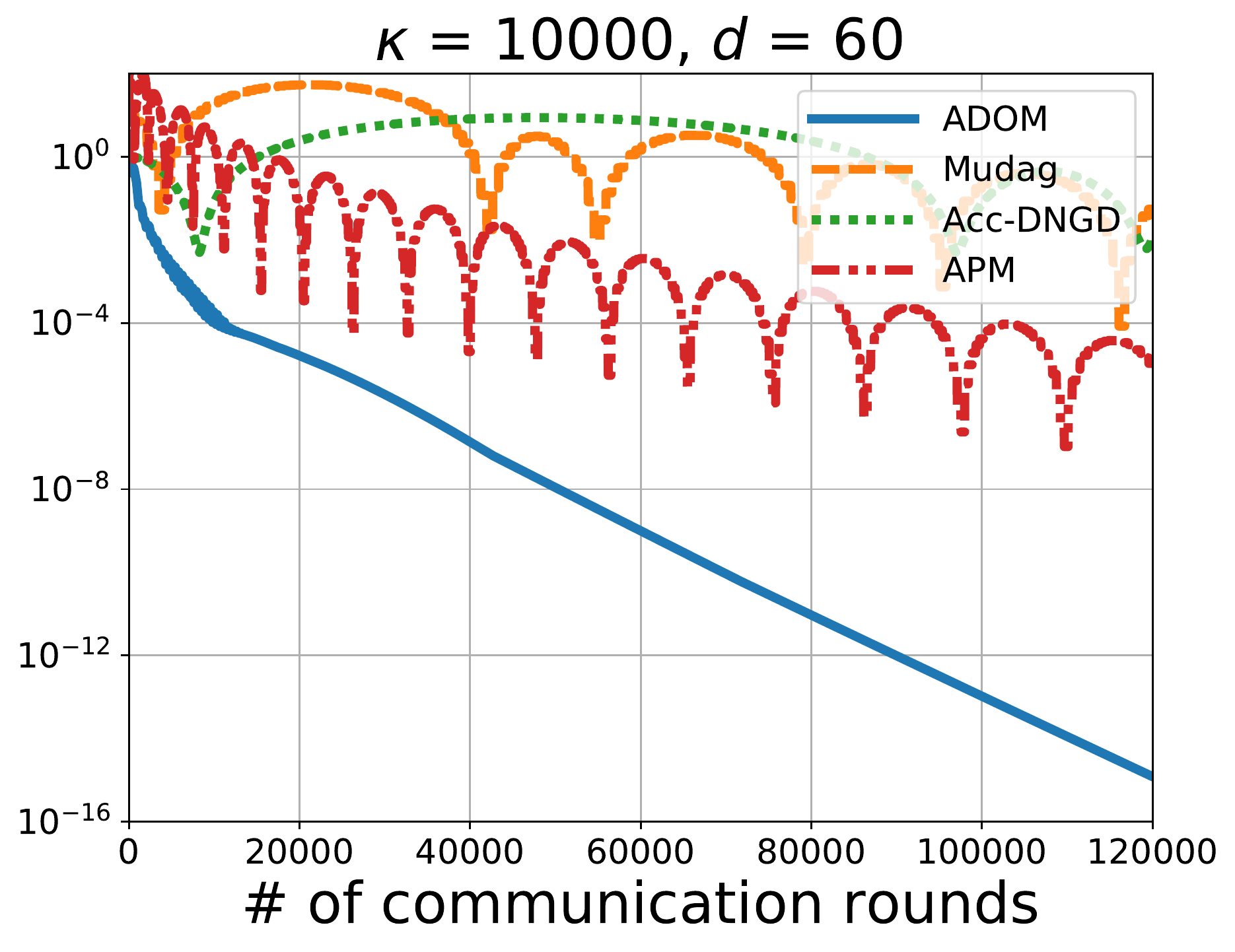}
	\includegraphics[width=0.24\linewidth]{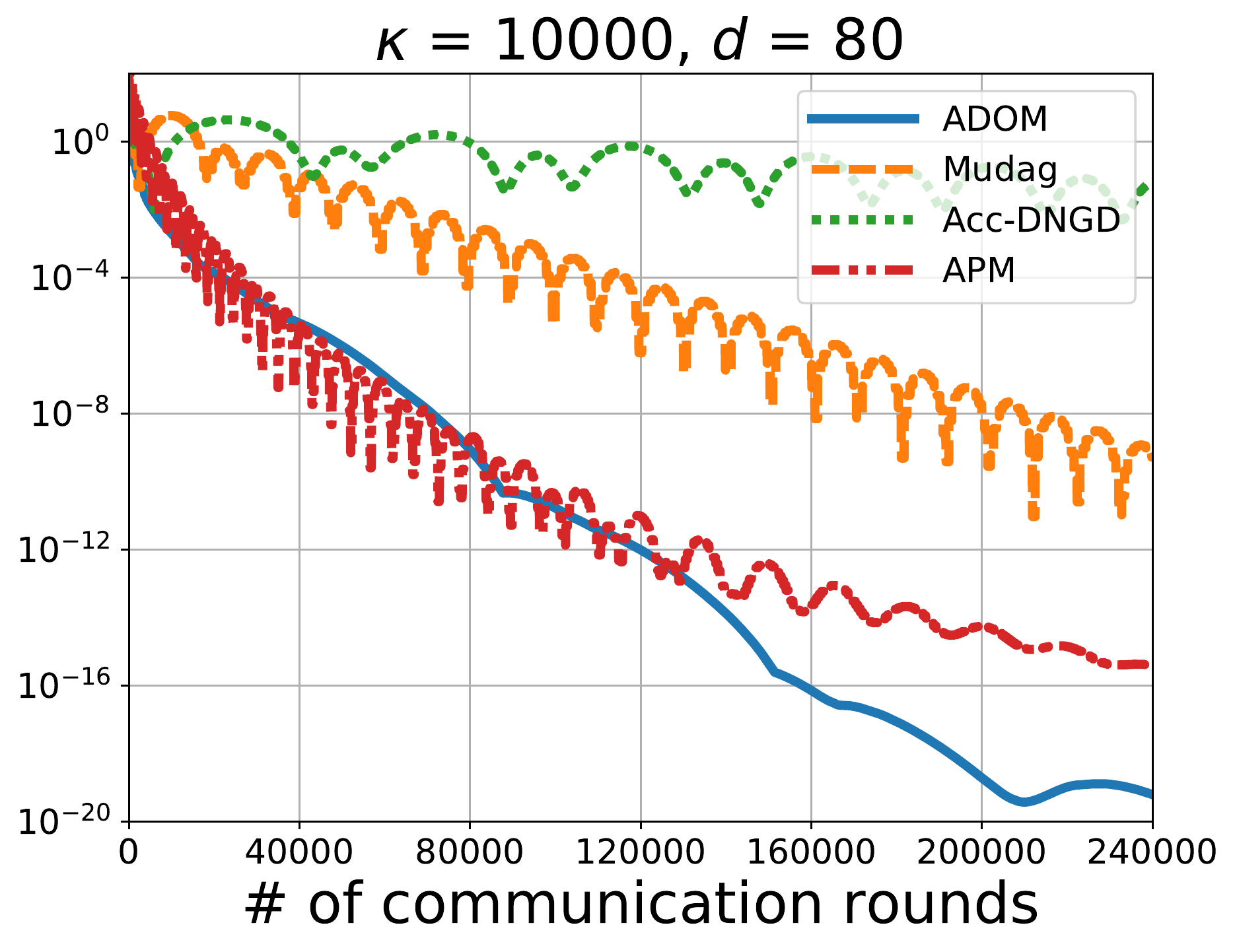}
	\includegraphics[width=0.24\linewidth]{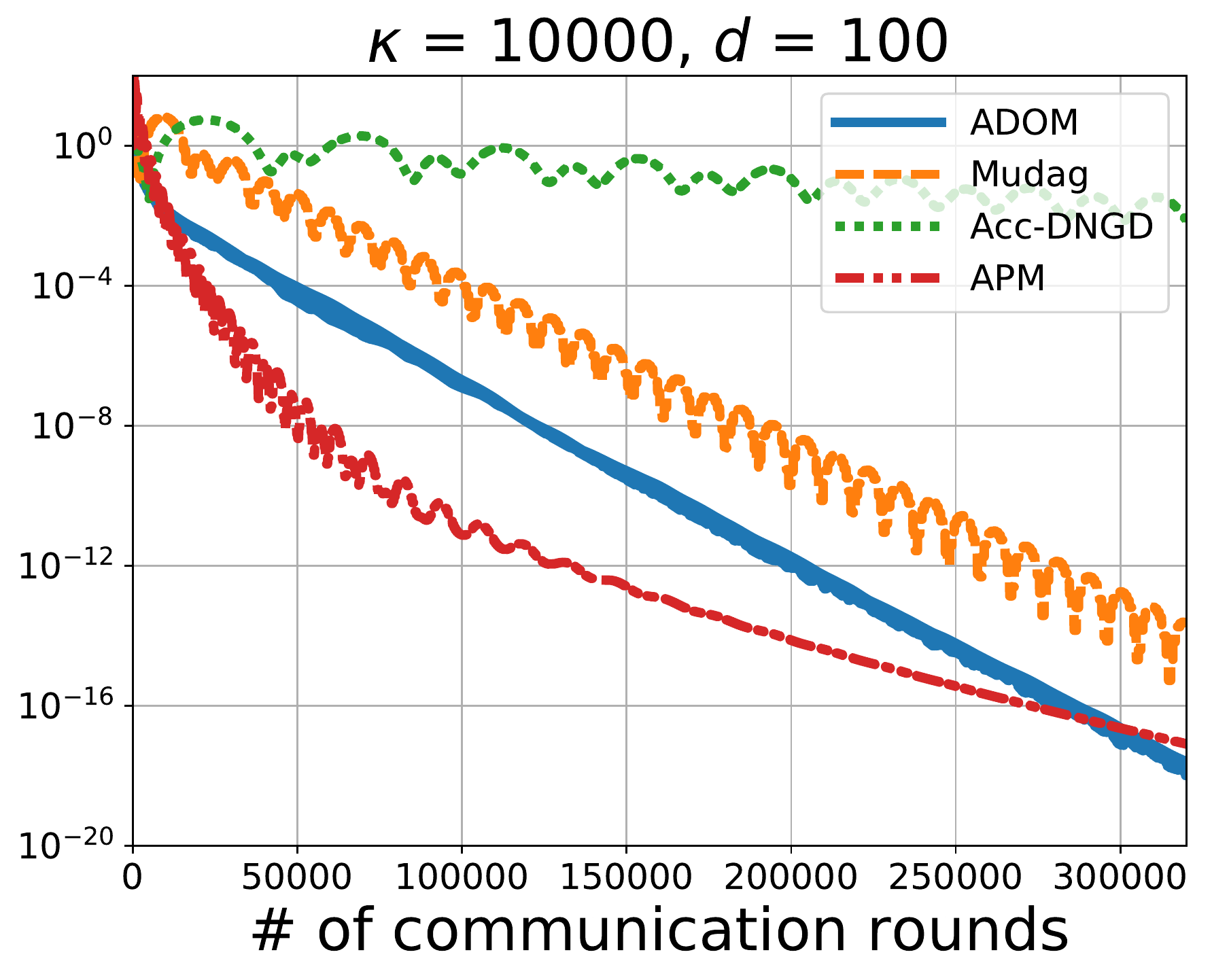}
	
	\caption{Comparison of Mudag, Acc-DNGD, APM and {\sf ADOM} on problems with $\chi \approx 30$, $d \in \{40,60,80,100\}$ and $\kappa \in \{10,10^2,10^3,10^4\}$.}
	\label{fig:kappa}	
\end{figure*}

\begin{figure*}[!htb]
	\centering
	\includegraphics[width=0.24\linewidth]{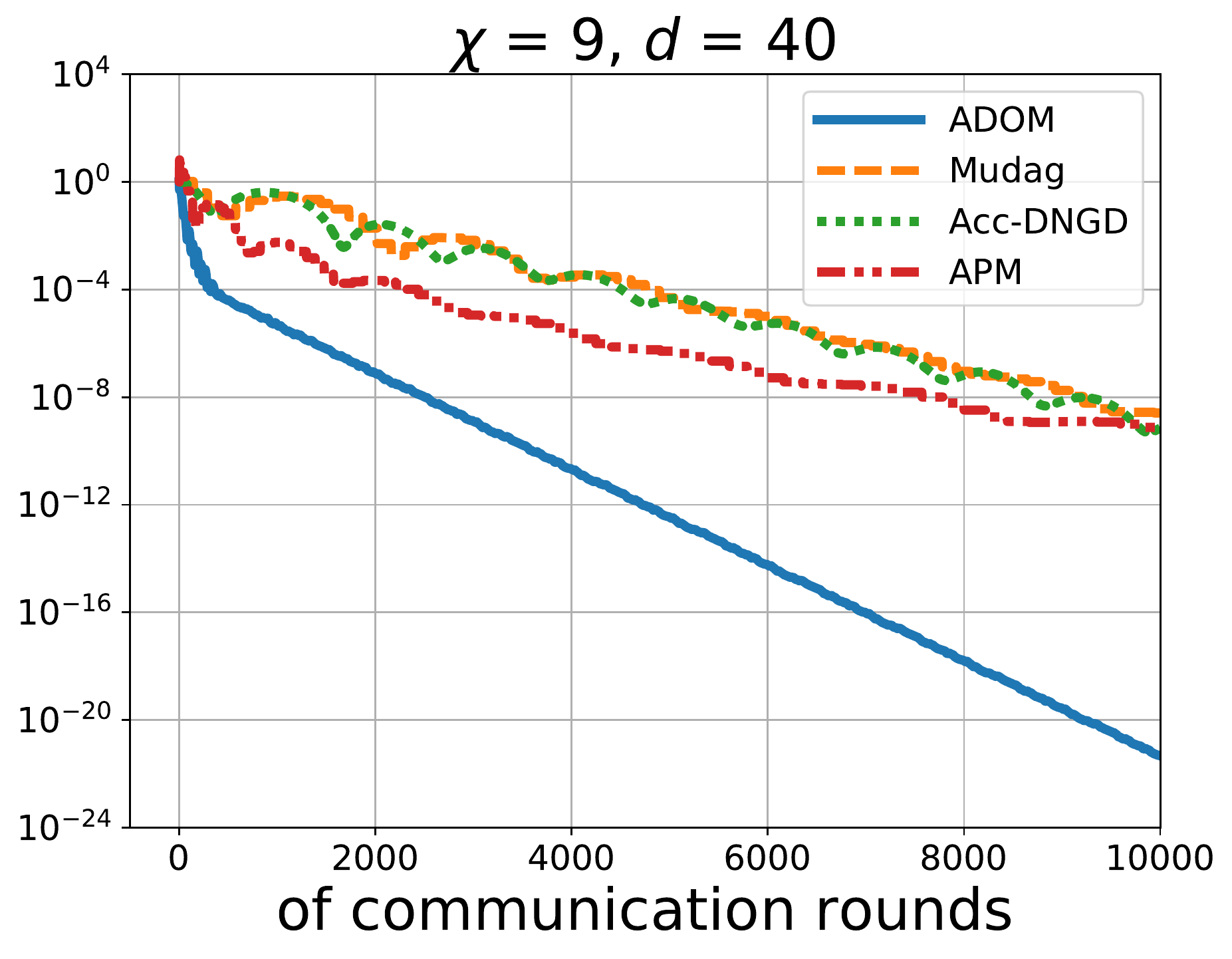}
	\includegraphics[width=0.24\linewidth]{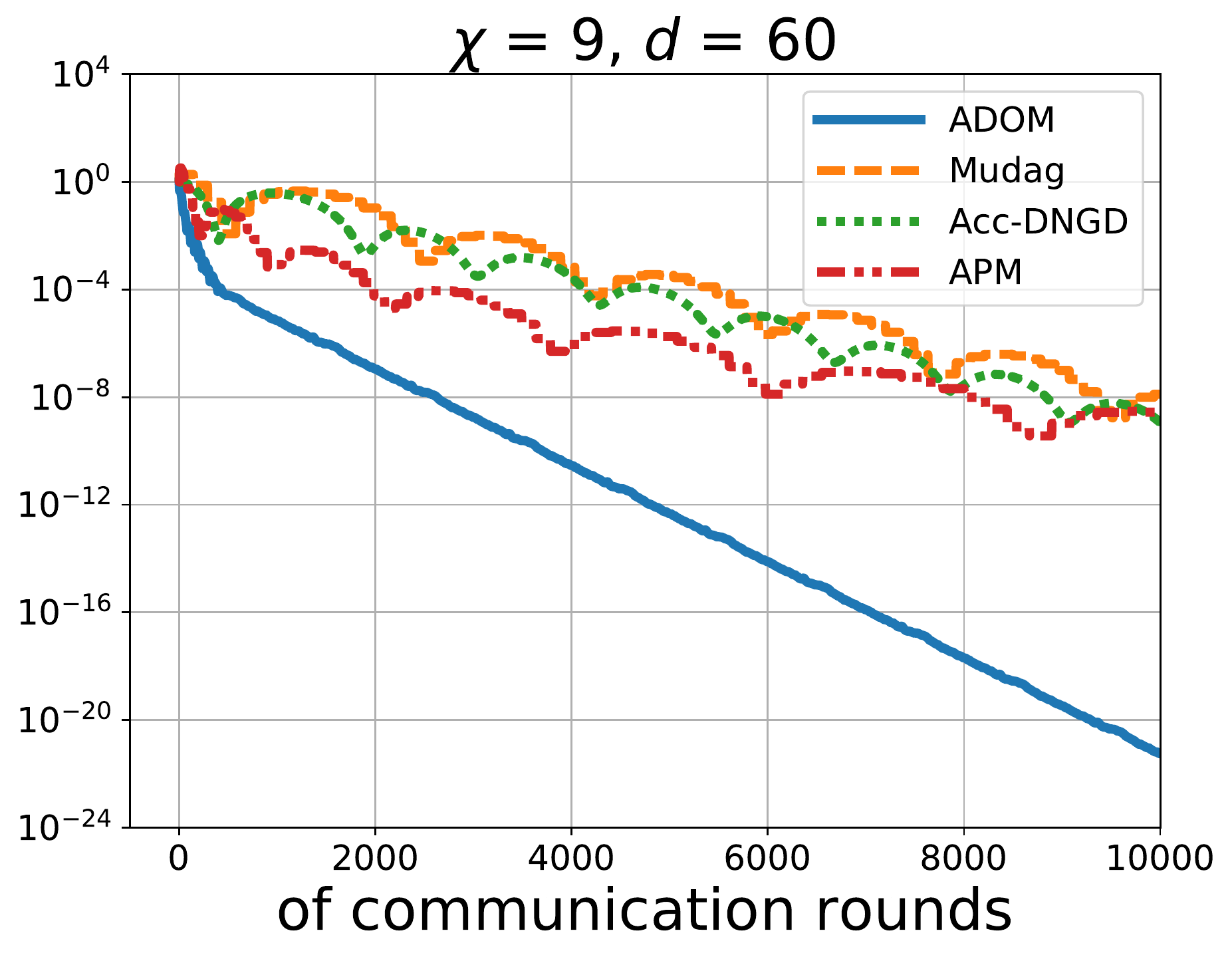}
	\includegraphics[width=0.24\linewidth]{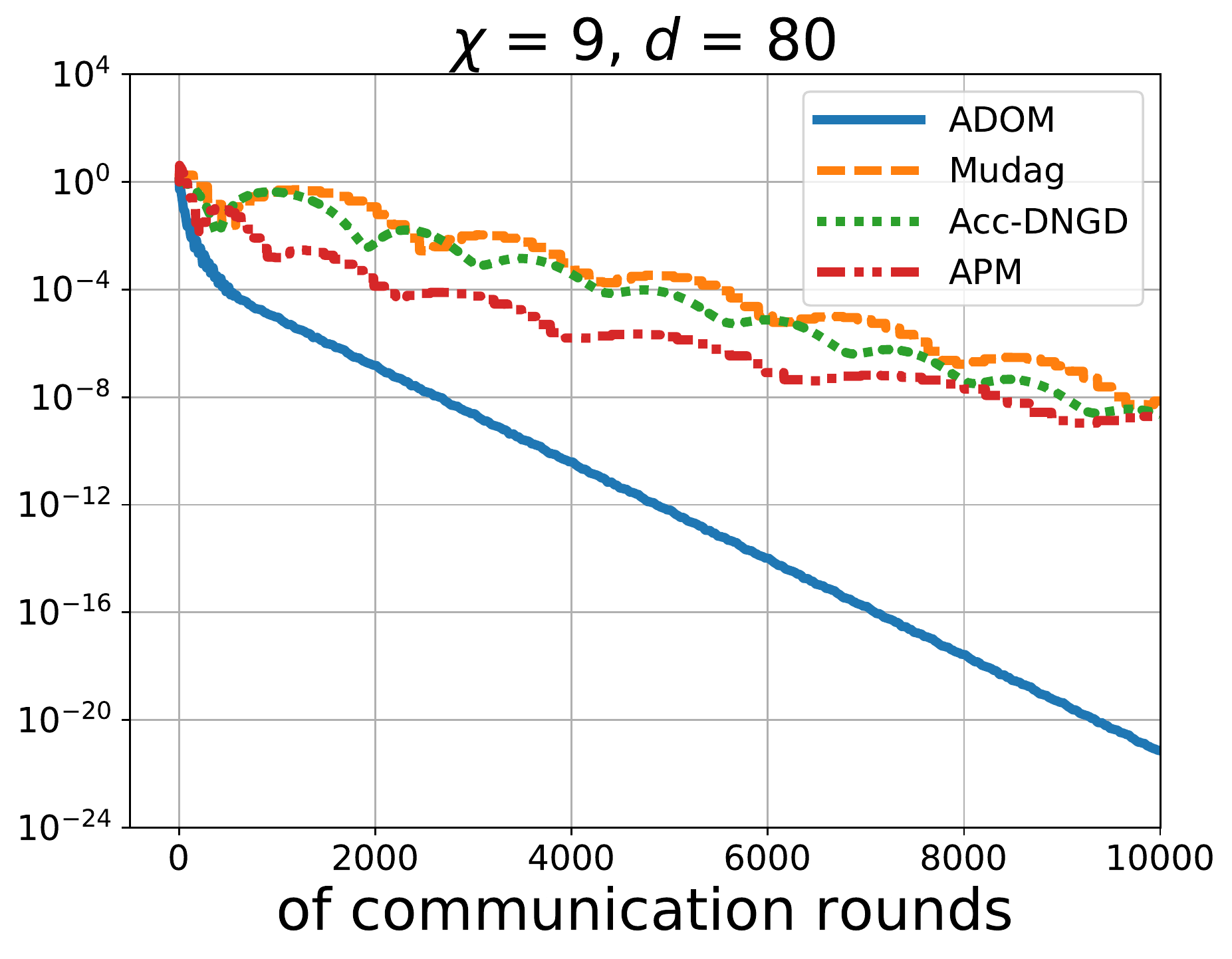}
	\includegraphics[width=0.24\linewidth]{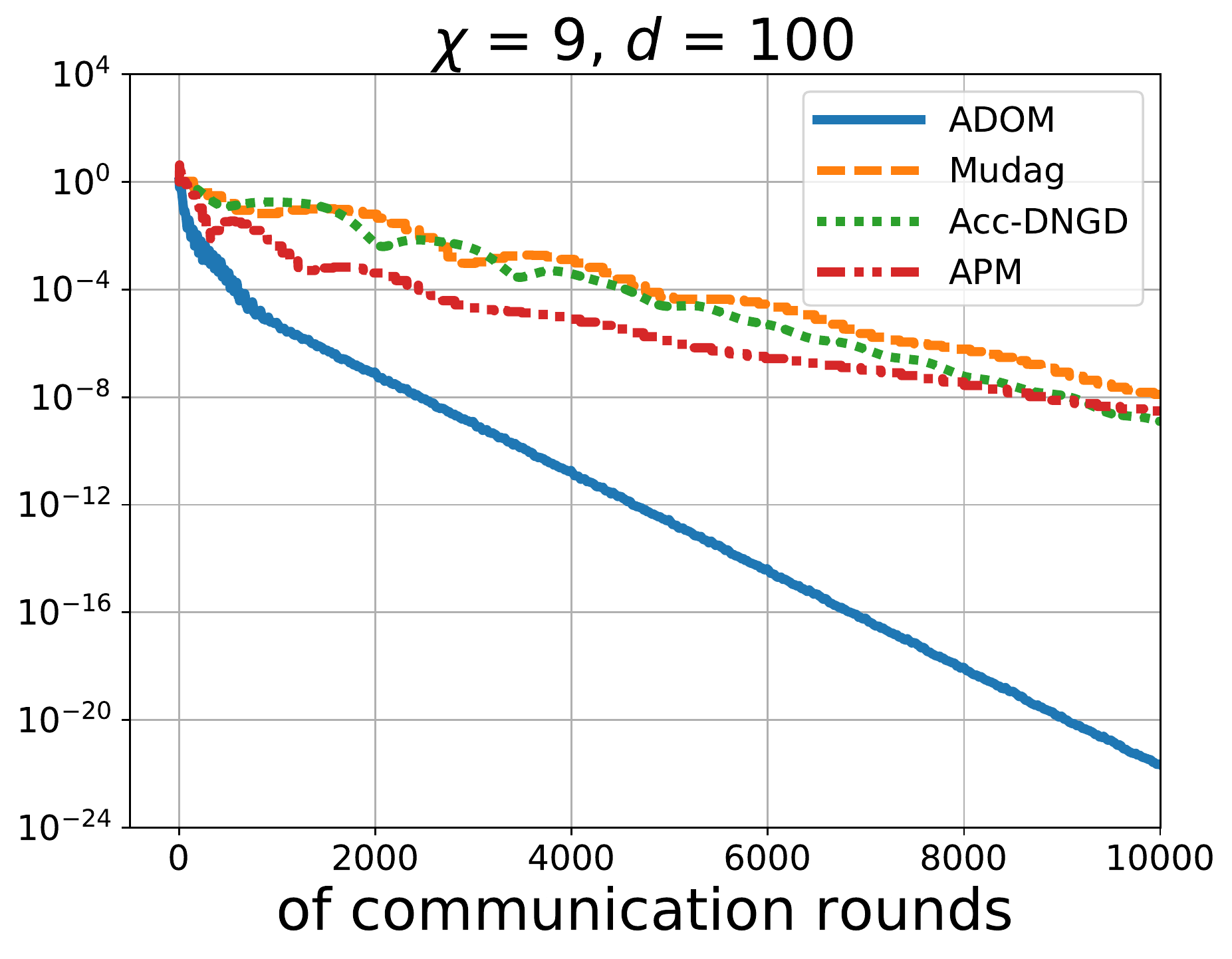}
	
	\includegraphics[width=0.24\linewidth]{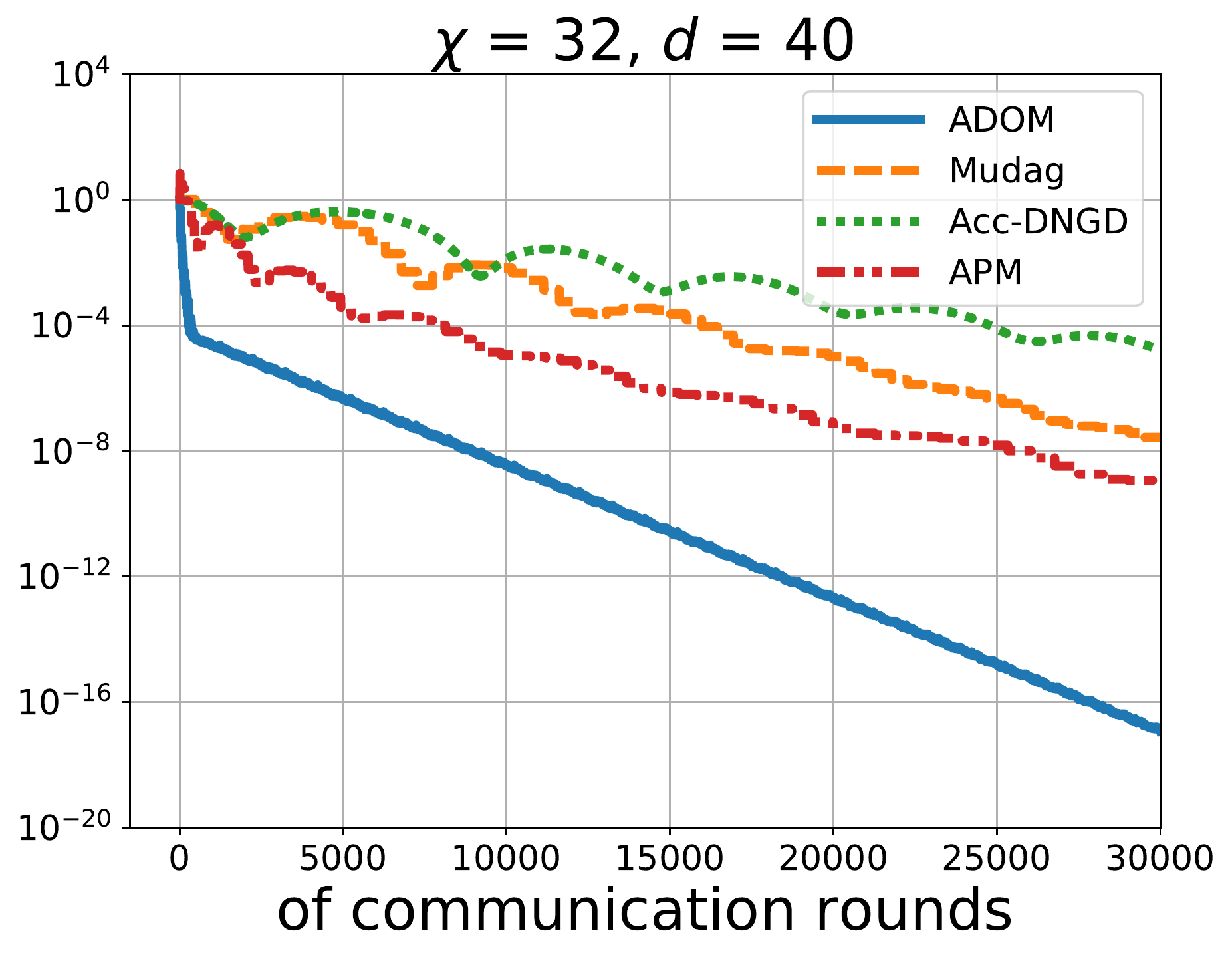}
	\includegraphics[width=0.24\linewidth]{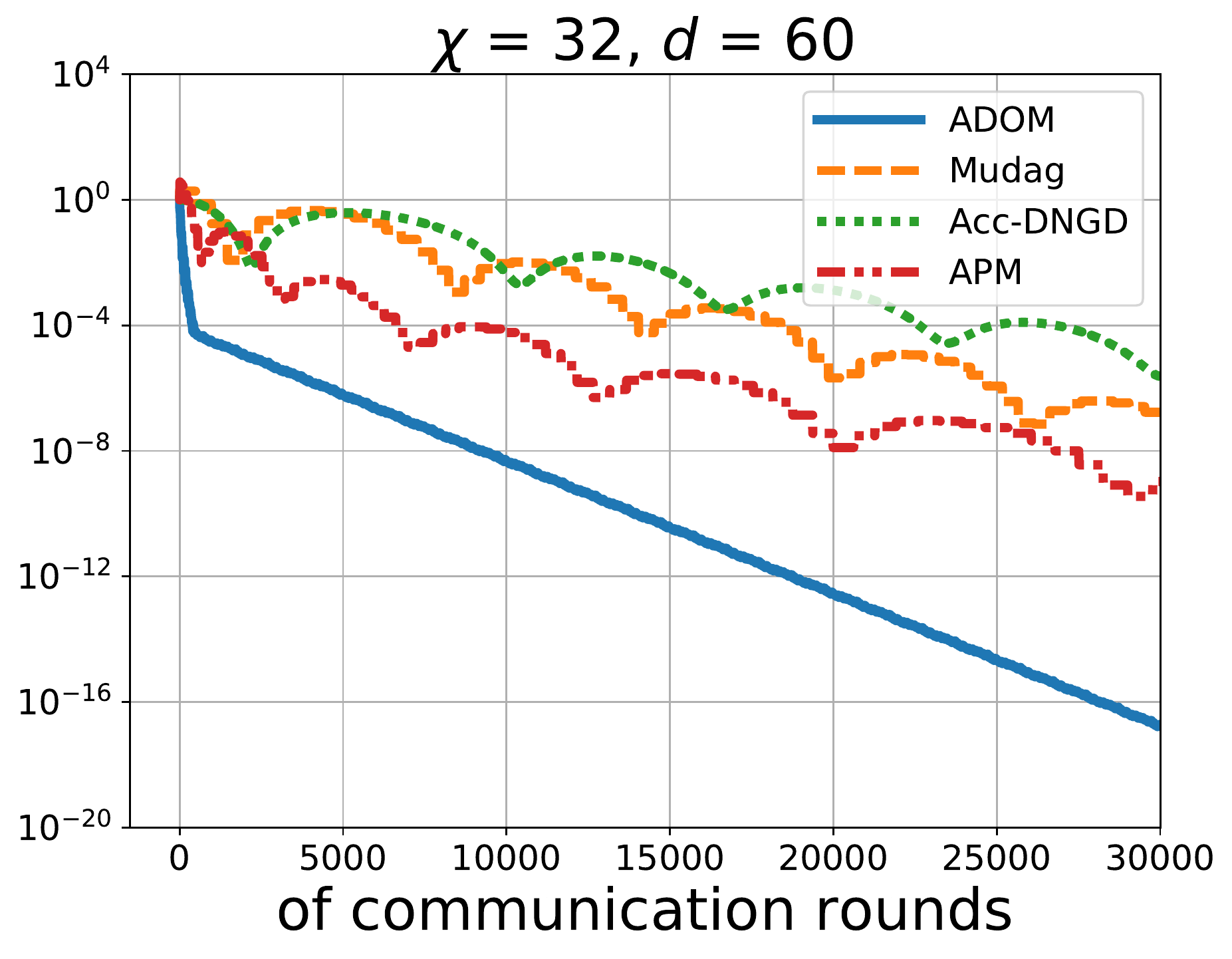}
	\includegraphics[width=0.24\linewidth]{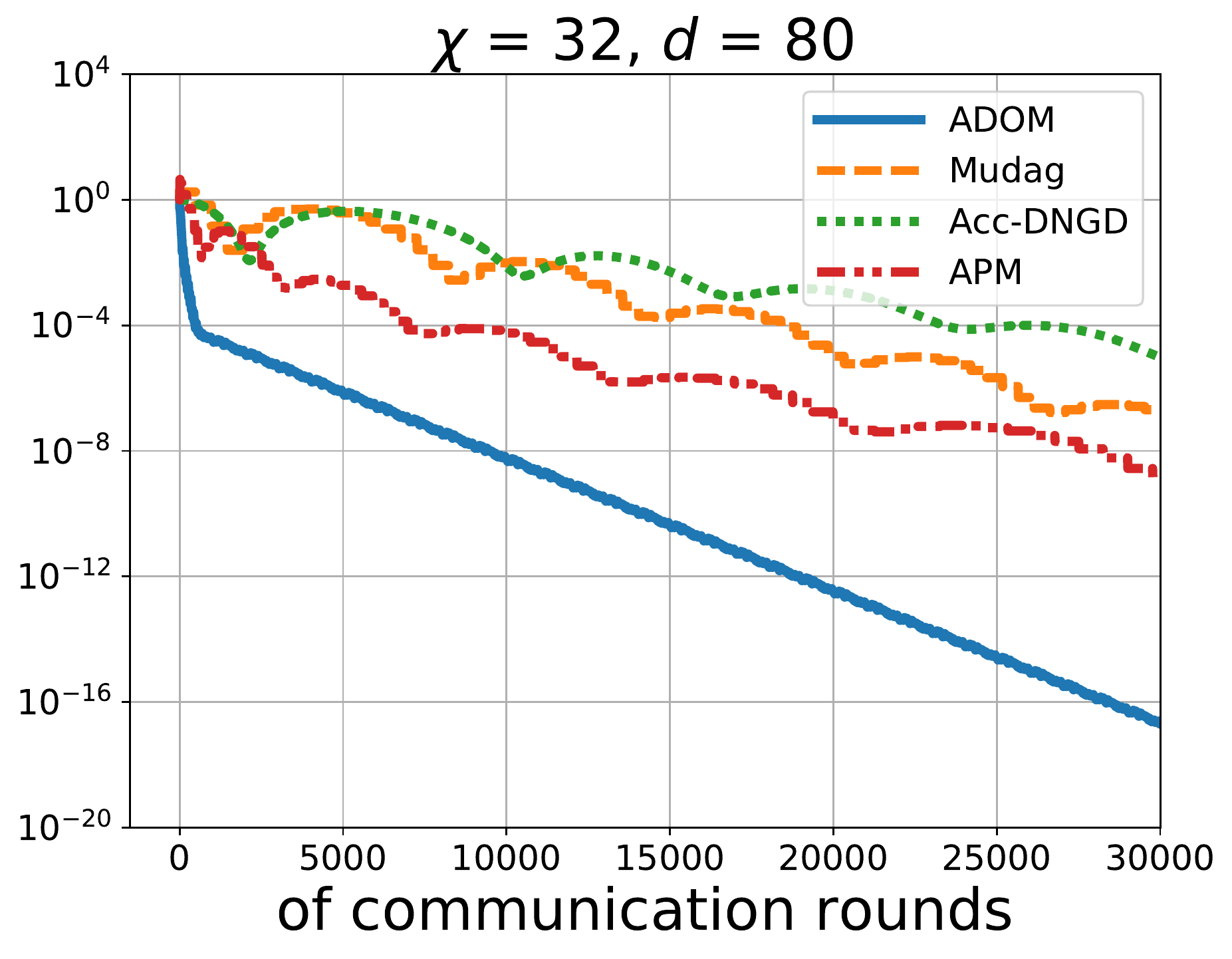}
	\includegraphics[width=0.24\linewidth]{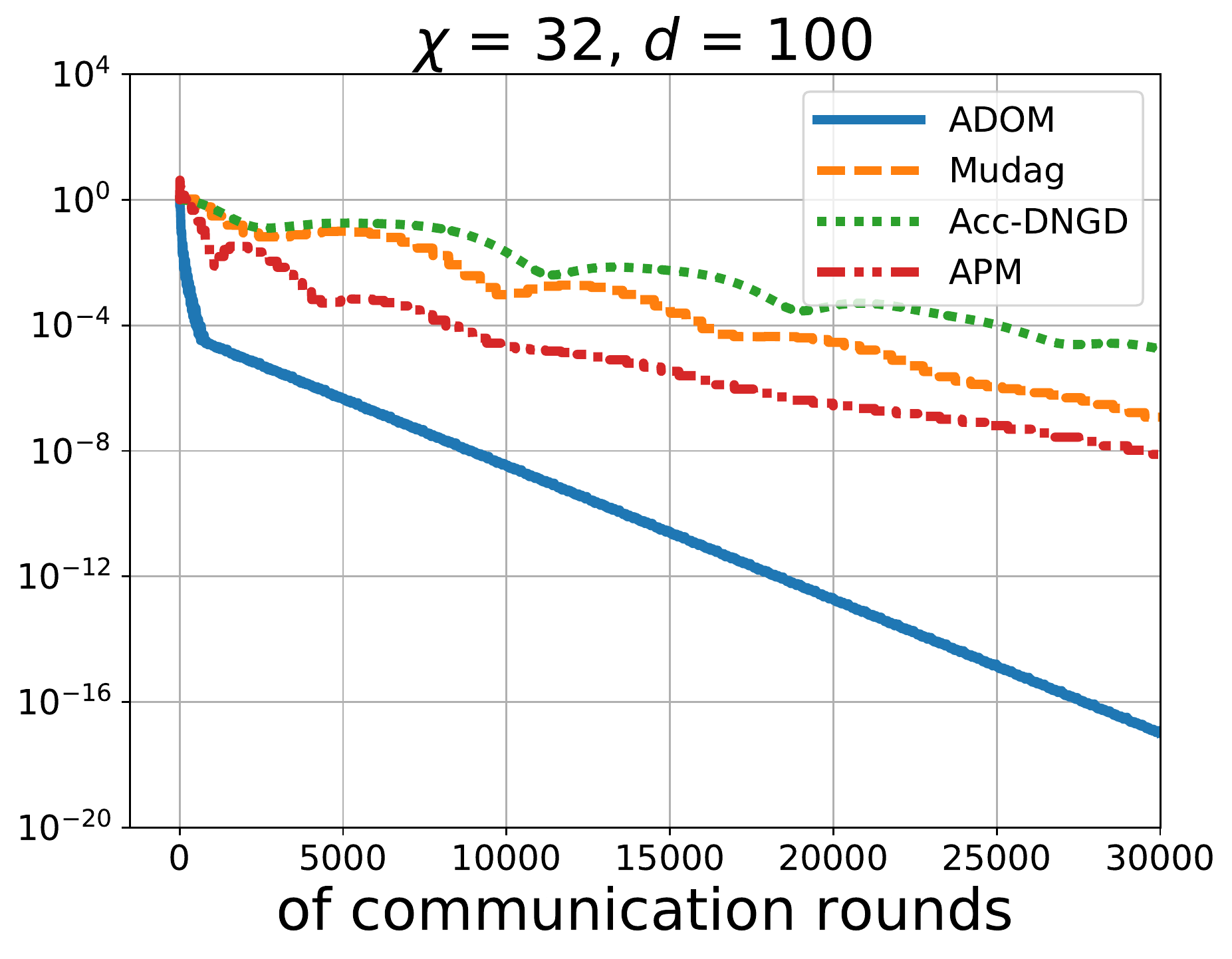}
	
	\includegraphics[width=0.24\linewidth]{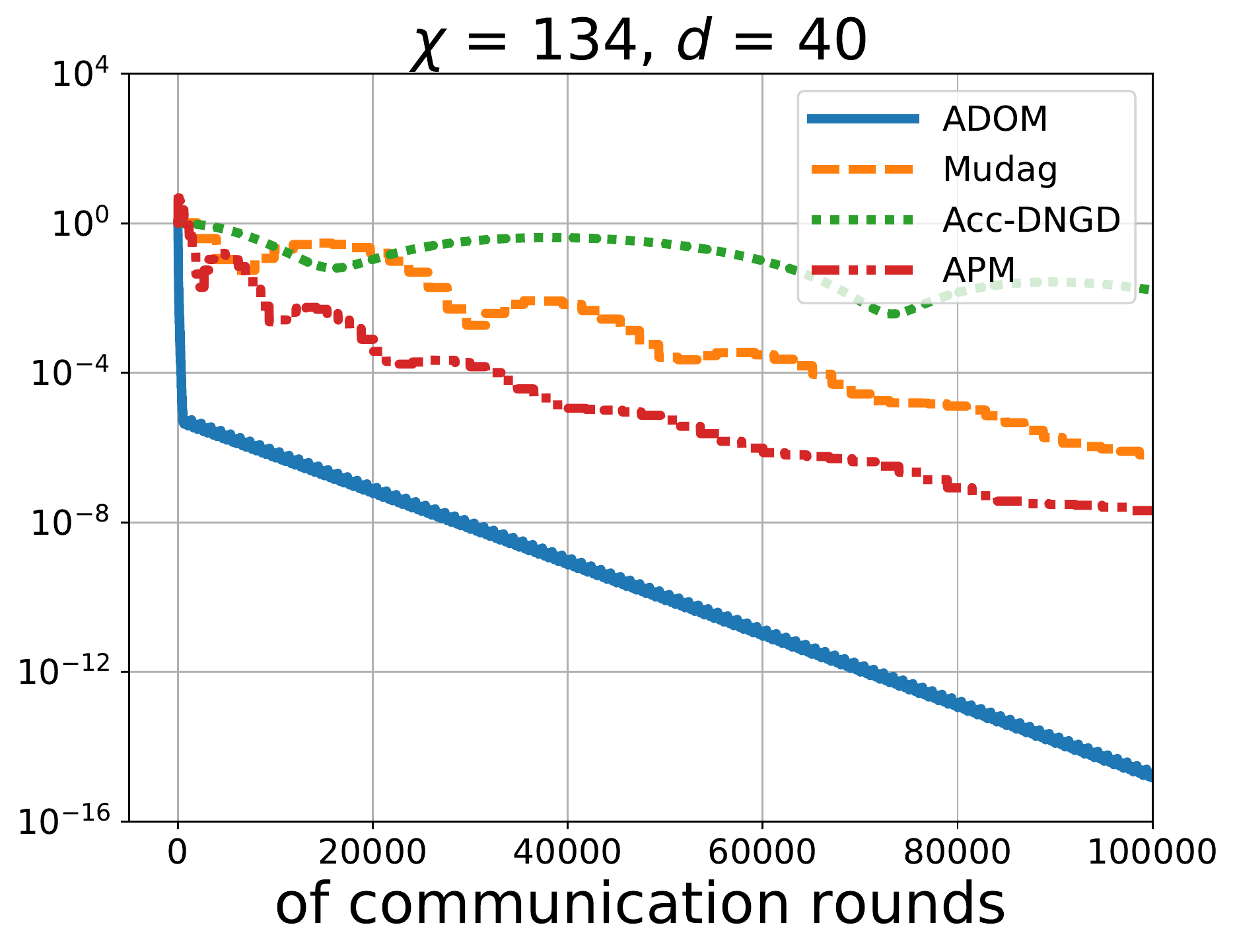}
	\includegraphics[width=0.24\linewidth]{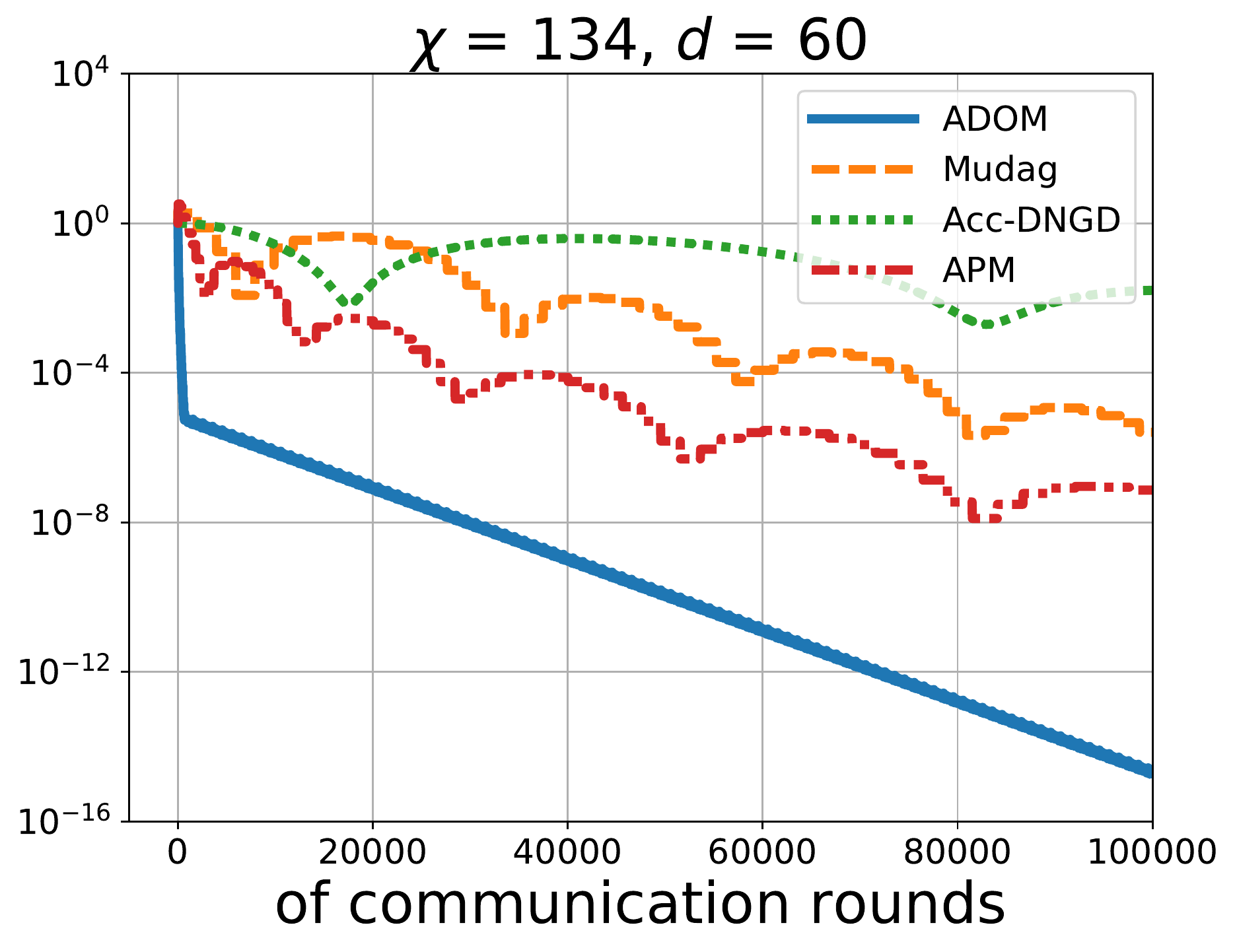}
	\includegraphics[width=0.24\linewidth]{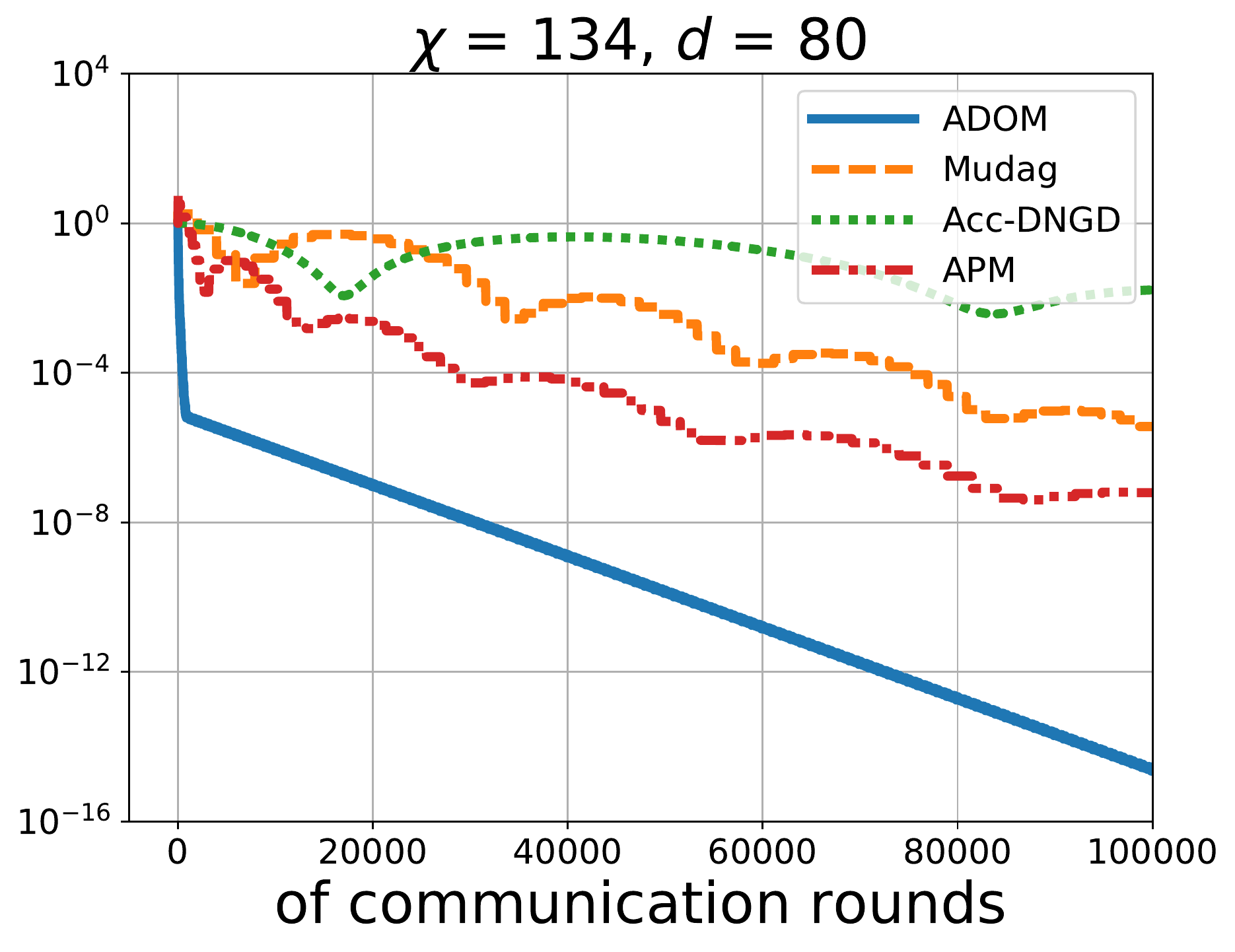}
	\includegraphics[width=0.24\linewidth]{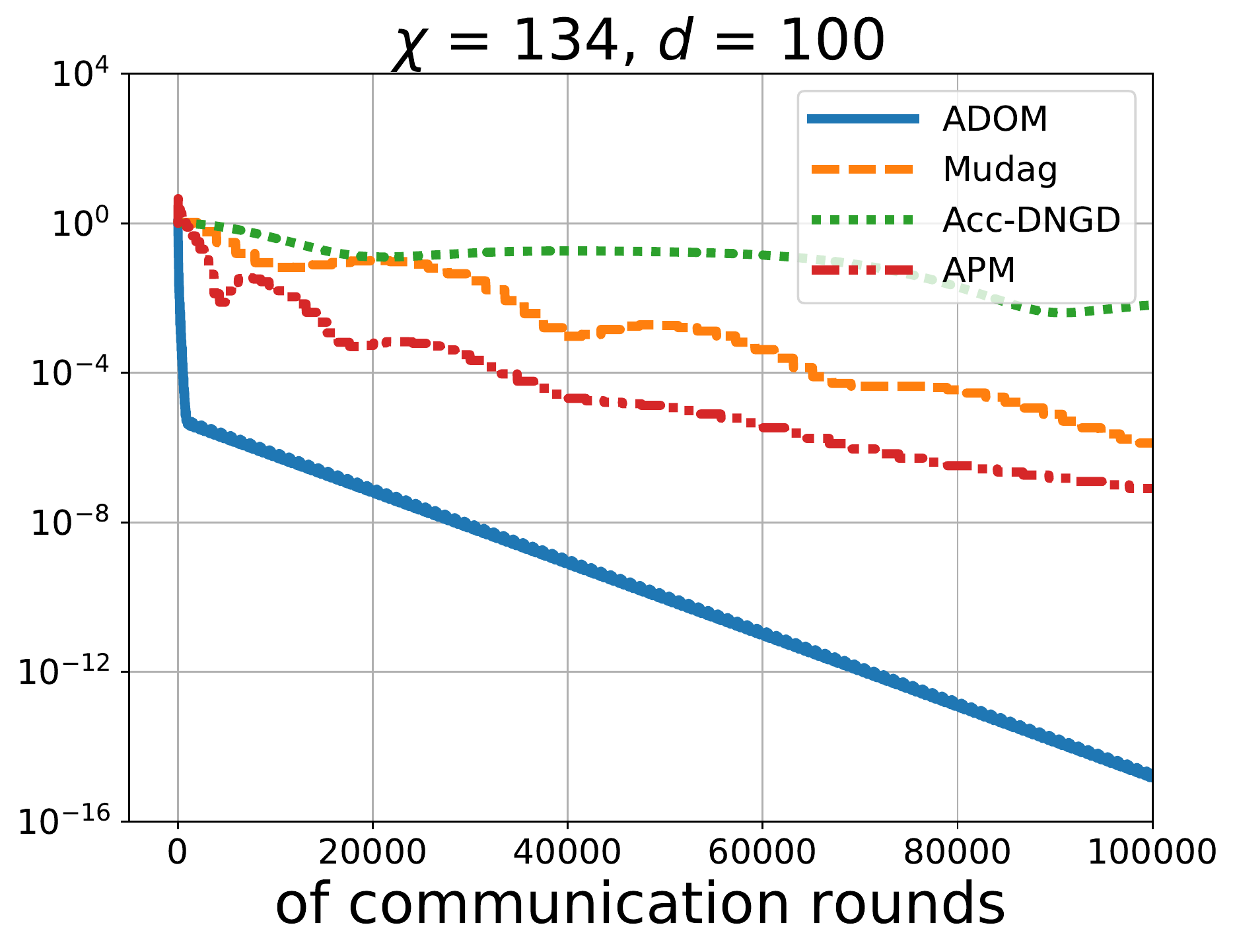}
	
	\includegraphics[width=0.24\linewidth]{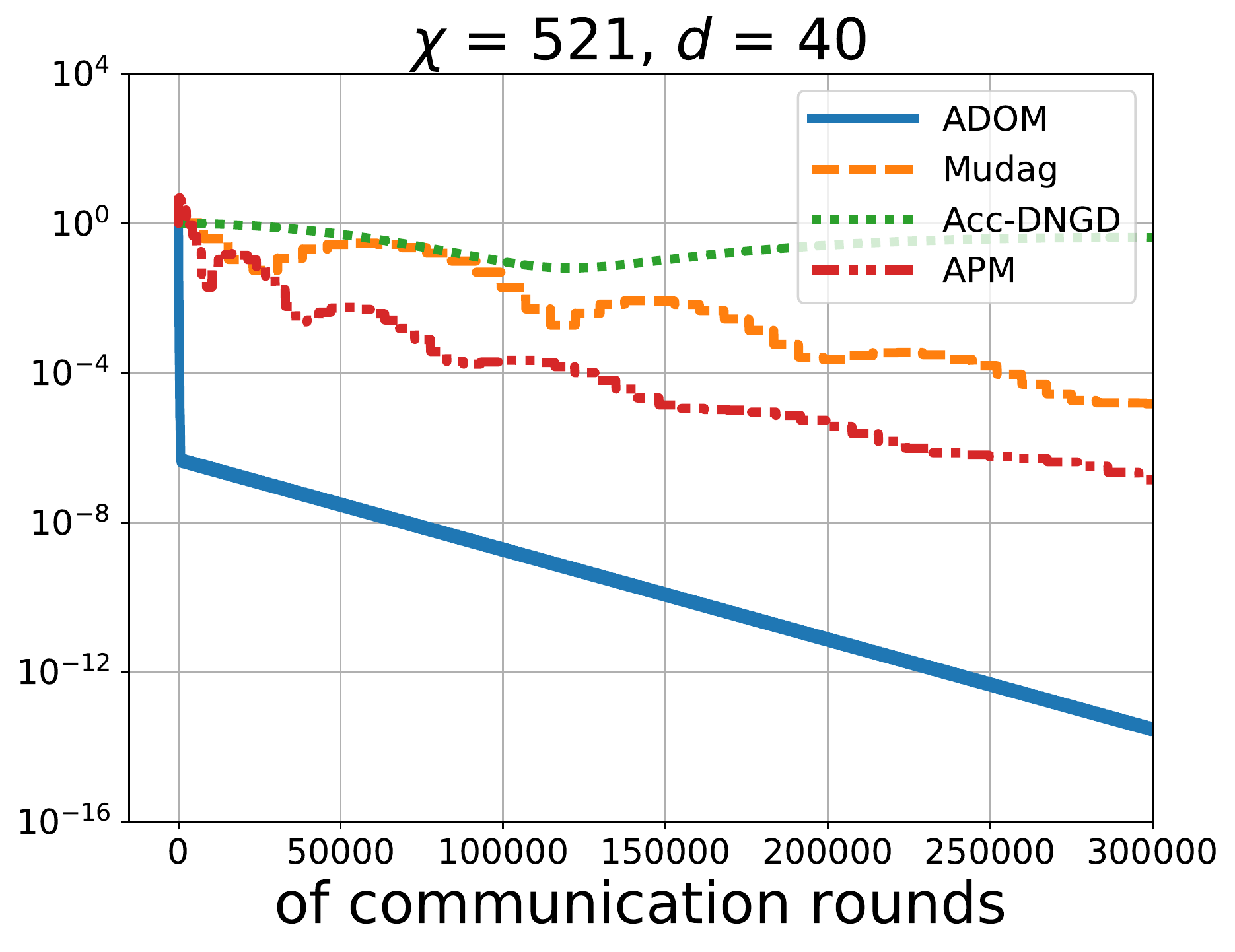}
	\includegraphics[width=0.24\linewidth]{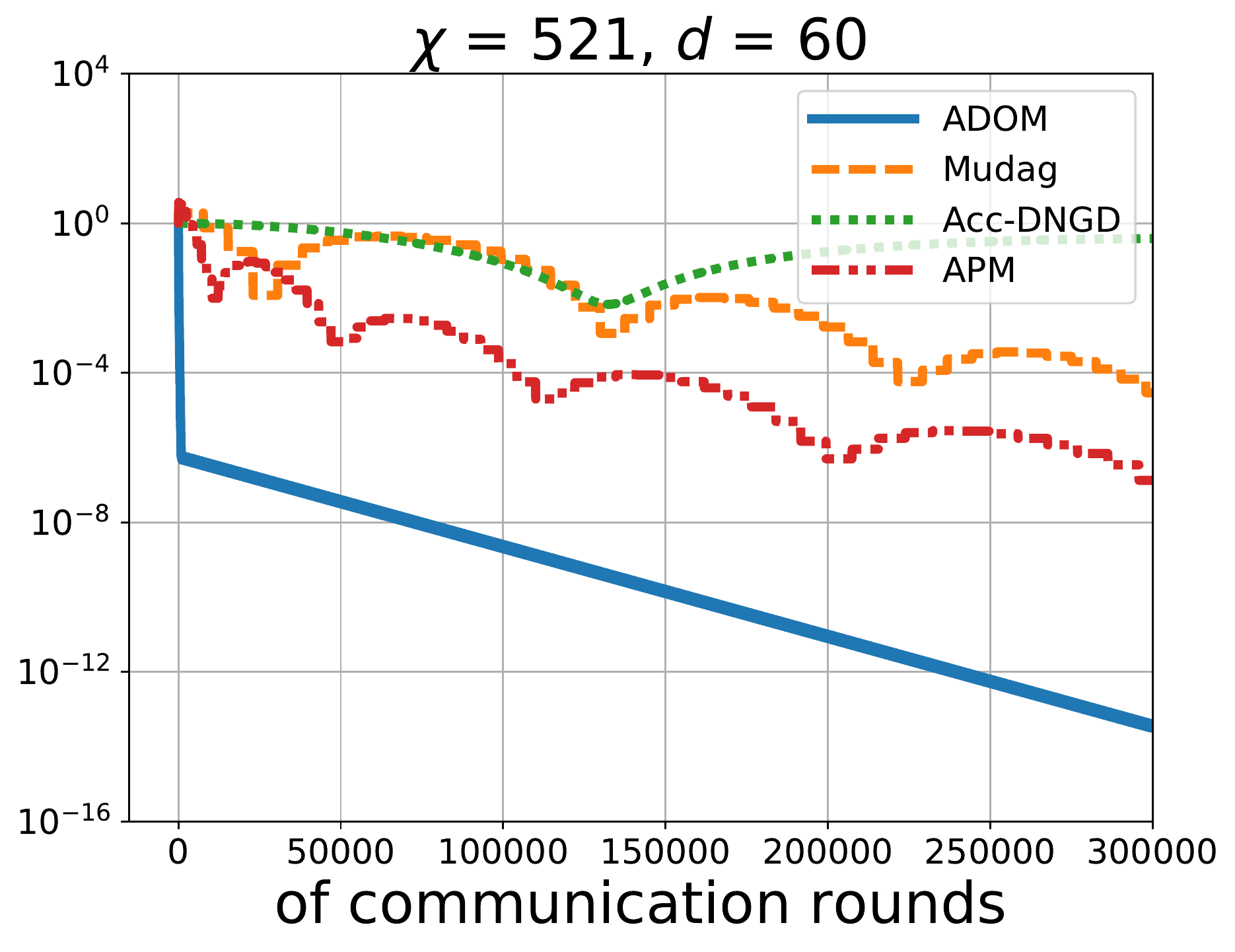}
	\includegraphics[width=0.24\linewidth]{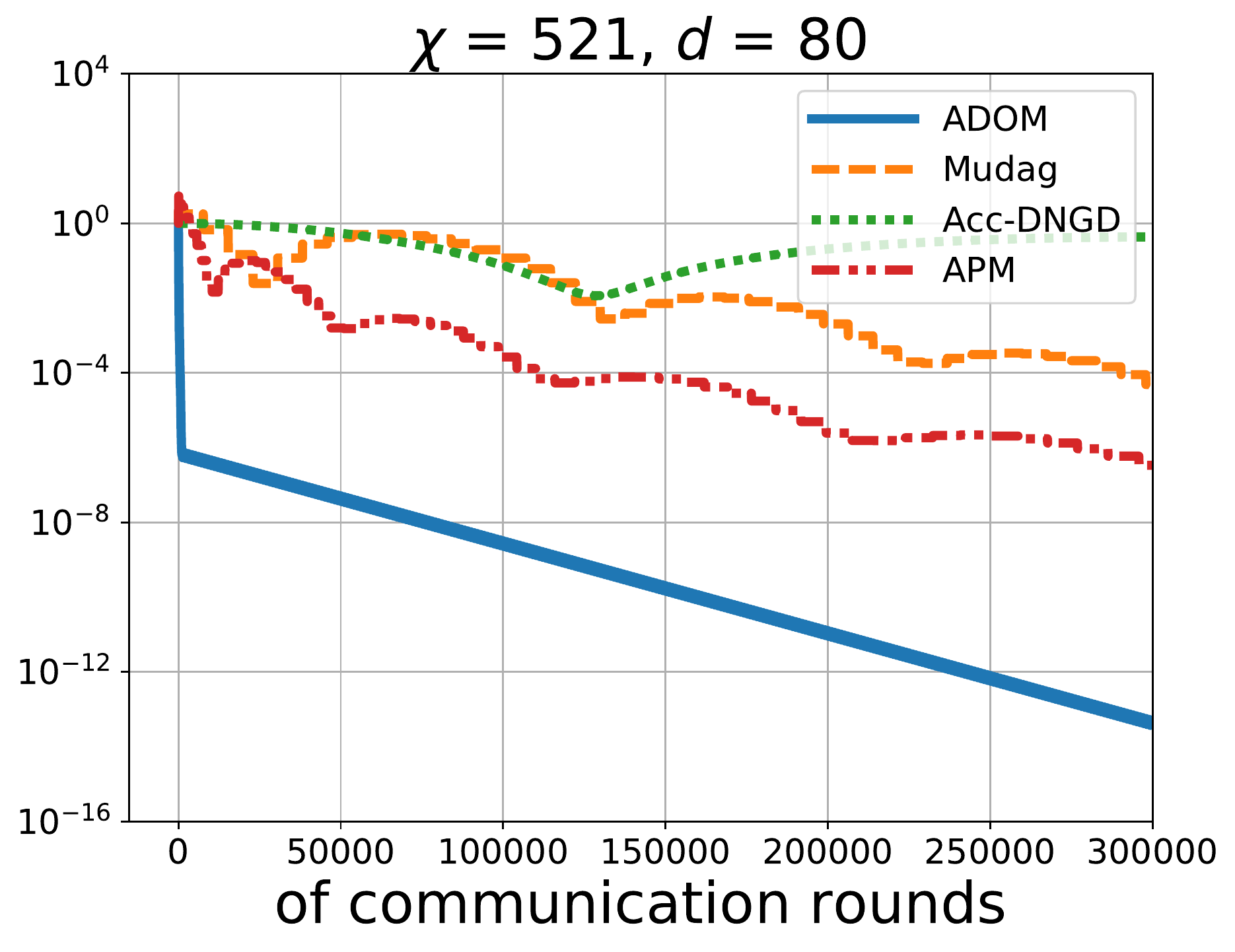}
	\includegraphics[width=0.24\linewidth]{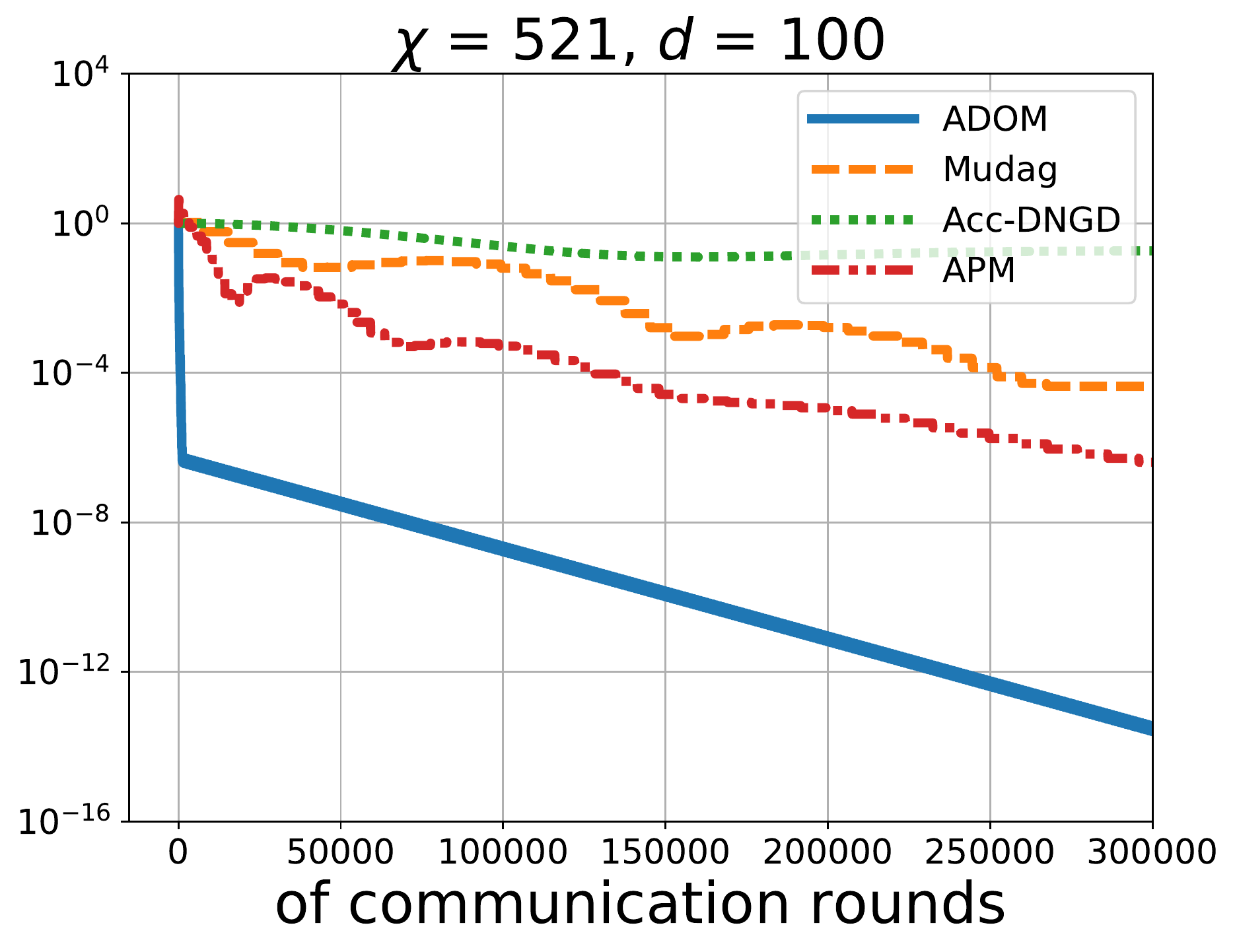}
	
	\caption{Comparison of Mudag, Acc-DNGD, APM and {\sf ADOM} on problems with $\chi \in \{9, 32,134,521\}$, $d \in \{40,60,80,100\}$ and $\kappa = 100$.}
	\label{fig:chi}	
\end{figure*}

In this section we perform experiments with logistic regression for binary classification with $\ell^2$ regularization. That is, our loss function has the form
\begin{equation} \label{eq:log_reg}
\squeeze 
f_i(x) = \frac{1}{m}\sum \limits_{j=1}^m\log (1 + \exp(-b_{ij}a_{ij}^\top x)) + \frac{r}{2}\sqn{x},
\end{equation}
where $a_{ij} \in \R^d$ and $b_{ij}\in \{-1,+1\}$ are data points and labels, $r>0$ is a regularization parameter, and $m$ is the number of data points stored on each node.
In our experiments we use function \texttt{sklearn.datasets.make\_classification} from scikit-learn library for dataset generation. We generate a number of datasets consisting of $10,000$ samples, distributed to the $n=100$ nodes of the network with $m=100$ samples on each node. We vary $r$ to obtain different values of the condition number $\kappa$. We also vary the number of features $d$. 

In order to simulate a time-varying network, we use geometric random graphs. That is, we generate $n=100$ nodes from the uniform distribution over $[0,1]^2 \subset \R^2$ and connect each pair of nodes whose distance is less than a certain {\em radius}. Since a geometric graph is likely to be disconnected when the radius is small, we enforce connectivity by adding a minimal number of edges. We obtain a sequence of networks $\{\cG^k\}_{k=0}^\infty$ by generating $1,000$ random geometric graphs and switching between them in a cyclic way. For each $k$, matrix $\mW^k$ is chosen to be the Laplacian of graph $\cG^k$ divided by its largest eigenvalue. We obtain different values of the time-varying network structure parameter $\chi$ by choosing different values of the radius.

One potential problem with {\sf ADOM} is that it has to calculate the dual gradient $\g F^*(z_g^k)$, which is known to be the solution of the following problem:
\begin{equation}\label{eq:dual_grad}
\g F^*(z_g^k) = \argmin_{x \in (\R^d)^\cV} F(x) - \<x, z_g^k>.
\end{equation}
In practice, $\g F^*(z_g^k)$ may be hard to compute. In our experiments we solve this issue by calculating $\g F^*(z_g^k)$ inexactly using $T$ iterations of Gradient Descent (GD) or Accelerated Gradient Descent (AGD) initialized with the previous estimate of $\g F^*(z_g^{k-1})$. It turns out that it is sufficient to use $T\leq 3$ to obtain a good convergence rate in practice.

\subsection{Experiment 1: While DNM may diverge, {\sf ADOM} is stable}\label{exp1}

We first compare {\sf ADOM} with the Distributed Nesterov Method (DNM) of \citet{rogozin2019optimal}. The condition number $\kappa$ is set to $30$, the number of features is $d=40$. To calculate the dual gradient $\g F^*(z)$ we use $T = 3$ steps of AGD in {\sf ADOM} and $T=30$ steps of AGD in DNM.

We switch between 2 networks  every $t$ iterations, where $t \in \{50,20,10,5\}$. We use the following choice of networks:
(i) two random geometric graphs with $\chi \approx 400$; see Figure~\ref{fig:slow2} (top row);  (ii) two networks with ring and star topology with $\chi \approx 1,000$; see Figure~\ref{fig:slow2} (bottom row).

DNM diverges in 7 out of 8  cases presented in Figure~\ref{fig:slow2}, while {\sf ADOM} converges in all cases. However, when DNM converges, it can converge faster than {\sf ADOM}, since its communication complexity has better dependence on $\chi$ ($\sqrt{\chi}$ of DNM vs $\chi$ of {\sf ADOM}).

\subsection{Experiment 2: Comparing {\sf ADOM} with the state of the art: Mudag,  Acc-DNGD and APM}\label{exp2}

We compare {\sf ADOM} with the following algorithms for decentralized optimization over time-varying networks, all equipped with Nesterov acceleration: Mudag \citep{ye2020multi}, Acc-DNGD \citep{qu2019accelerated} and Accelerated Penalty Method (APM) \citep{li2018sharp,rogozin2020towards}.  
We do not compare {\sf ADOM} with PANDA \citep{maros2018panda} and DIGing \citep{nedic2017achieving} because they are not accelerated and have very slow convergence rate both in theory and practice. We use $T=1$ iterations of GD to calculate $\g F^*(z_g^k)$ in {\sf ADOM}.

We generate random datasets with the number of features $d \in \{40,60,80,100\}$. In Figure~\ref{fig:kappa} we fix the network structure parameter $\chi\approx 30$ and perform comparison for condition number $\kappa \in \{10,10^2,10^3,10^4\}$. In Figure~\ref{fig:chi} we fix $\kappa = 100$ and perform comparison for $\chi \in \{9,32,134,521\}$. 

 Overall, {\sf ADOM} is better than the contenders. Acc-DNGD performs worse as the values of $\chi$ and $\kappa$ grow since it has the worst dependence on them. One can also observe that APM suffers from sub-linear convergence, which becomes clear as the number of iterations grows (see bottom row of Figure~\ref{fig:kappa}) since its  communication complexity is proportional to $\log^2 \frac{1}{\epsilon}$.

\bibliography{reference}
\bibliographystyle{icml2021}

\newpage

\onecolumn
\appendix

\part*{Appendix}

\section{Proof of Lemma~\ref{dual:lem:descent}}

\begin{proof}
	We start with $\frac{1}{\mu}$-smoothness of $F^*$:
	\begin{align*}
	F^*(z_f^{k+1}) \leq F^*(z_g^k) + \<\g F^*(z_g^k), z_f^{k+1} - z_g^k> + \frac{1}{2\mu}\sqn{z_f^{k+1} - z_g^k}.
	\end{align*}
	Using line~\ref{dual:line:z3} of Algorithm~\ref{alg:dual} together with \eqref{eq:spectrum} we get
	\begin{align*}
	F^*(z_f^{k+1}) &\leq F^*(z_g^k) - \theta \sqn{\g F^*(z_g^k)}_{\mW(k)} + \frac{\theta^2}{2\mu}\sqn{\g F^*(z_g^k)}_{\mW^2(k)}
	\\&\leq
	F^*(z_g^k) - \frac{\theta\lminp}{2}\sqn{\g F^*(z_g^k)}_{\mP} - \frac{\theta}{2}\sqn{\g F^*(z_g^k)}_{\mW(k)} + \frac{\theta^2\lmax}{2\mu}\sqn{\g F^*(z_g^k)}_{\mW(k)}
	\\&=
	F^*(z_g^k) - \frac{\theta\lminp}{2}\sqn{\g F^*(z_g^k)}_{\mP} +\frac{\theta}{2}\left(\frac{\theta\lmax}{\mu}- 1\right)\sqn{\g F^*(z_g^k)}_{\mW(k)}.
	\end{align*}
	Using condition $\theta \leq \frac{\mu}{\lmax}$ we get
	\begin{align*}
	F^*(z_f^{k+1})
	&\leq
	F^*(z_g^k) - \frac{\theta\lminp}{2}\sqn{\g F^*(z_g^k)}_{\mP}.
	\end{align*}
\end{proof}

\section{Proof of Lemma~\ref{dual:lem:error}}

\begin{proof}
	Using \eqref{eq:PP} and \eqref{eq:PW} together with lines~\ref{dual:line:delta} and~\ref{dual:line:m} of Algorithm~\ref{alg:dual} we obtain
	\begin{align*}
		\sqn{m^{k+1}}_\mP
		&=
		\sqn{m^{k} - \eta \g F^*(z_g^k) - \Delta^k}_\mP
		\\&=
		\sqn{(\mP -\sigma\mW(k))(m^{k} - \eta \g F^*(z_g^k))}
		\\&=
		\sqn{m^{k} - \eta \g F^*(z_g^k)}_\mP
		-
		2\sigma\sqn{m^{k} - \eta \g F^*(z_g^k)}_{\mW(k)}
		+
		\sigma^2\sqn{m^{k} - \eta \g F^*(z_g^k)}_{\mW^2(k)}.
	\end{align*}
	Using \eqref{eq:spectrum} we obtain
	\begin{align*}
	\sqn{m^{k+1}}_\mP
	&\leq
	\sqn{m^{k} - \eta \g F^*(z_g^k)}_\mP
	-
	\sigma\lminp\sqn{m^{k} - \eta \g F^*(z_g^k)}_{\mP}
	\\&-
	\sigma\sqn{m^{k} - \eta \g F^*(z_g^k)}_{\mW(k)}
	+
	\sigma^2\lmax\sqn{m^{k} - \eta \g F^*(z_g^k)}_{\mW(k)}
	\\&=
	\sqn{m^{k} - \eta \g F^*(z_g^k)}_\mP
	-
	\sigma\lminp\sqn{m^{k} - \eta \g F^*(z_g^k)}_{\mP}
	\\&+
	\sigma(\sigma\lmax - 1)\sqn{m^{k} - \eta \g F^*(z_g^k)}_{\mW(k)}
	\end{align*}
	Using condition $\sigma \leq\frac{1}{\lmax}$ we get
	\begin{align*}
	\sqn{m^{k+1}}_\mP
	&\leq
	(1-\sigma\lminp)\sqn{m^{k} - \eta \g F^*(z_g^k)}_\mP.
	\end{align*}
	Using Young's inequality we get
	\begin{align*}
	\sqn{m^{k+1}}_\mP
	&\leq
	(1-\sigma\lminp)\left(
		\left(1 + \frac{\sigma\lminp}{2(1-\sigma\lminp)}\right)\sqn{m^{k}}_\mP
		+
			\left(1 + \frac{2(1-\sigma\lminp)}{\sigma\lminp}\right)\sqn{\eta \g F^*(z_g^k)}_\mP
	\right)
	\\&=
	\left(1 - \frac{\sigma\lminp}{2}\right)\sqn{m^k}_{\mP}
	+
	\eta^2\frac{(1-\sigma\lminp)(2-\sigma\lminp)}{\sigma\lminp}\sqn{\g F^*(z_g^k)}_\mP
	\\&\leq
	\left(1 - \frac{\sigma\lminp}{2}\right)\sqn{m^k}_{\mP}
	+
	\frac{2\eta^2}{\sigma\lminp}\sqn{\g F^*(z_g^k)}_\mP.
	\end{align*}
	Rearranging concludes the proof.
\end{proof}

\section{New Lemma}

\begin{lemma}\label{dual:lem:main}
	Let
	\begin{equation}\label{dual:alpha}
	\alpha = \frac{1}{2L},
	\end{equation}
	\begin{equation}
	\eta = \frac{2\lminp\sqrt{\mu L}}{7\lmax},\label{dual:eta}
	\end{equation}
	\begin{equation}
		\theta = \frac{\mu}{\lmax},\label{dual:theta}
	\end{equation}
	\begin{equation}
	\sigma = \frac{1}{\lmax},\label{dual:sigma}
	\end{equation}
	\begin{equation}
	\tau = \frac{\lminp}{7\lmax}\sqrt{\frac{\mu}{L}}.\label{dual:tau}
	\end{equation}
Define the Lyapunov function
	\begin{equation}\label{dual:psi}
		\Psi^k \eqdef \sqn{\hat{z}^k - z^*} + \frac{2\eta(1-\eta\alpha)}{\tau}(F^*(z_f^k) - F^*(z^*) )+6\sqn{m^k}_{\mP},
	\end{equation}
	where $\hat{z}^k$ is defined by
	\begin{equation}\label{dual:zhat}
		\hat{z}^k = z^k + \mP m^k.
	\end{equation}
	Then the following inequality holds:
	\begin{equation}\label{dual:eq:recurrence}
		\Psi^{k+1} \leq \left(1-\frac{\lminp}{7\lmax}\sqrt{\frac{\mu}{L}}\right)\Psi^k.
	\end{equation}
\end{lemma}

\begin{proof}
	Using \eqref{dual:zhat} together with lines~\ref{dual:line:m} and~\ref{dual:line:z2} of Algorithm~\ref{alg:dual}, we get
	\begin{align*}
	\hat{z}^{k+1} &= z^{k+1} + \mP n^{k+1}\\
	&= z^k + \eta\alpha(z_g^k - z^k) + \Delta^k + \mP( m^k - \eta\g F^*(z_g^k) - \Delta^k)
	\\&=
	z^k + \mP m^k + \eta\alpha(z_g^k - z^k) - \eta\mP \g F^*(z_g^k) + \Delta^k - \mP\Delta^k.
	\end{align*}
	From line~\ref{dual:line:delta} of Algorithm~\ref{alg:dual} and \eqref{eq:PW} it follows that $\mP\Delta^k = \Delta^k$, which implies
	\begin{align*}
	\hat{z}^{k+1} &= 
	z^k + \mP m^k + \eta\alpha(z_g^k - z^k) - \eta\mP \g F^*(z_g^k)
	\\&=
	\hat{z}^k + \eta\alpha(z_g^k - z^k) - \eta\mP \g F^*(z_g^k).
	\end{align*}
	Hence,
	\begin{align*}
		\sqn{\hat{z}^{k+1} - z^*}
		&=
		\sqn{\hat{z}^k - z^* + \eta\alpha(z_g^k - z^k) - \eta \mP \g F^*(z_g^k)}
		\\&=
		\sqn{(1 - \eta\alpha)(\hat{z}^k - z^* )+ \eta\alpha(z_g^k + \mP m^k - z^*)}
		+
		\eta^2\sqn{\g F^*(z_g^k)}_\mP
		\\&-
		2\eta\<\mP\g F^*(z_g^k), z^k  + \mP m^k- z^* + \eta\alpha(z_g^k - z^k)>
		\\&\leq
		(1-\eta\alpha)\sqn{\hat{z}^k - z^*} + \eta\alpha\sqn{z_g^k + \mP m^k - z^*}
		+
		\eta^2\sqn{\g F^*(z_g^k)}_\mP
		\\&-
		2\eta\<\g F^*(z_g^k),\mP(z_g^k - z^*)>
		+
		2\eta(1-\eta\alpha)\<\g F^*(z_g^k), \mP(z_g^k - z^k)>
		-
		2\eta\<\mP F^*(z_g^k),m^k>
		\\&\leq
		(1-\eta\alpha)\sqn{\hat{z}^k - z^*} + 2\eta\alpha\sqn{z_g^k - z^*}
		+
		2\eta\alpha\sqn{m^k}_\mP
		+
		\eta^2\sqn{\g F^*(z_g^k)}_\mP
		\\&-
		2\eta\<\g F^*(z_g^k),\mP(z_g^k - z^*)>
		+
		2\eta(1-\eta\alpha)\<\g F^*(z_g^k), \mP(z_g^k - z^k)>
		-
		2\eta\<\mP F^*(z_g^k),m^k>
	\end{align*}
	One can observe, that $z^k,z_g^k,z^* \in \cL^\perp$. Hence,
	\begin{align*}
	\sqn{\hat{z}^{k+1} - z^*}
	&\leq
	(1-\eta\alpha)\sqn{\hat{z}^k - z^*} + 2\eta\alpha\sqn{z_g^k - z^*}
	+
	2\eta\alpha\sqn{m^k}_\mP
	+
	\eta^2\sqn{\g F^*(z_g^k)}_\mP
	\\&-
	2\eta\<\g F^*(z_g^k),z_g^k - z^*>
	+
	2\eta(1-\eta\alpha)\<\g F^*(z_g^k), z_g^k - z^k>
	-
	2\eta\<\mP F^*(z_g^k),m^k>.
	\end{align*}
	Using line~\ref{dual:line:z1} of Algorithm~\ref{alg:dual} we get
	\begin{align*}
	\sqn{\hat{z}^{k+1} - z^*}
	&\leq
	(1-\eta\alpha)\sqn{\hat{z}^k - z^*} + 2\eta\alpha\sqn{z_g^k - z^*}
	+
	2\eta\alpha\sqn{m^k}_\mP
	+
	\eta^2\sqn{\g F^*(z_g^k)}_\mP
	\\&-
	2\eta\<\g F^*(z_g^k),z_g^k - z^*>
	+
	2\eta(1-\eta\alpha)\frac{(1-\tau)}{\tau}\<\g F^*(z_g^k), z_f^k - z_g^k>
	-
	2\eta\<\mP F^*(z_g^k),m^k>.
	\end{align*}
	Using convexity and $\frac{1}{L}$-strong convexity of $F^*(z)$ we get
	\begin{align*}
	\sqn{\hat{z}^{k+1} - z^*}
	&\leq
	(1-\eta\alpha)\sqn{\hat{z}^k - z^*} + 2\eta\alpha\sqn{z_g^k - z^*}
	+
	2\eta\alpha\sqn{m^k}_\mP
	+
	\eta^2\sqn{\g F^*(z_g^k)}_\mP
	\\&-
	2\eta(F^*(z_g^k) - F^*(z^*)) - \frac{\eta}{L}\sqn{z_g^k - z^*}
	+
	2\eta(1-\eta\alpha)\frac{(1-\tau)}{\tau}(F^*(z_f^k) - F^*(z_g^k))
	-
	2\eta\<\mP F^*(z_g^k),m^k>
	\\&=
	(1-\eta\alpha)\sqn{\hat{z}^k - z^*}
	+
	\left(2\eta\alpha - \frac{\eta}{L}\right)\sqn{z_g^k - z^*}
	+
	\eta^2\sqn{\g F^*(z_g^k)}_\mP
	\\&-
	2\eta(F^*(z_g^k) - F^*(z^*))
	+
	2\eta(1-\eta\alpha)\frac{(1-\tau)}{\tau}(F^*(z_f^k) - F^*(z_g^k))
	-
	2\eta\<\mP F^*(z_g^k),m^k>
	+
	2\eta\alpha\sqn{m^k}_\mP.
	\end{align*}
	Using $\alpha$ defined by \eqref{dual:alpha} we get
	\begin{align*}
	\sqn{\hat{z}^{k+1} - z^*}
	&\leq
	\left(1-\frac{\eta}{2L}\right)\sqn{\hat{z}^k - z^*}
	+
	\eta^2\sqn{\g F^*(z_g^k)}_\mP
	\\&-
	2\eta(F^*(z_g^k) - F^*(z^*))
	+
	2\eta(1-\eta\alpha)\frac{(1-\tau)}{\tau}(F^*(z_f^k) - F^*(z_g^k))
	-
	2\eta\<\mP F^*(z_g^k),m^k>
	+
	2\eta\alpha\sqn{m^k}_\mP.
	\end{align*}
	Since $F^*(z_g^k) \geq F^*(z^*)$, we get
	\begin{align*}
	\sqn{\hat{z}^{k+1} - z^*}
	&\leq
	\left(1-\frac{\eta}{2L}\right)\sqn{\hat{z}^k - z^*}
	+
	\eta^2\sqn{\g F^*(z_g^k)}_\mP
	\\&-
	2\eta(1-\eta\alpha)(F^*(z_g^k) - F^*(z^*))
	+
	2\eta(1-\eta\alpha)\frac{(1-\tau)}{\tau}(F^*(z_f^k) - F^*(z_g^k))
	\\&-
	2\eta\<\mP F^*(z_g^k),m^k>
	+
	2\eta\alpha\sqn{m^k}_\mP
	\\&=
	\left(1-\frac{\eta}{2L}\right)\sqn{\hat{z}^k - z^*}
	+
	\eta^2\sqn{\g F^*(z_g^k)}_\mP
	\\&+
	2\eta(1-\eta\alpha)\left(
	\frac{(1-\tau)}{\tau}F^*(z_f^k) + F^*(z^*) - \frac{1}{\tau}F^*(z_g^k)
	\right)
	\\&-
	2\eta\<\mP F^*(z_g^k),m^k>
	+
	2\eta\alpha\sqn{m^k}_\mP.
	\end{align*}
	Using \eqref{dual:eq:1} and $\theta$ defined by \eqref{dual:theta} we get
	\begin{align*}
	\sqn{\hat{z}^{k+1} - z^*}
	&\leq
	\left(1-\frac{\eta}{2L}\right)\sqn{\hat{z}^k - z^*}
	+
	\left(\eta^2-\frac{(1-\eta\alpha)\eta\mu\lminp}{\tau\lmax}\right)\sqn{\g F^*(z_g^k)}_\mP
	\\&+
	(1-\tau)\frac{2\eta(1-\eta\alpha)}{\tau}(F^*(z_f^k) - F^*(z^*) )
	-
	\frac{2\eta(1-\eta\alpha)}{\tau}(F^*(z_f^{k+1}) - F^*(z^*) )
	\\&-
	2\eta\<\mP F^*(z_g^k),m^k>
	+
	2\eta\alpha\sqn{m^k}_\mP.
	\end{align*}
	Using Young's inequality we get
	\begin{align*}
	\sqn{\hat{z}^{k+1} - z^*}
	&\leq
	\left(1-\frac{\eta}{2L}\right)\sqn{\hat{z}^k - z^*}
	+
	\left(\eta^2-\frac{(1-\eta\alpha)\eta\mu\lminp}{\tau\lmax}\right)\sqn{\g F^*(z_g^k)}_\mP
	\\&+
	(1-\tau)\frac{2\eta(1-\eta\alpha)}{\tau}(F^*(z_f^k) - F^*(z^*) )
	-
	\frac{2\eta(1-\eta\alpha)}{\tau}(F^*(z_f^{k+1}) - F^*(z^*) )
	\\&+
	\frac{\eta^2\lmax}{\lminp}\sqn{\g F^*(z_g^k)}_\mP + \frac{\lminp}{\lmax}\sqn{m^k}_\mP
	+
	2\eta\alpha\sqn{m^k}_\mP
	\\&=
	\left(1-\frac{\eta}{2L}\right)\sqn{\hat{z}^k - z^*}
	+
	\left(\eta^2+\frac{\eta^2\lmax}{\lminp}-\frac{(1-\eta\alpha)\eta\mu\lminp}{\tau\lmax}\right)\sqn{\g F^*(z_g^k)}_\mP
	\\&+
	(1-\tau)\frac{2\eta(1-\eta\alpha)}{\tau}(F^*(z_f^k) - F^*(z^*) )
	-
	\frac{2\eta(1-\eta\alpha)}{\tau}(F^*(z_f^{k+1}) - F^*(z^*) )
	\\&+
	\left(\frac{\lminp}{\lmax}+2\eta\alpha\right)\sqn{m^k}_\mP.
	\end{align*}
	Using \eqref{dual:eta} and \eqref{dual:alpha}, that imply $\eta\alpha\leq \frac{\lminp}{4\lmax}$, we obtain
	\begin{align*}
	\sqn{\hat{z}^{k+1} - z^*}
	&\leq
	\left(1-\frac{\eta}{2L}\right)\sqn{\hat{z}^k - z^*}
	+
	\left(\eta^2+\frac{\eta^2\lmax}{\lminp}-\frac{3\eta\mu\lminp}{4\tau\lmax}\right)\sqn{\g F^*(z_g^k)}_\mP
	\\&+
	(1-\tau)\frac{2\eta(1-\eta\alpha)}{\tau}(F^*(z_f^k) - F^*(z^*) )
	-
	\frac{2\eta(1-\eta\alpha)}{\tau}(F^*(z_f^{k+1}) - F^*(z^*) )
	+
	\frac{3\lminp}{2\lmax}\sqn{m^k}_\mP.
	\end{align*}
	Using \eqref{dual:eq:2} and $\sigma$ defined by \eqref{dual:sigma} we get
	\begin{align*}
	\sqn{\hat{z}^{k+1} - z^*}
	&\leq
	\left(1-\frac{\eta}{2L}\right)\sqn{\hat{z}^k - z^*}
	+
	\left(\eta^2+\frac{\eta^2\lmax}{\lminp}-\frac{3\eta\mu\lminp}{4\tau\lmax}\right)\sqn{\g F^*(z_g^k)}_\mP
	\\&+
	(1-\tau)\frac{2\eta(1-\eta\alpha)}{\tau}(F^*(z_f^k) - F^*(z^*) )
	-
	\frac{2\eta(1-\eta\alpha)}{\tau}(F^*(z_f^{k+1}) - F^*(z^*) )
	\\&+
	\left(1 - \frac{\lminp}{4\lmax}\right)6\sqn{m^k}_{\mP}
	-
	6\sqn{m^{k+1}}_\mP
	+
	\frac{12\eta^2\lmax}{\lminp}\sqn{\g F^*(z_g^k)}_\mP
	\\&\leq
	\left(1-\frac{\eta}{2L}\right)\sqn{\hat{z}^k - z^*}
	+
	\left(\frac{14\eta^2\lmax}{\lminp}-\frac{3\eta\mu\lminp}{4\tau\lmax}\right)\sqn{\g F^*(z_g^k)}_\mP
	\\&+
	(1-\tau)\frac{2\eta(1-\eta\alpha)}{\tau}(F^*(z_f^k) - F^*(z^*) )
	-
	\frac{2\eta(1-\eta\alpha)}{\tau}(F^*(z_f^{k+1}) - F^*(z^*) )
	\\&+
	\left(1 - \frac{\lminp}{4\lmax}\right)6\sqn{m^k}_{\mP}
	-
	6\sqn{m^{k+1}}_\mP.
	\end{align*}
	Using $\eta$ defined by \eqref{dual:eta} and $\tau$ defined by \eqref{dual:tau} we get
	\begin{align*}
	\sqn{\hat{z}^{k+1} - z^*}
	&\leq
	\left(1-\frac{\lminp}{7\lmax}\sqrt{\frac{\mu}{L}}\right)\sqn{\hat{z}^k - z^*}
	+
	\left(1 - \frac{\lminp}{4\lmax}\right)6\sqn{m^k}_{\mP}
	-
	6\sqn{m^{k+1}}_\mP
	\\&+
	\left(1-\frac{\lminp}{7\lmax}\sqrt{\frac{\mu}{L}}\right)\frac{2\eta(1-\eta\alpha)}{\tau}(F^*(z_f^k) - F^*(z^*) )
	-
	\frac{2\eta(1-\eta\alpha)}{\tau}(F^*(z_f^{k+1}) - F^*(z^*) )
	\\&\leq
	\left(1-\frac{\lminp}{7\lmax}\sqrt{\frac{\mu}{L}}\right)
	\left(\sqn{\hat{z}^k - z^*} + \frac{2\eta(1-\eta\alpha)}{\tau}(F^*(z_f^k) - F^*(z^*) )+6\sqn{m^k}_{\mP}\right)
	\\&-
	\frac{2\eta(1-\eta\alpha)}{\tau}(F^*(z_f^{k+1}) - F^*(z^*) )
	-
	6\sqn{m^{k+1}}_\mP.
	\end{align*}
	Rearranging and using \eqref{dual:psi} concludes the proof.
\end{proof}

\section{Proof of Theorem~\ref{thm:main}}
\begin{proof}
	Using $\frac{1}{\mu}$-smoothness of $F^*$ and the fact that $\g F^*(z^*) = x^*$, we get:
	\begin{align*}
		\sqn{\g F^*(z_g^k) - x^*}
		&=
		\sqn{\g F^*(z_g^k) - \g F^*(z^*)}
		\leq
		\frac{1}{\mu^2}\sqn{z_g^k - z^*}.
	\end{align*}
	Using line~\ref{dual:line:z1} of Algorithm~\ref{alg:dual} we get
	\begin{align*}
		\sqn{\g F^*(z_g^k) - x^*}
		&\leq\frac{\tau}{\mu^2}\sqn{z^k - z^*} + \frac{(1-\tau)}{\mu^2}\sqn{z_f^k - z^*}.
	\end{align*}
	Using $\frac{1}{L}$-strong convexity of $F^*$ we get
	\begin{align*}
	\sqn{\g F^*(z_g^k) - x^*}
	&\leq\frac{\tau}{\mu^2}\sqn{z^k - z^*} + \frac{2(1-\tau)L}{\mu^2}(F^*(z_f^k) - F^*(z^*)).
	\end{align*}
	Using \eqref{dual:zhat} we get
	\begin{align*}
			\sqn{\g F^*(z_g^k) - x^*}
			&\leq
			\frac{2\tau}{\mu^2}\sqn{\hat{z}^k - z^*}
			+
			\frac{2\tau}{\mu^2}\sqn{m^k}_\mP
			+
			\frac{2(1-\tau)L}{\mu^2}(F^*(z_f^k) - F^*(z^*))
			\\&=
			\frac{2\tau}{\mu^2}\sqn{\hat{z}^k - z^*}
			+
			\frac{\tau(1-\tau)L}{\eta(1-\eta\alpha)\mu^2}\frac{2\eta(1-\eta\alpha)}{\tau}(F^*(z_f^k) - F^*(z^*))
			+
			\frac{\tau}{3\mu^2}6\sqn{m^k}_\mP.
			\\&\leq
			\max\left\{\frac{2\tau}{\mu^2},\frac{\tau(1-\tau)L}{\eta(1-\eta\alpha)\mu^2}, \frac{\tau}{3\mu^2}\right\}\left(\sqn{\hat{z}^k - z^*} + \frac{2\eta(1-\eta\alpha)}{\tau}(F^*(z_f^k) - F^*(z^*) )+6\sqn{m^k}_{\mP}\right)
			\\&=
			\max\left\{ \frac{2\tau}{  \mu^2},\frac{\tau(1-\tau)L}{\eta(1-\eta\alpha)\mu^2}\right\}\left(\sqn{\hat{z}^k - z^*} + \frac{2\eta(1-\eta\alpha)}{\tau}(F^*(z_f^k) - F^*(z^*) )+6\sqn{m^k}_{\mP}\right).
	\end{align*}
	Using the definition of $\Psi^k$ \eqref{dual:psi} and denoting $C = \Psi^0	\max\left\{ \frac{2\tau}{\mu^2},\frac{\tau(1-\tau)L}{\eta(1-\eta\alpha)\mu^2}\right\}$ we get
	\begin{align*}
		\sqn{\g F^*(z_g^k) - x^*} \leq \frac{C}{\Psi^0}\Psi^k.
	\end{align*}
	One can observe, that conditions of Lemma~\ref{dual:lem:main} are satisfied. Hence, inequality \eqref{dual:eq:recurrence} holds, which implies
	\begin{eqnarray*}
	\sqn{\g F^*(z_g^k) - x^*} &\leq & \frac{C}{\Psi^0} \left(1-\frac{\lminp}{7\lmax}\sqrt{\frac{\mu}{L}}\right)\Psi^{k-1}
	\\&\leq & \frac{C}{\Psi^0}\left(1-\frac{\lminp}{7\lmax}\sqrt{\frac{\mu}{L}}\right)^2\Psi^{k-2}
	\\&  \vdots &
	\\&\leq & \frac{C}{\Psi^0}\left(1-\frac{\lminp}{7\lmax}\sqrt{\frac{\mu}{L}}\right)^k\Psi^{0}
	\\&=&
	C\left(1-\frac{\lminp}{7\lmax}\sqrt{\frac{\mu}{L}}\right)^k,
	\end{eqnarray*}
	which concludes the proof.
\end{proof}

\section{Additional Experiments}

\subsection{Real data}\label{real_data}

In this section, we perform experiments for the same problem~\eqref{eq:log_reg} and network setup as in Section~\ref{experiments} (and~\ref{exp2}), but with LIBSVM\footnote{The LIBSVM~\citep{chang2011libsvm} dataset collection is available at \href{https://www.csie.ntu.edu.tw/~cjlin/libsvmtools/datasets/}{https://www.csie.ntu.edu.tw/~cjlin/libsvmtools/datasets/}} datasets: \emph{a6a, w6a, ijcnn1} instead of the synthetic ones (see Table~\ref{tbl:87gf9dg8d}). In Figure~\ref{fig:real_kappa}, the network structure parameter is fixed ($\chi\approx 30$), and condition number $\kappa \in \{10,10^2,10^3,10^4\}$ changes. In Figure~\ref{fig:real_chi}, we fix $\kappa = 100$ and perform comparison for $\chi \in \{9, 32, 134, 521\}$. 

\begin{table}[H]
\centering
\begin{tabular}{|c|c|c|}
\hline
dataset & samples & dimension \\ \hline
a6a     & 11220   & 122       \\ \hline
w6a     & 17188   & 300       \\ \hline
ijcnn1  & 49990   & 22        \\ \hline
\end{tabular}

\caption{Details of the datasets}
\label{tbl:87gf9dg8d}
\end{table}

To summarize the obtained results, {\sf ADOM} outperforms all other methods for every set of parameters. This becomes even more evident on real data. One can also observe that for some cases, Acc-DNGD almost does not converge. Apart from that, competing methods (such as APM and Mudag) often show divergence during the first iterations, while {\sf ADOM} consistently demonstrates significant progress during the initial phase. Besides, it is enough to use one iteration ($T=1$) of GD to calculate $\g F^*(z_g^k)$ in {\sf ADOM} to ensure liner convergence.

\begin{figure*}[!htb]
	\centering
	\includegraphics[width=0.24\linewidth]{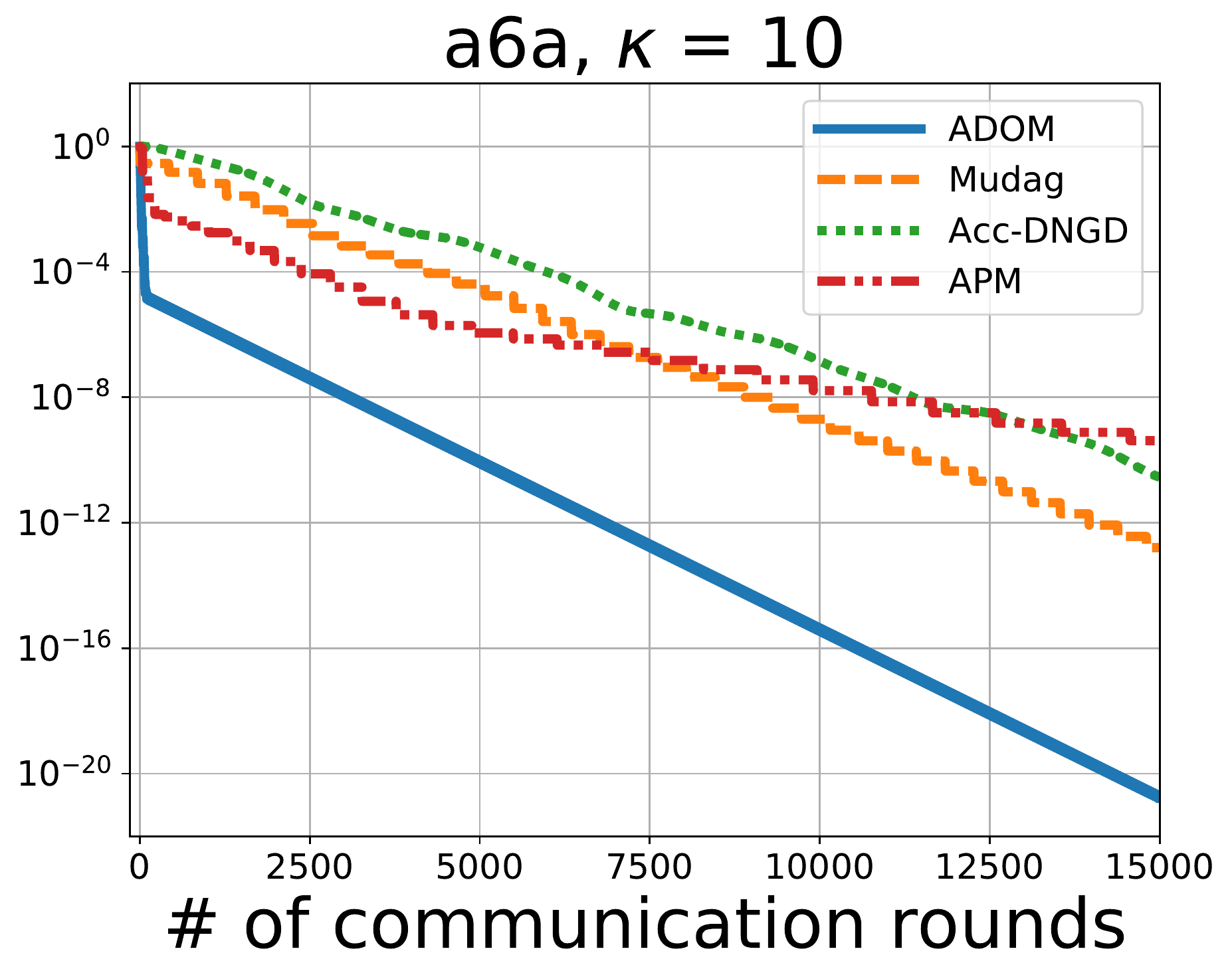}
	\includegraphics[width=0.24\linewidth]{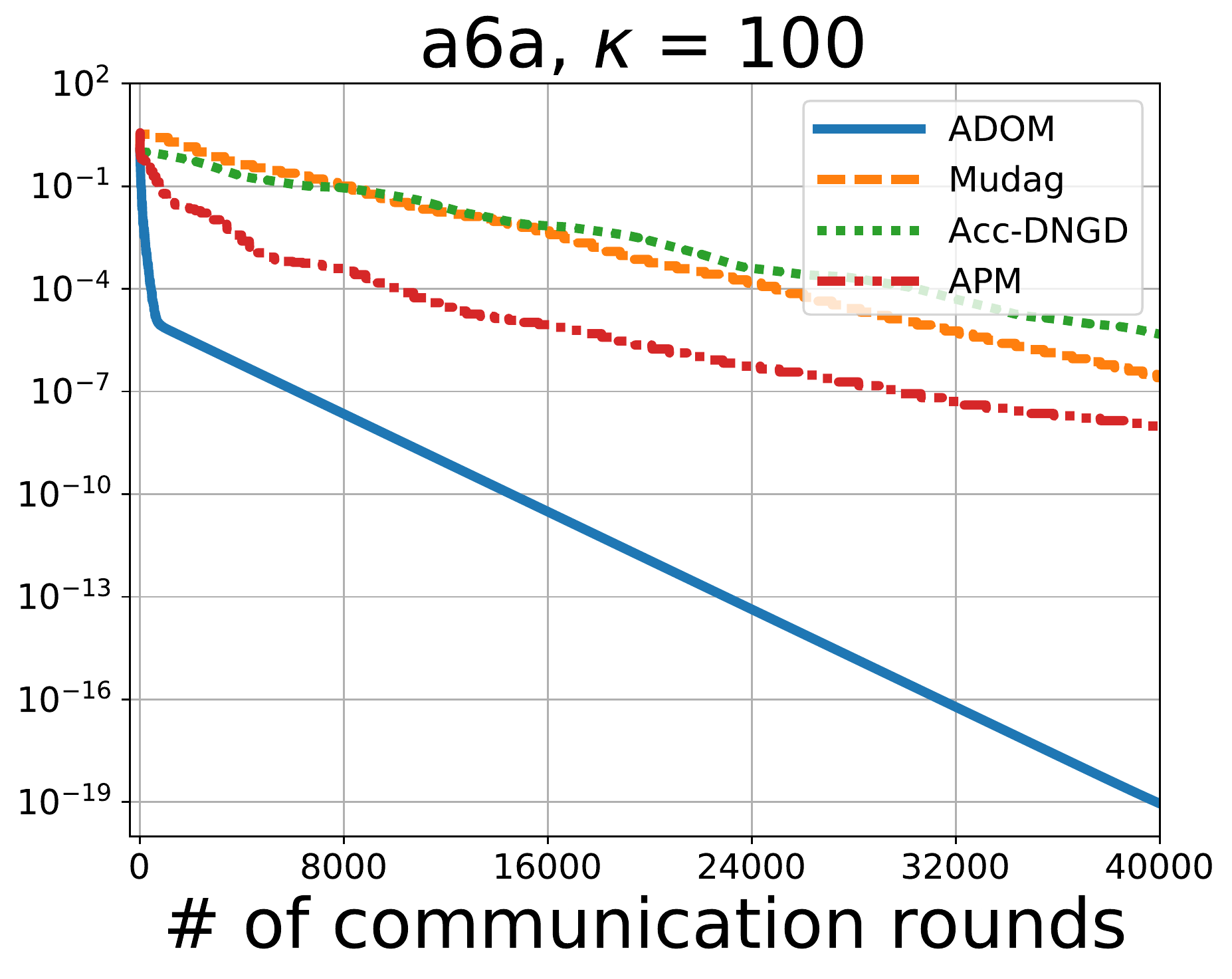}
	\includegraphics[width=0.24\linewidth]{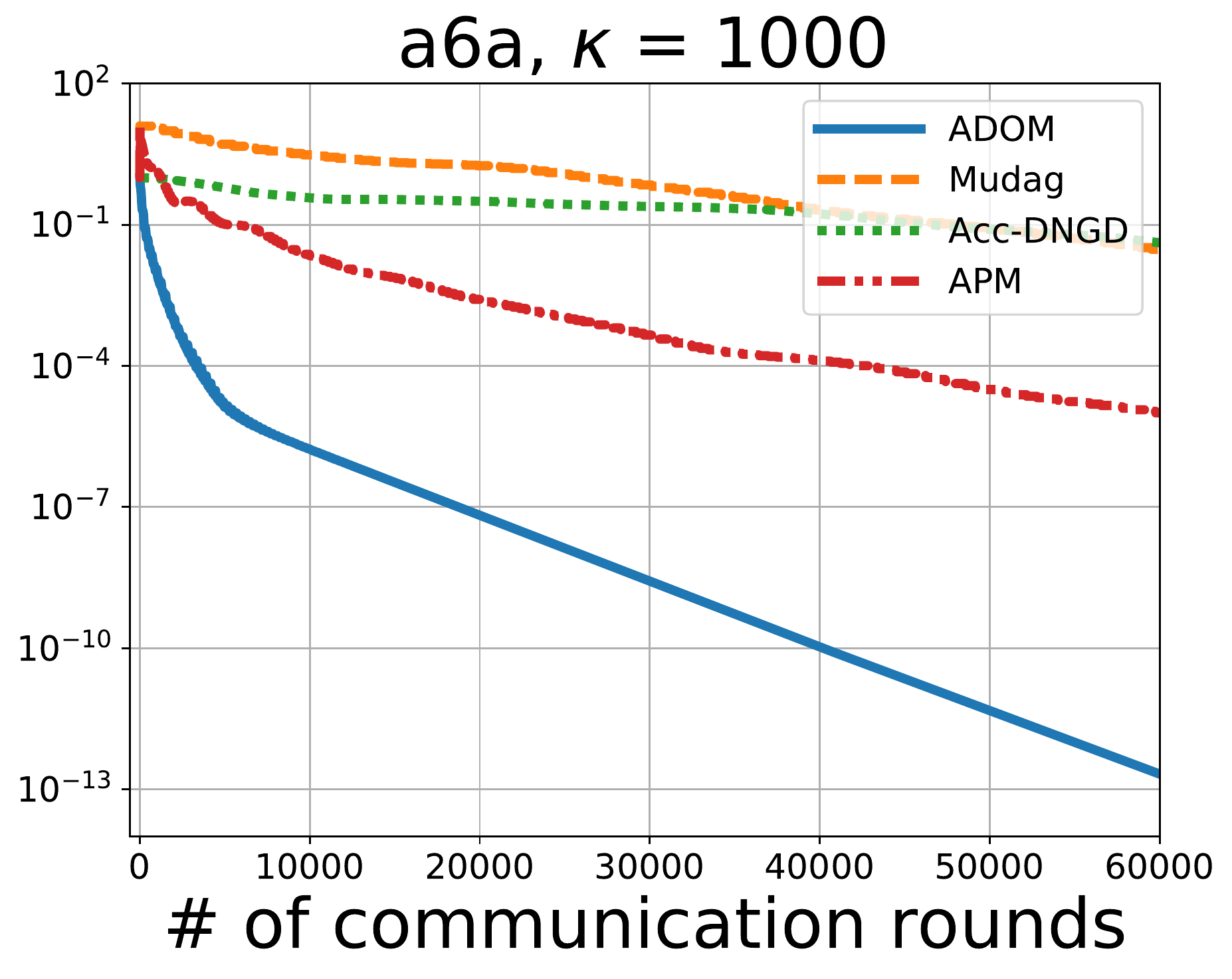}
	\includegraphics[width=0.24\linewidth]{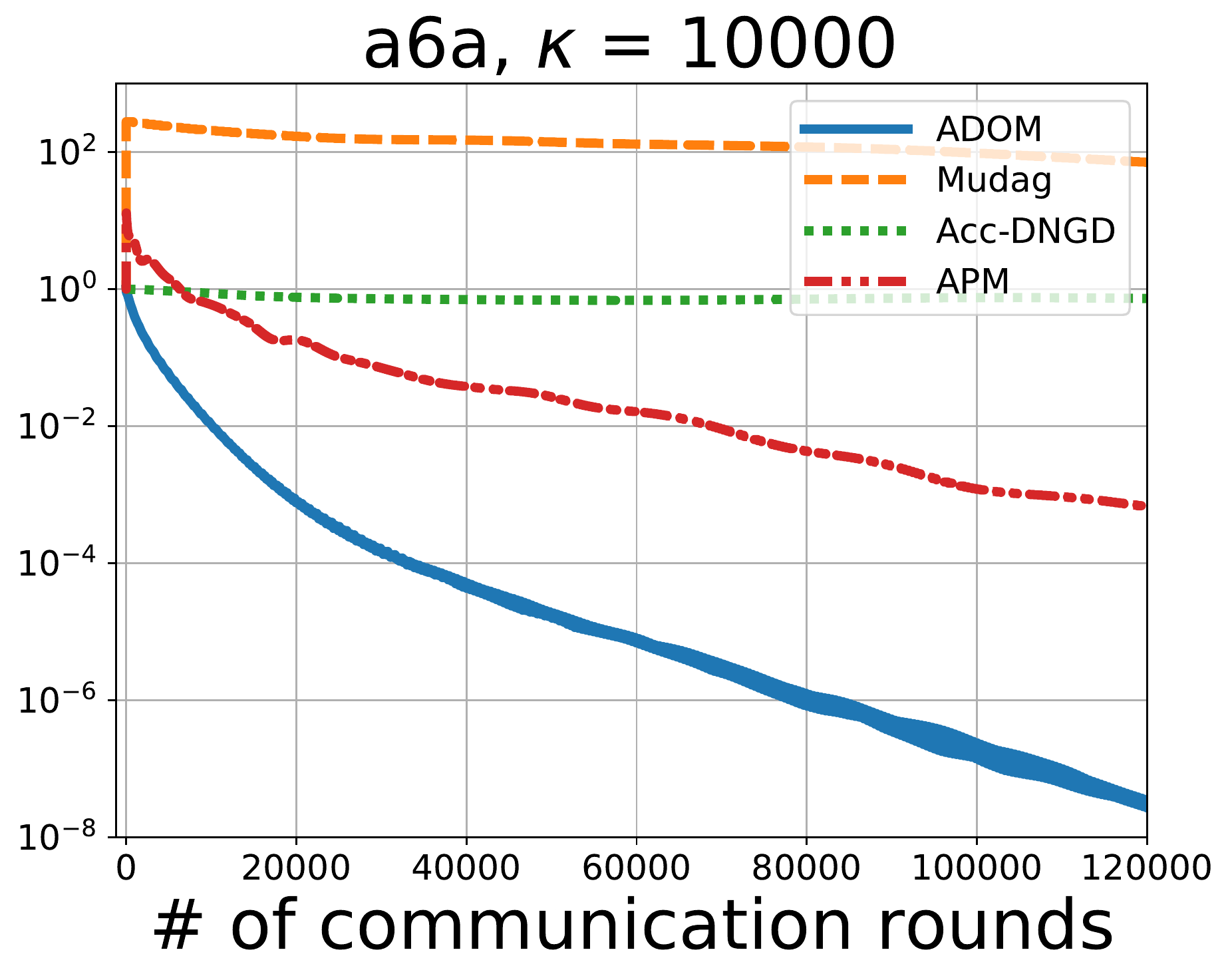}
	
	\includegraphics[width=0.24\linewidth]{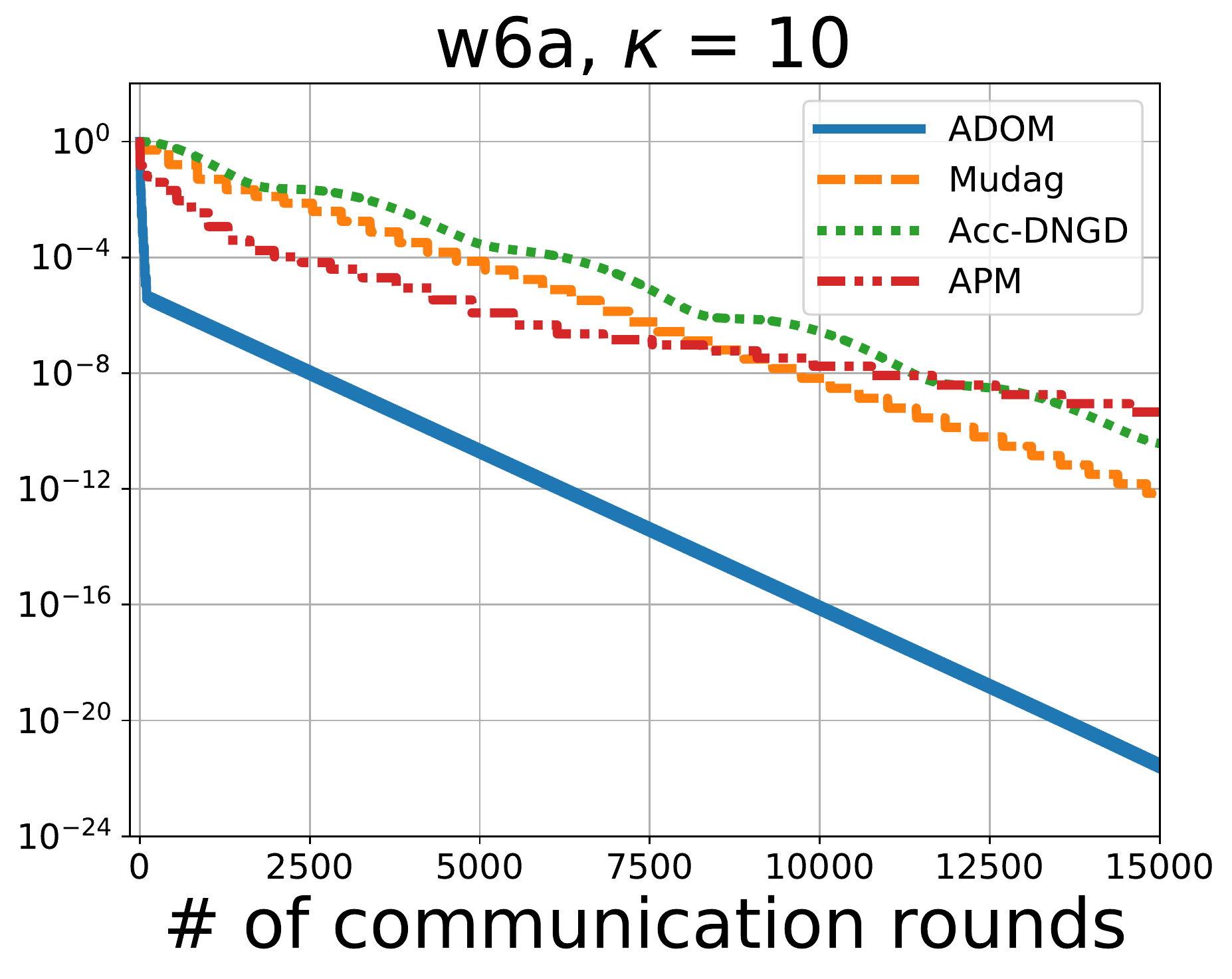}
	\includegraphics[width=0.24\linewidth]{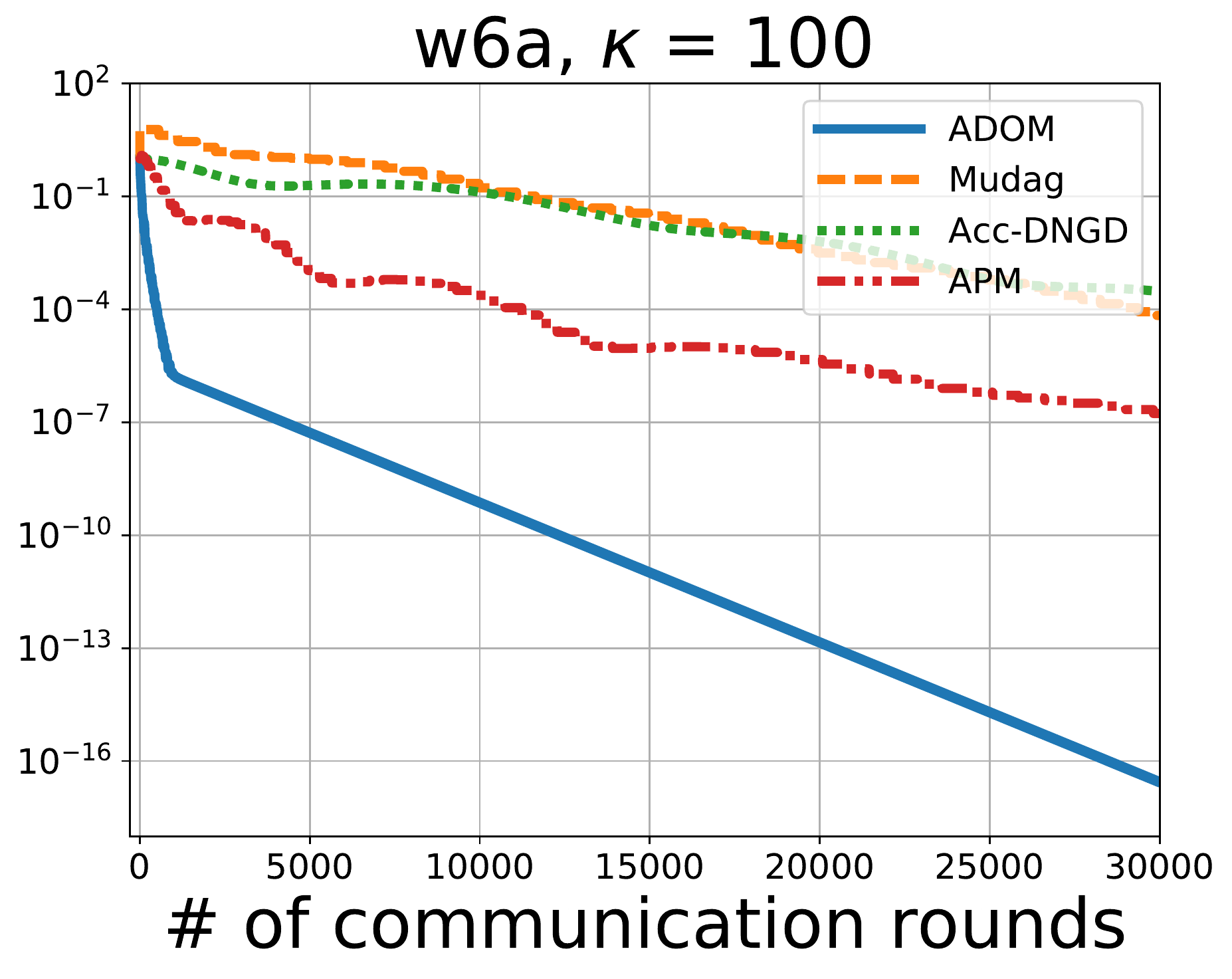}
	\includegraphics[width=0.24\linewidth]{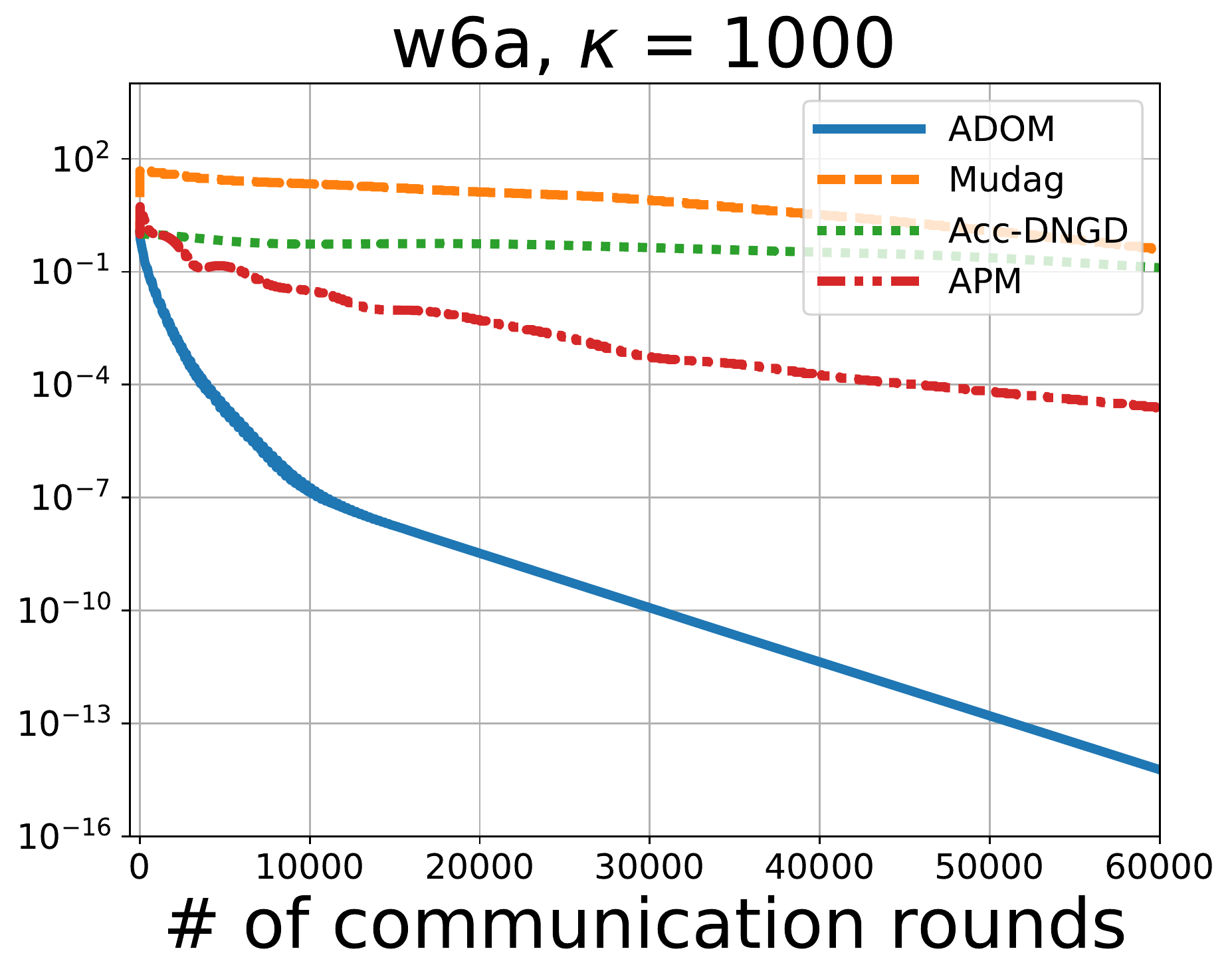}
	\includegraphics[width=0.24\linewidth]{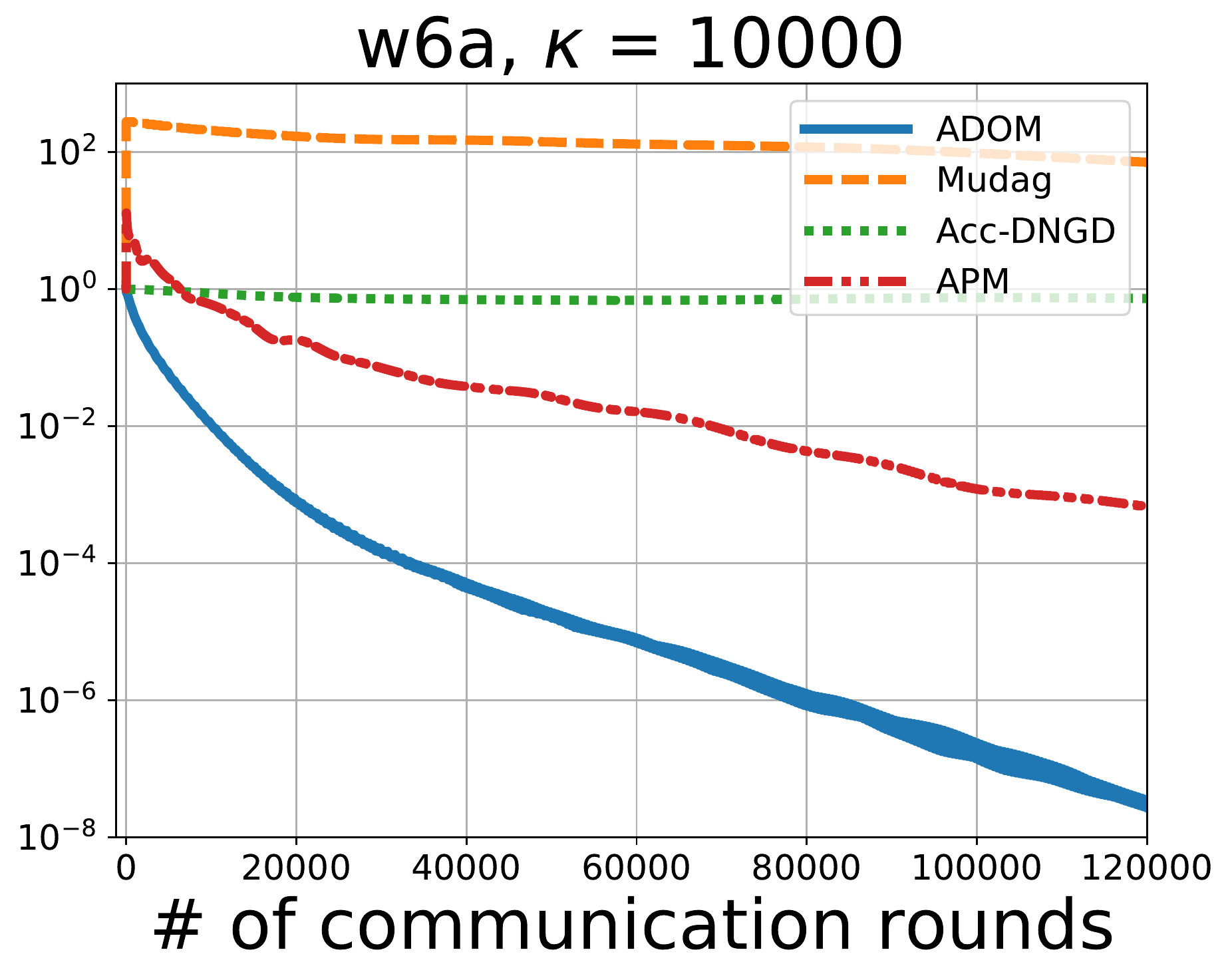}
	
	\includegraphics[width=0.24\linewidth]{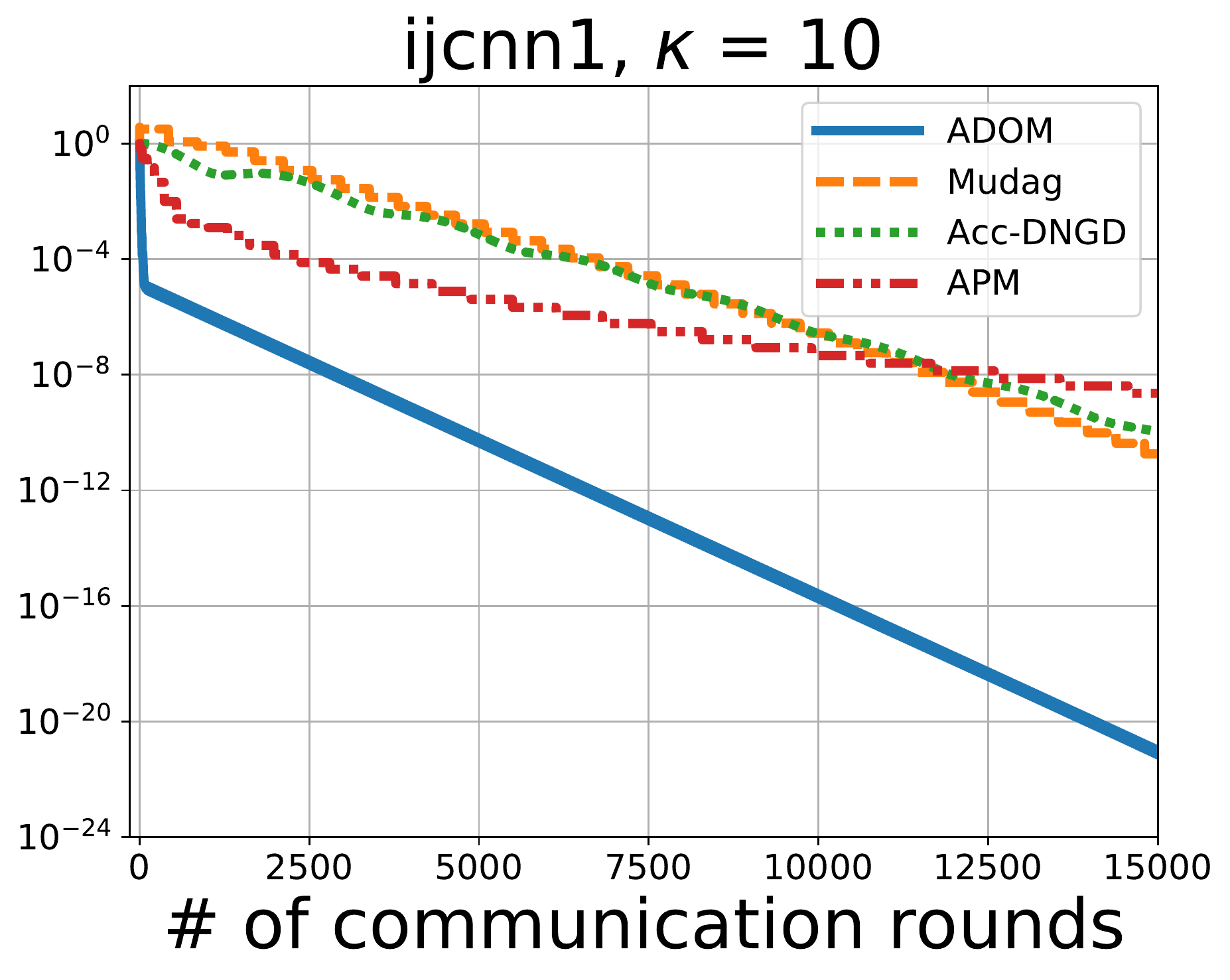}
	\includegraphics[width=0.24\linewidth]{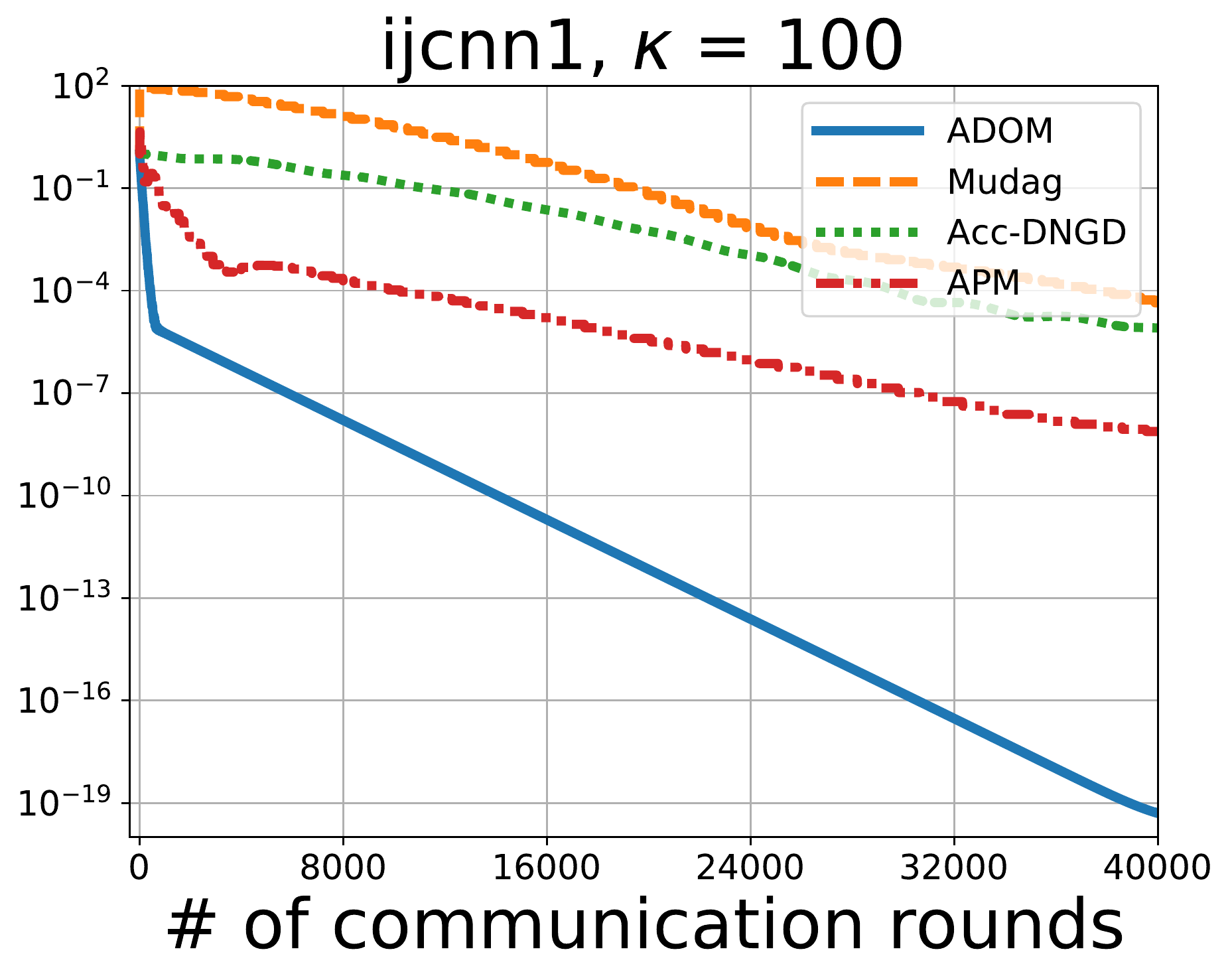}
	\includegraphics[width=0.24\linewidth]{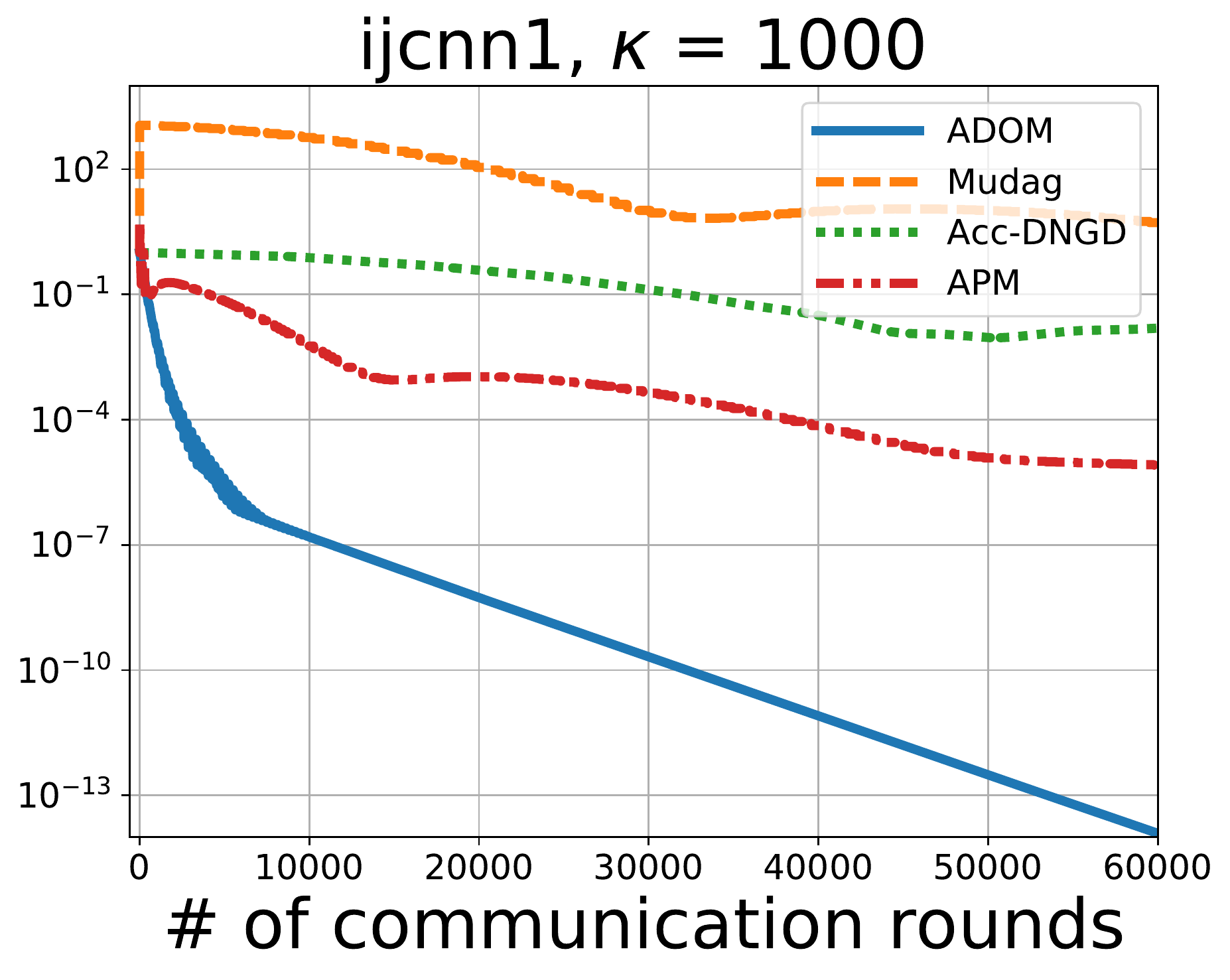}
	\includegraphics[width=0.24\linewidth]{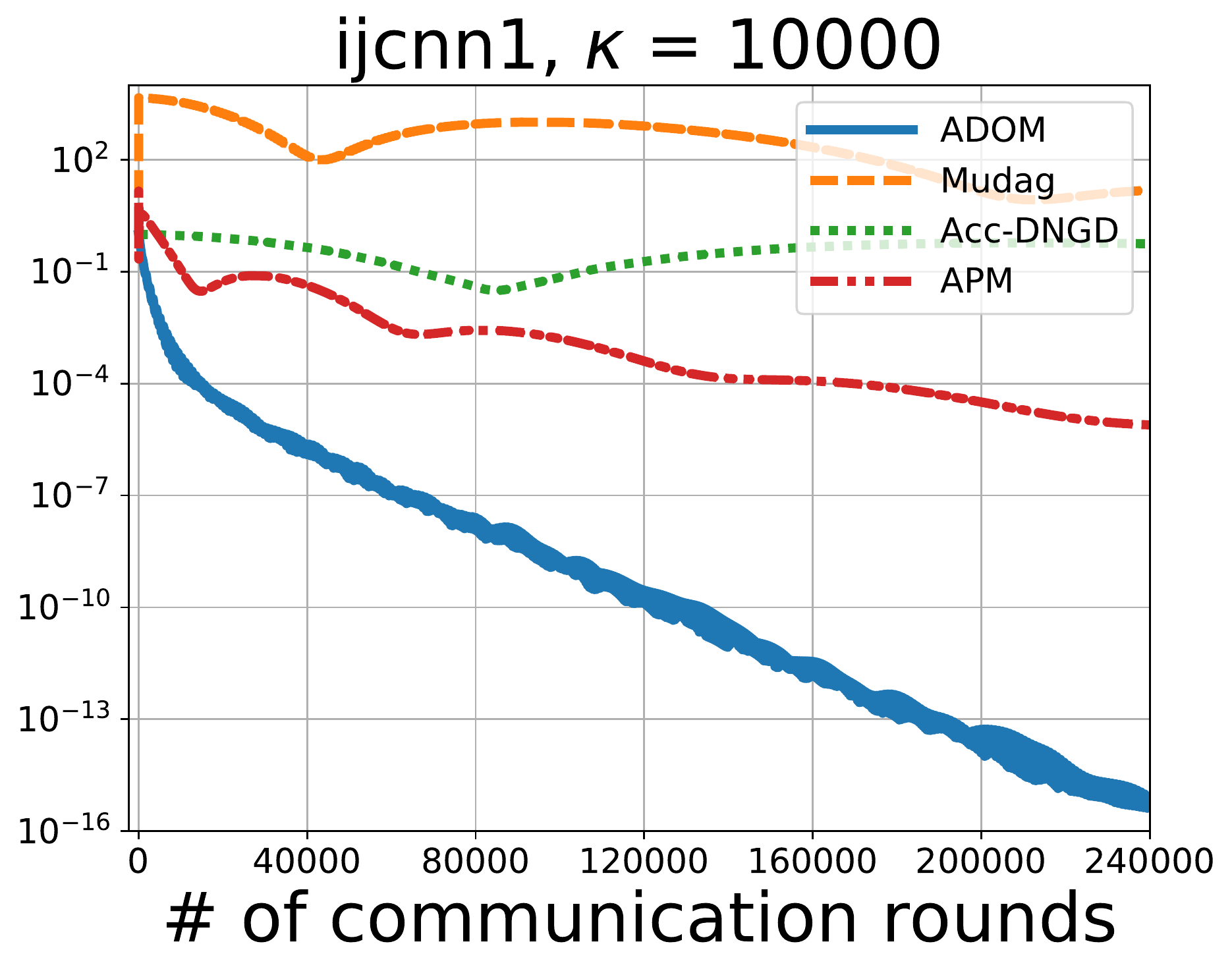}
	
	\caption{Comparison of Mudag, Acc-DNGD, APM and {\sf ADOM} on LIBSVM datasets (\emph{a6a, w6a, ijcnn1} in separate rows) with $\chi \approx 30$, and $\kappa \in \{10,10^2,10^3,10^4\}$ (in different columns).}
	\label{fig:real_kappa}	
\end{figure*}

\begin{figure*}[!htb]
	\centering
	\includegraphics[width=0.24\linewidth]{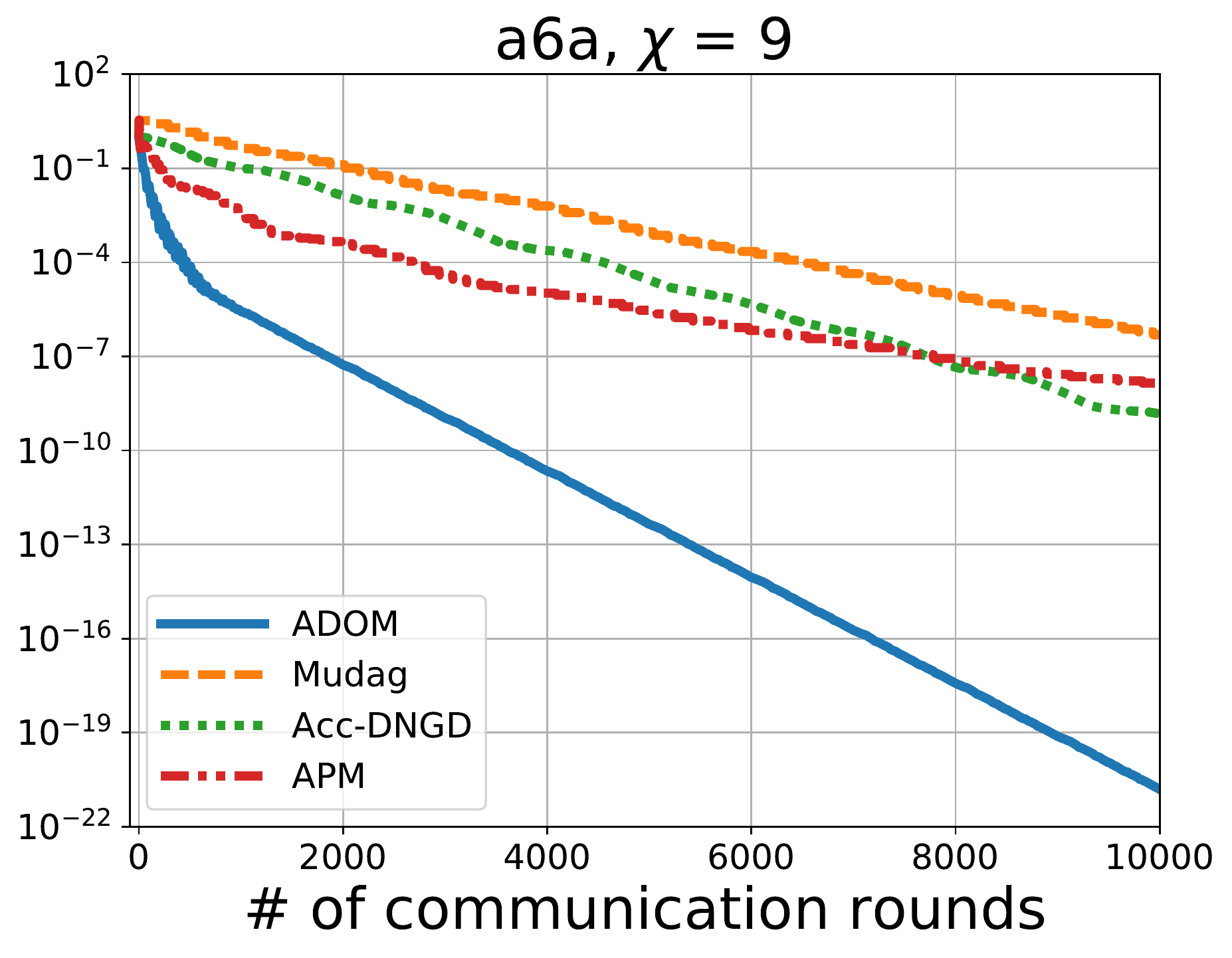}
	\includegraphics[width=0.24\linewidth]{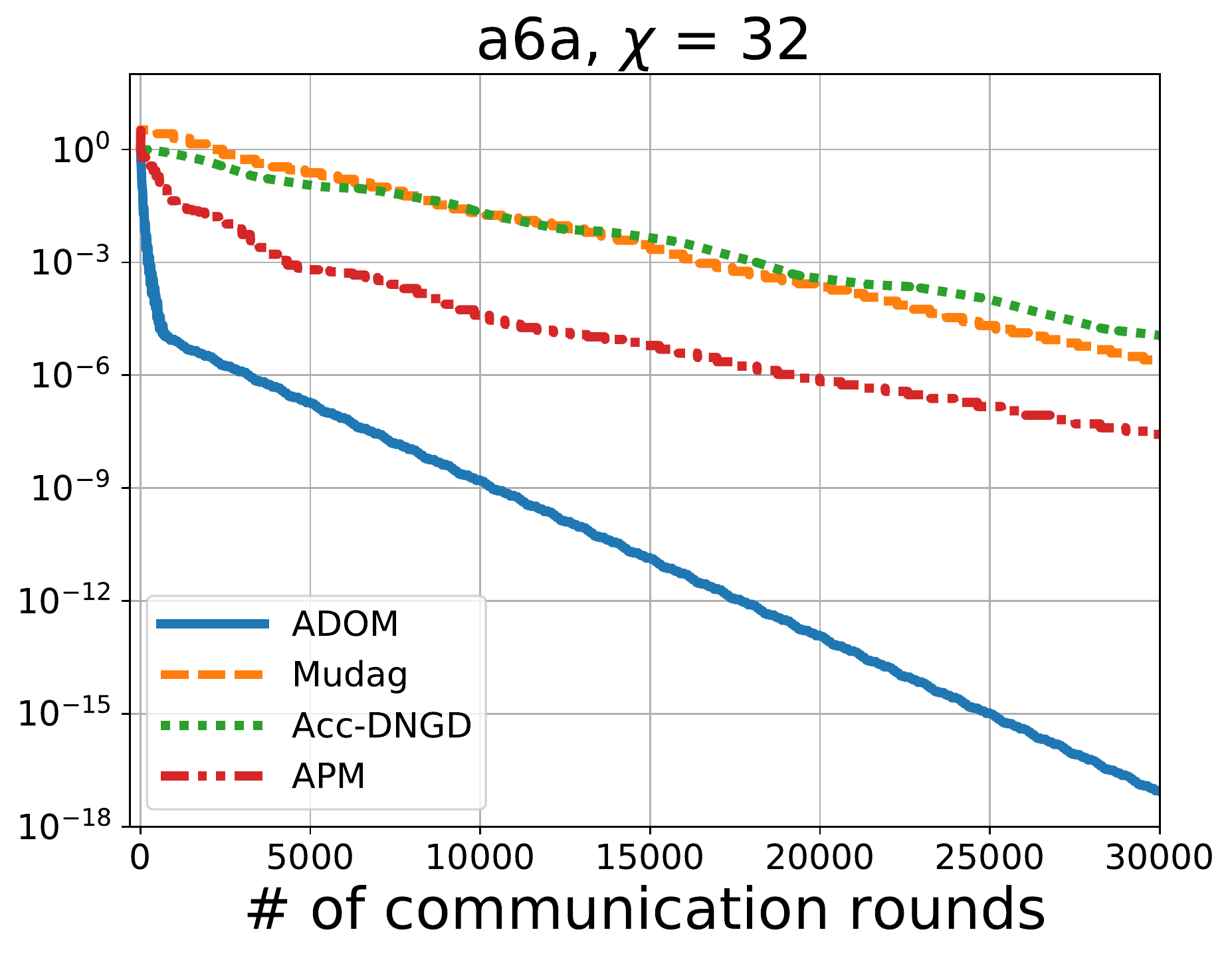}
	\includegraphics[width=0.24\linewidth]{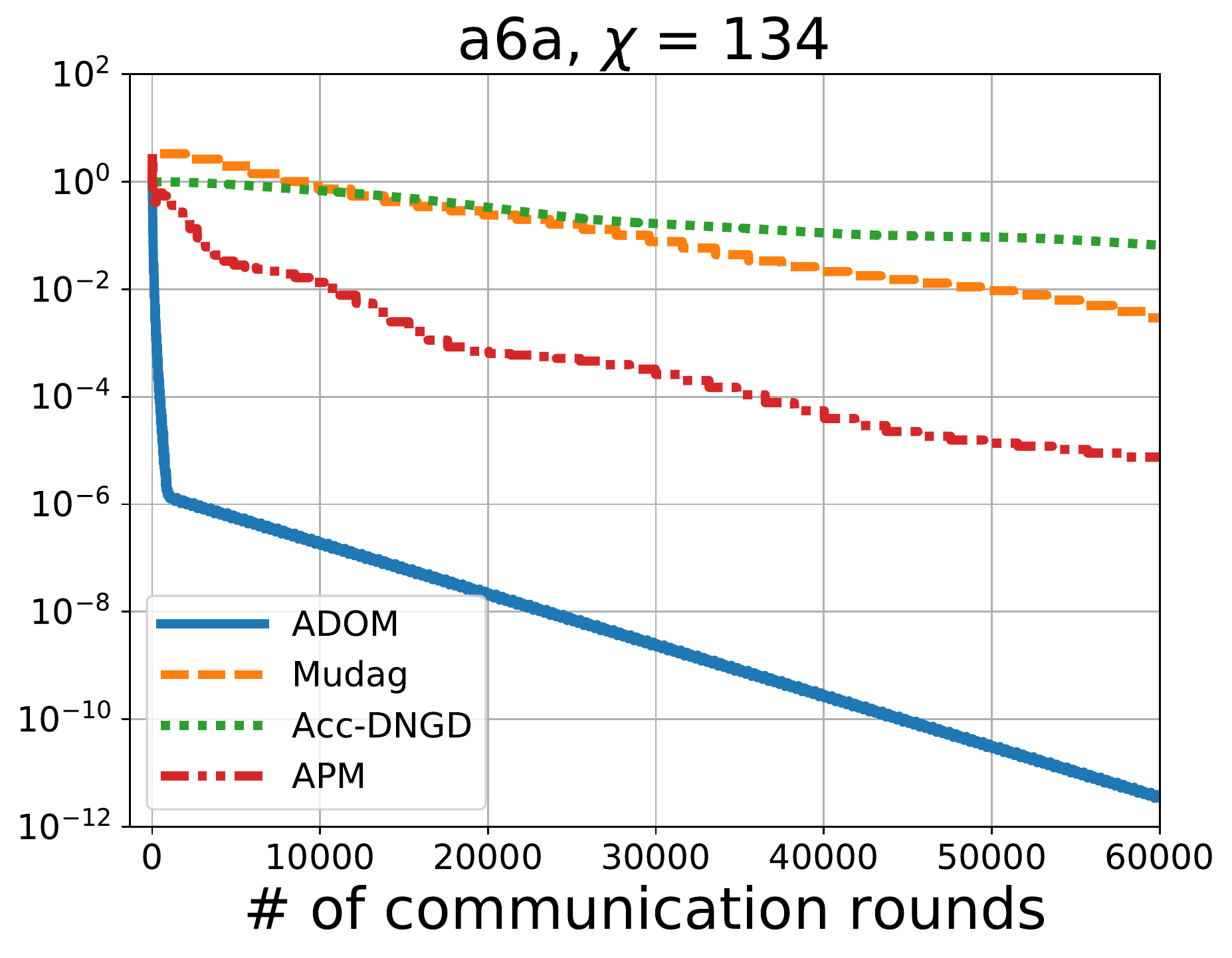}
	\includegraphics[width=0.24\linewidth]{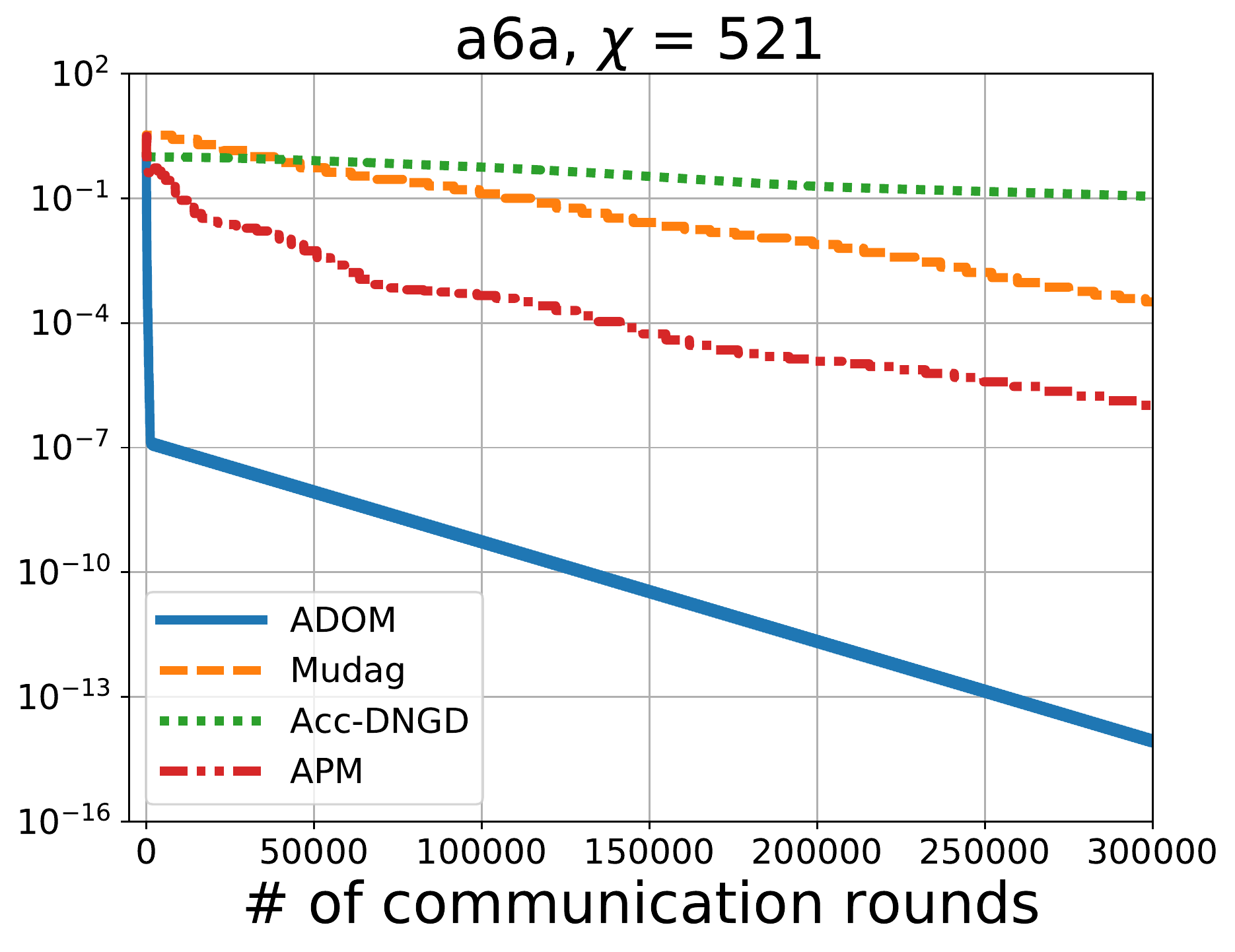}
	
	\includegraphics[width=0.24\linewidth]{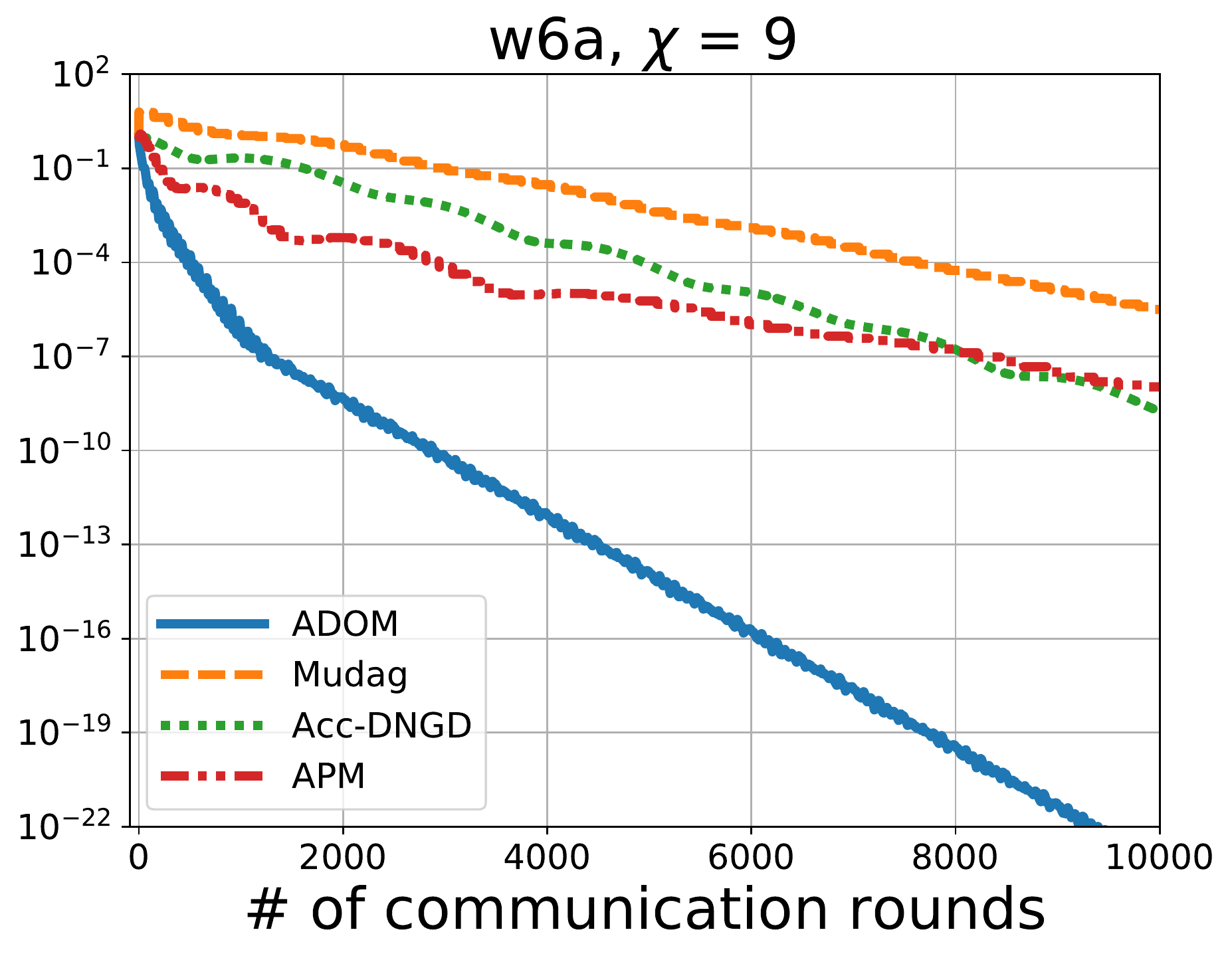}
	\includegraphics[width=0.24\linewidth]{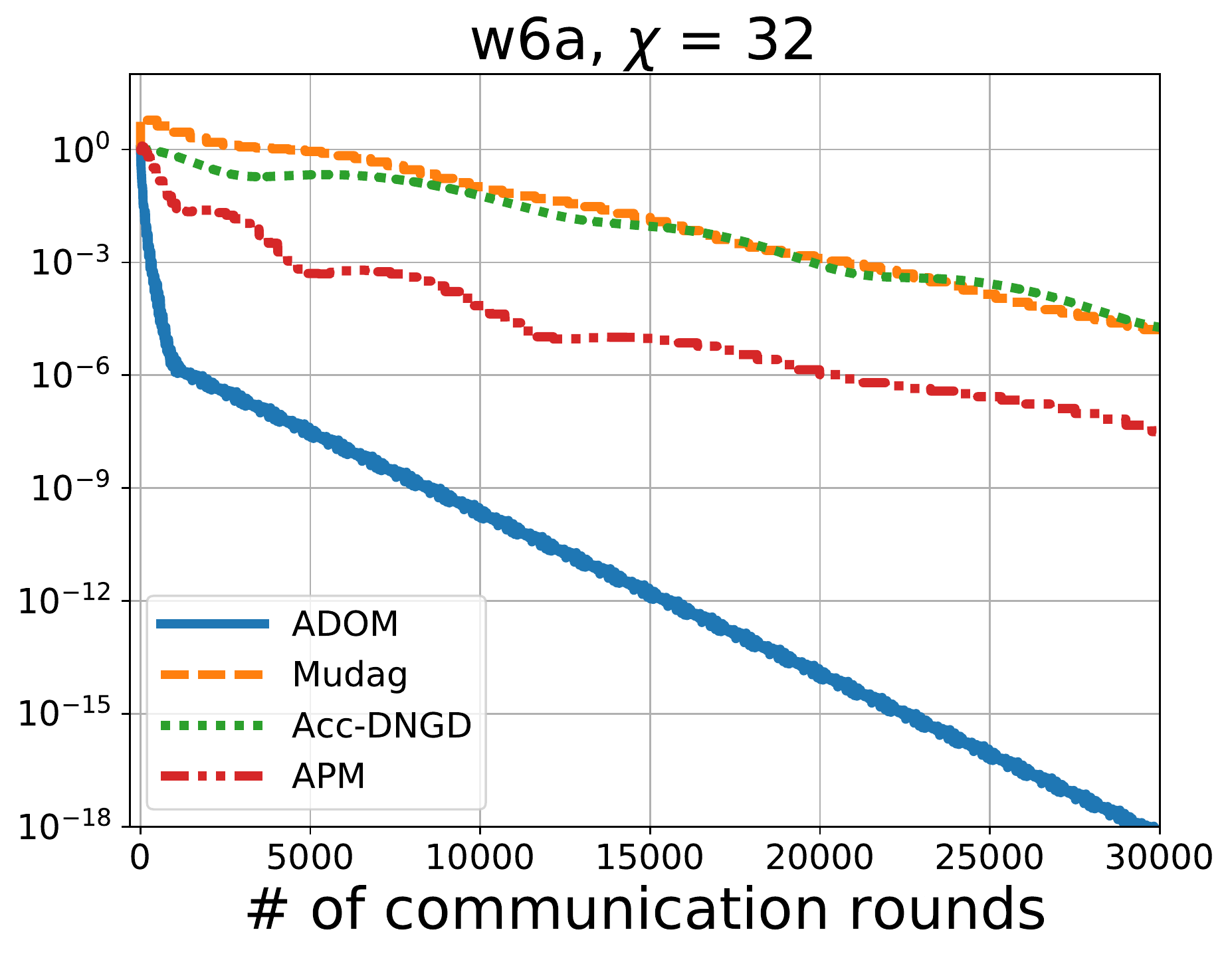}
	\includegraphics[width=0.24\linewidth]{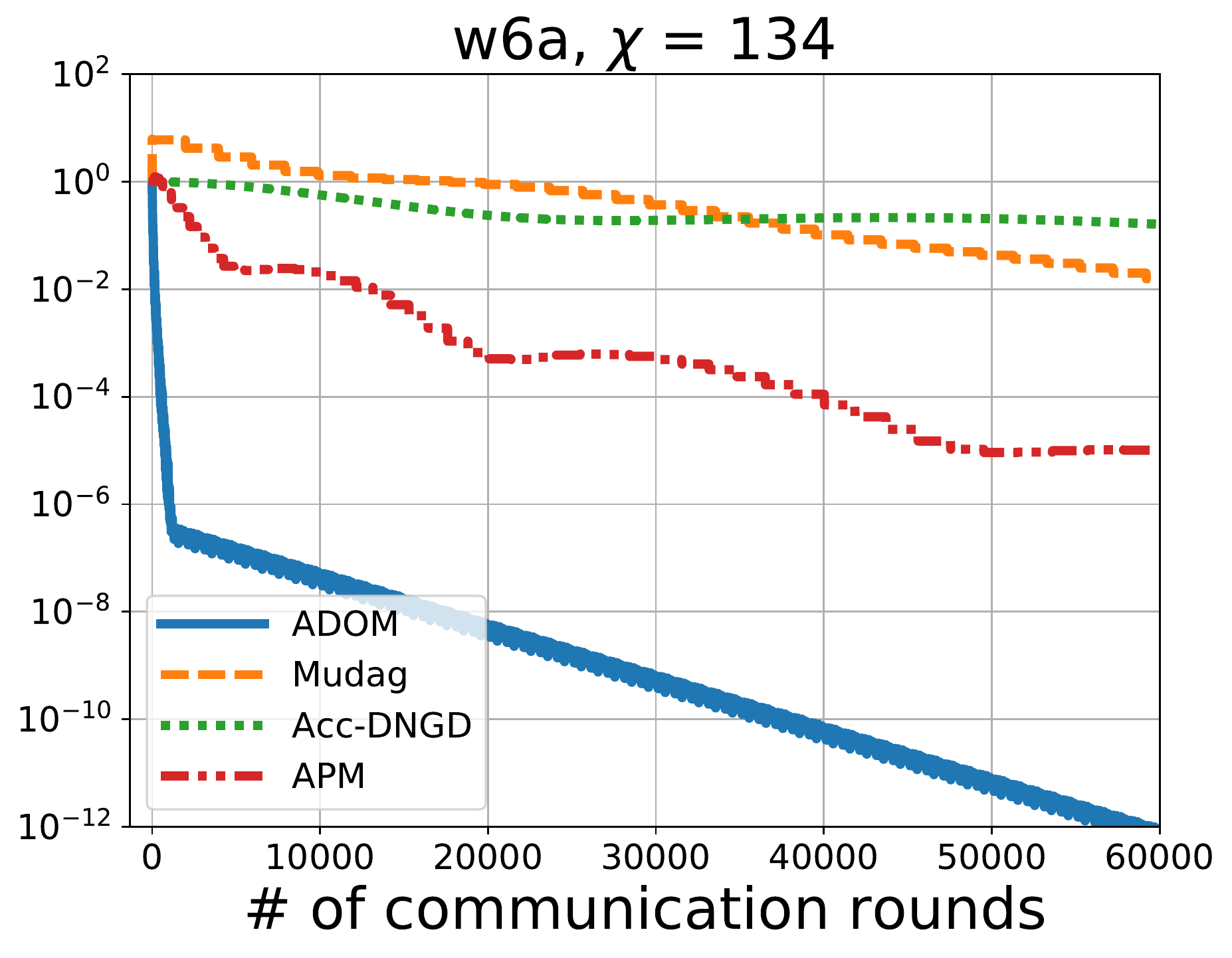}
	\includegraphics[width=0.24\linewidth]{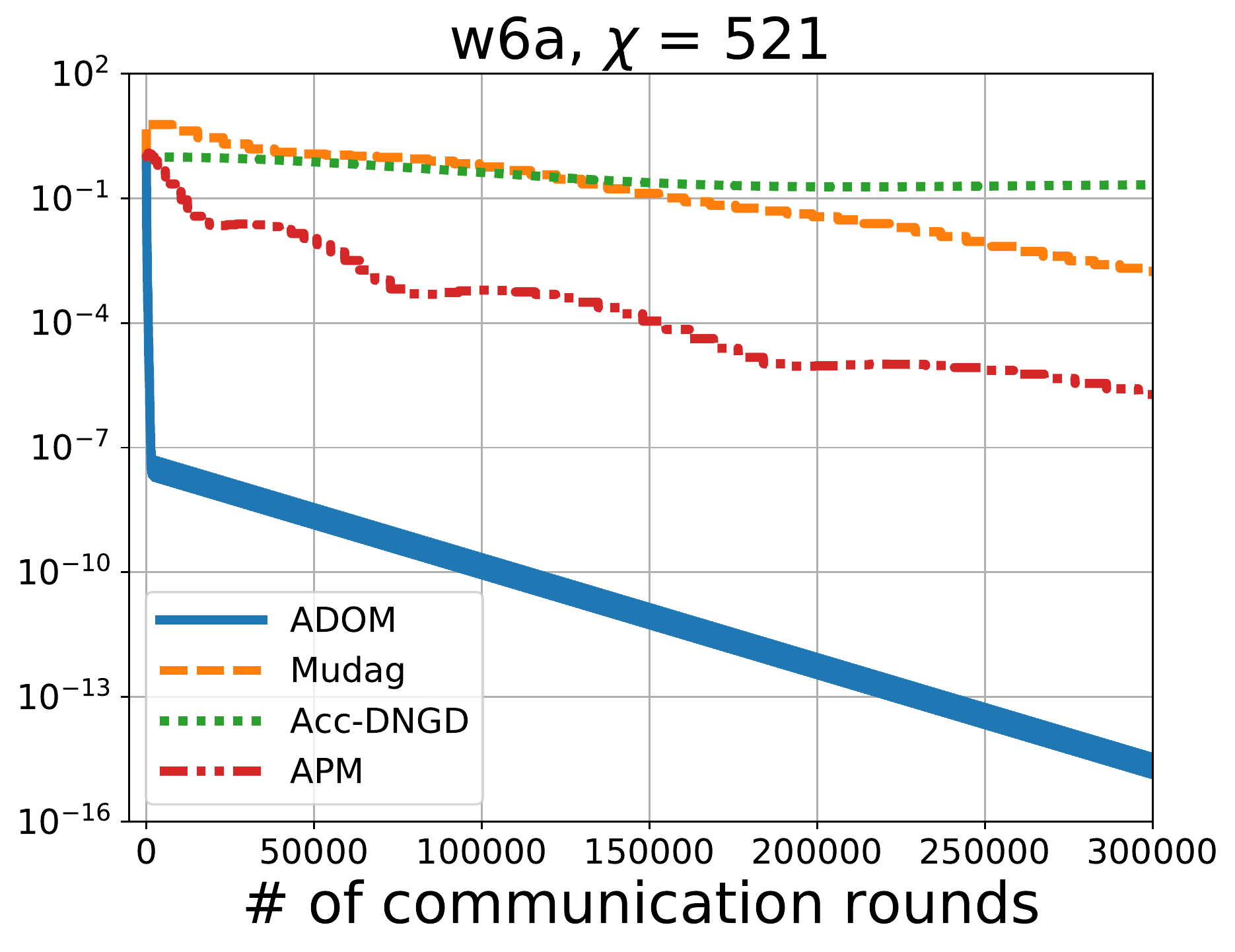}
	
	\includegraphics[width=0.24\linewidth]{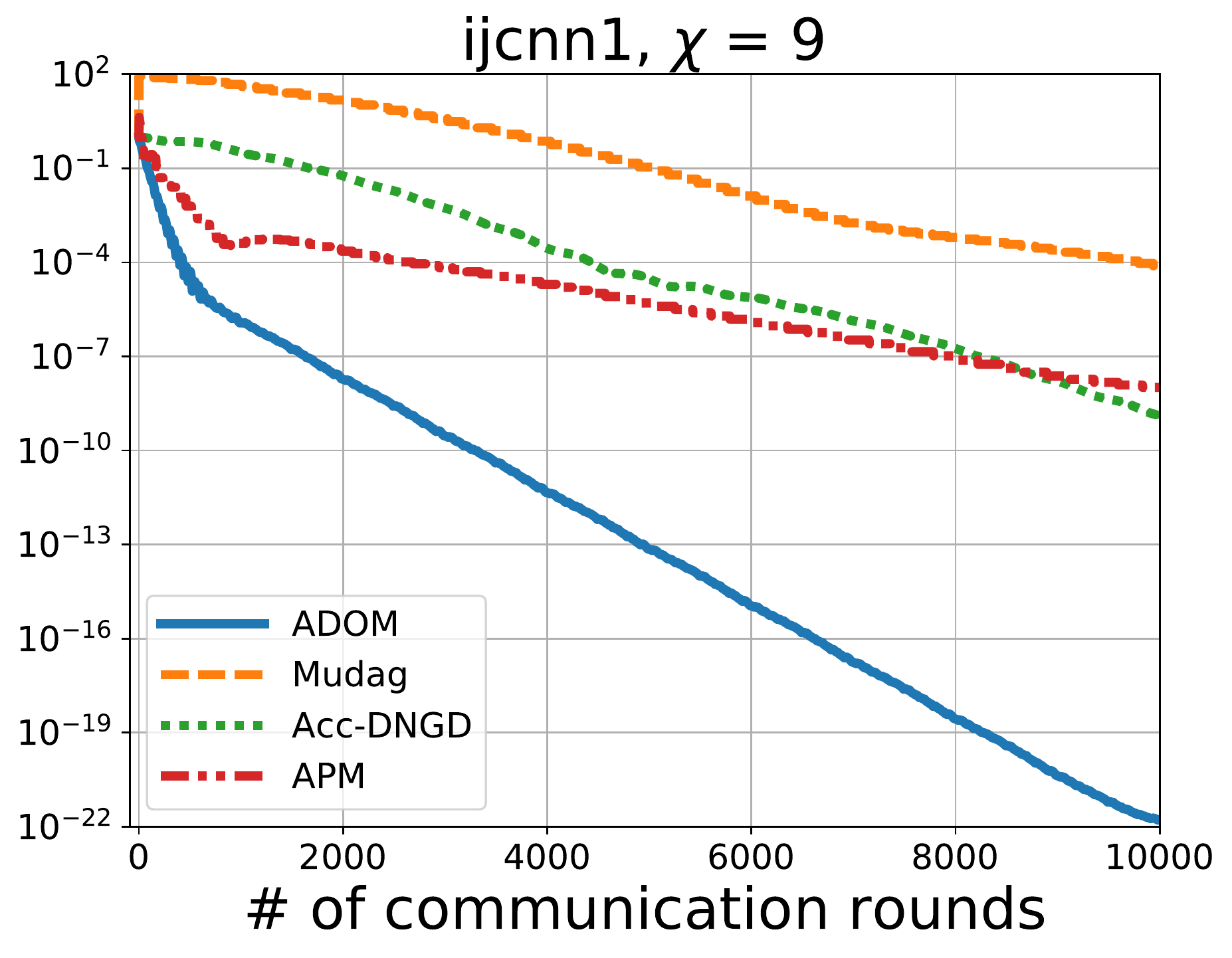}
	\includegraphics[width=0.24\linewidth]{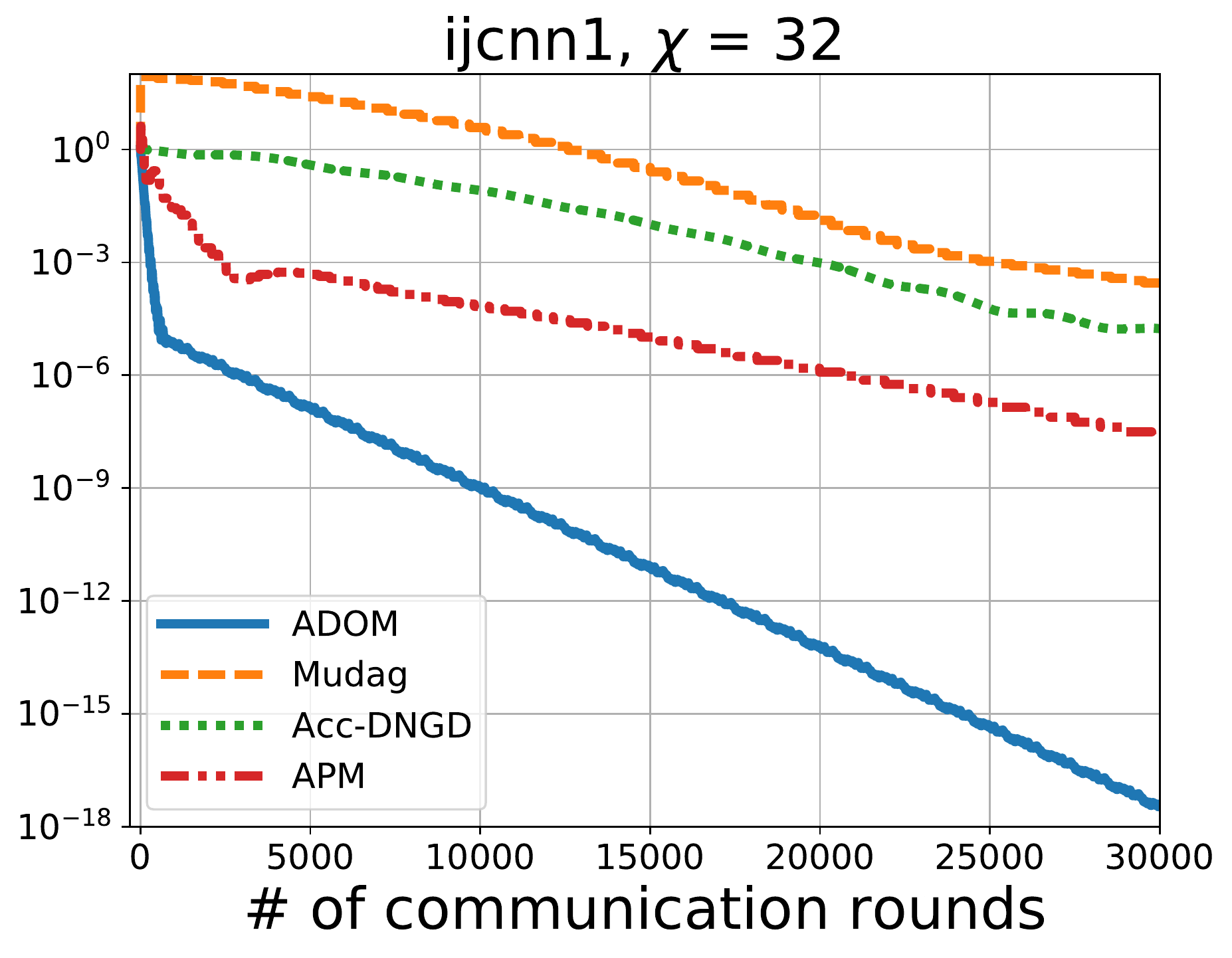}
	\includegraphics[width=0.24\linewidth]{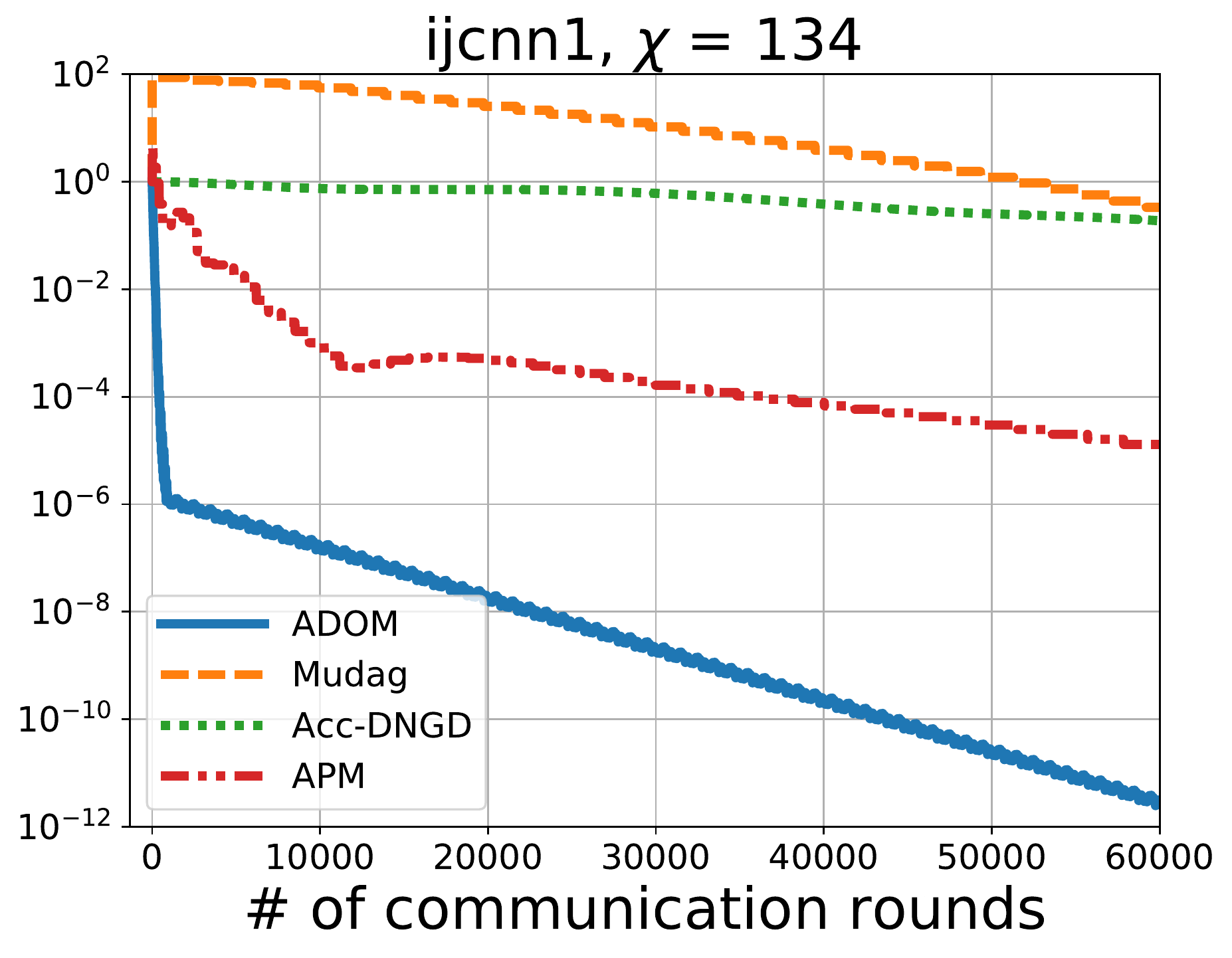}
	\includegraphics[width=0.24\linewidth]{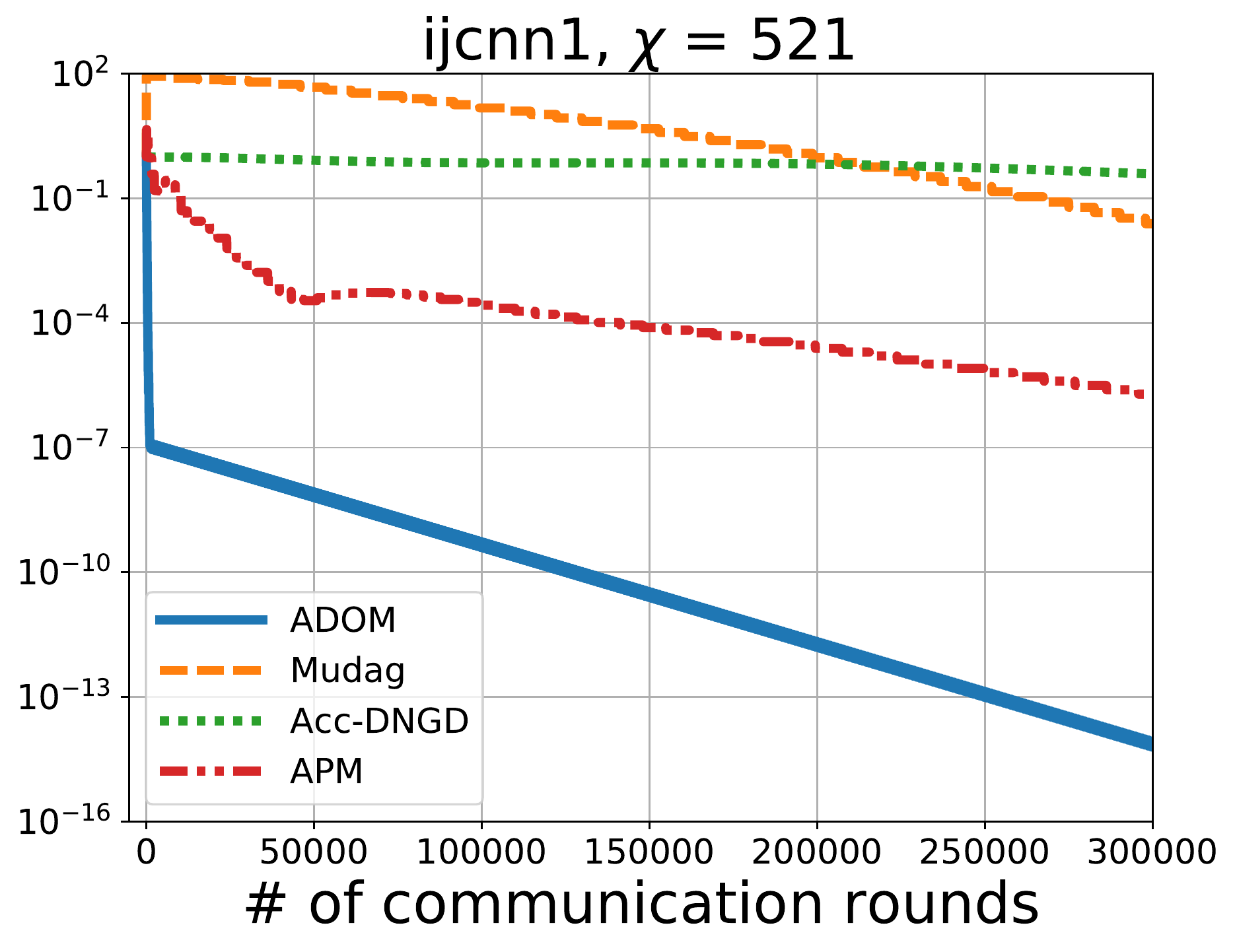}
	
	\caption{Comparison of Mudag, Acc-DNGD, APM and {\sf ADOM} on LIBSVM datasets with $\chi \in \{9, 32,134,521\}$ and $\kappa = 100$.}
	\label{fig:real_chi}	
\end{figure*}

\subsection{Real networks}
For the next set of experiments, we use a real-world temporal graph dataset \emph{infectious\_ct1} representing social interactions from the TUDataset\footnote{Dataset \emph{infectious\_ct1} is available in {\bf Social networks} section at \href{https://chrsmrrs.github.io/datasets/docs/datasets/}{https://chrsmrrs.github.io/datasets/docs/datasets/}.} collection~\citep{morris2020tudataset}. It consists of $200$ graphs $\cG^k$ on $n=50$ nodes with $\chi \approx 232$. For each $k$, matrix $\mW^k$ is chosen to be the Laplacian of graph $\cG^k$ divided by its largest eigenvalue.

Our experimental results are presented in Figure~\ref{fig:real_networks}. We solve the regularized logistic regression problem~\eqref{eq:log_reg} described in Section~\ref{experiments} for $\kappa \in \{10, 10^4\}$ with the same LIBSVM datasets from Section~\ref{real_data}. Overall, the algorithms perform in a similar fashion as in the case of the synthetic geometric graphs. 
Notice that for smaller $\kappa$ ($\kappa=10$), Mudag outperforms APM  after reaching a certain solution accuracy, while for $\kappa=10^3$ the situation is the opposite. Superiority of {\sf ADOM} persists for every dataset and condition number $\kappa$. We use one  iteration ($T=1$) of GD to calculate $\g F^*(z_g^k)$ in {\sf ADOM}.

\begin{figure*}[!htb]
	\centering
	\includegraphics[width=0.24\linewidth]{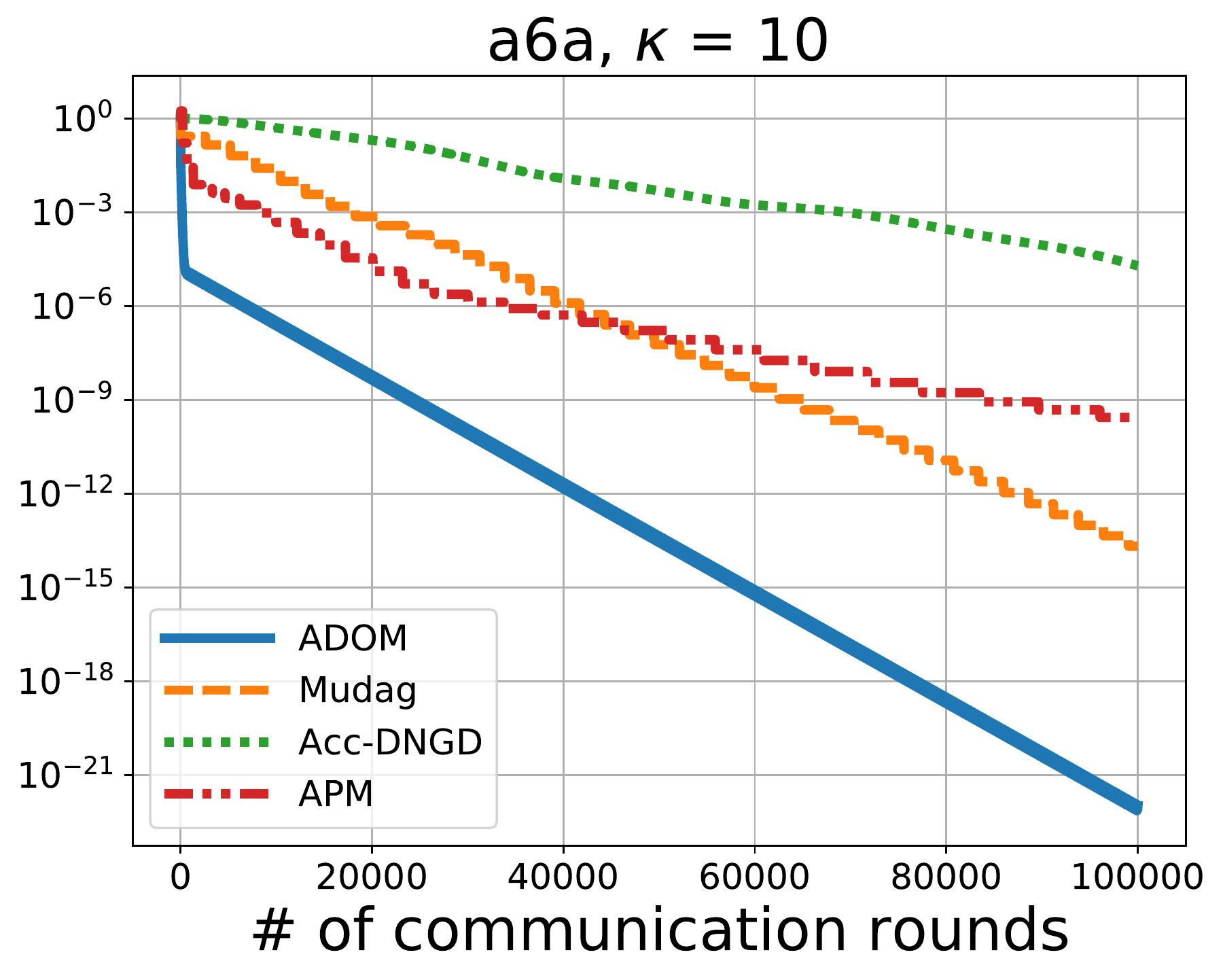}
	\includegraphics[width=0.24\linewidth]{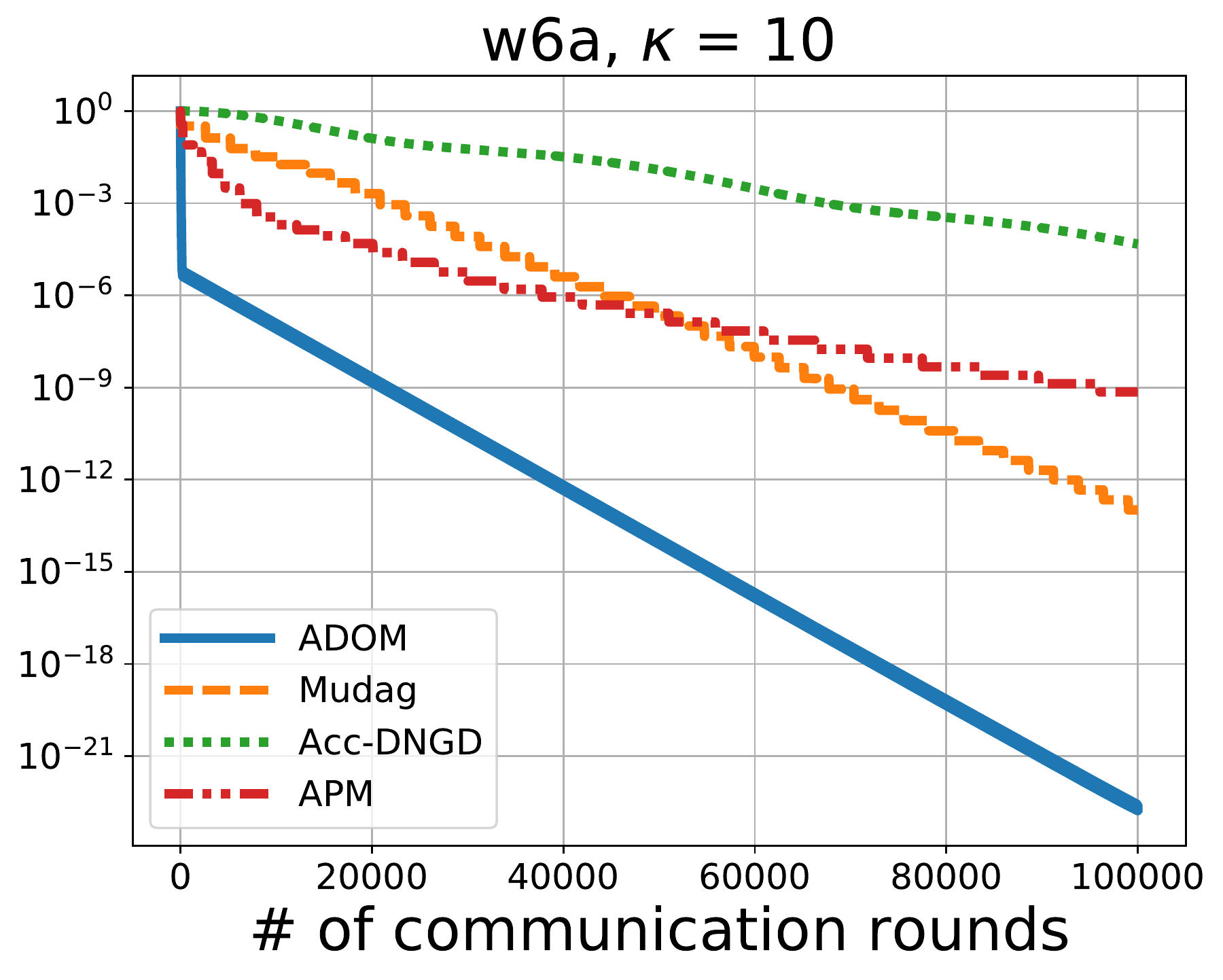}
	\includegraphics[width=0.24\linewidth]{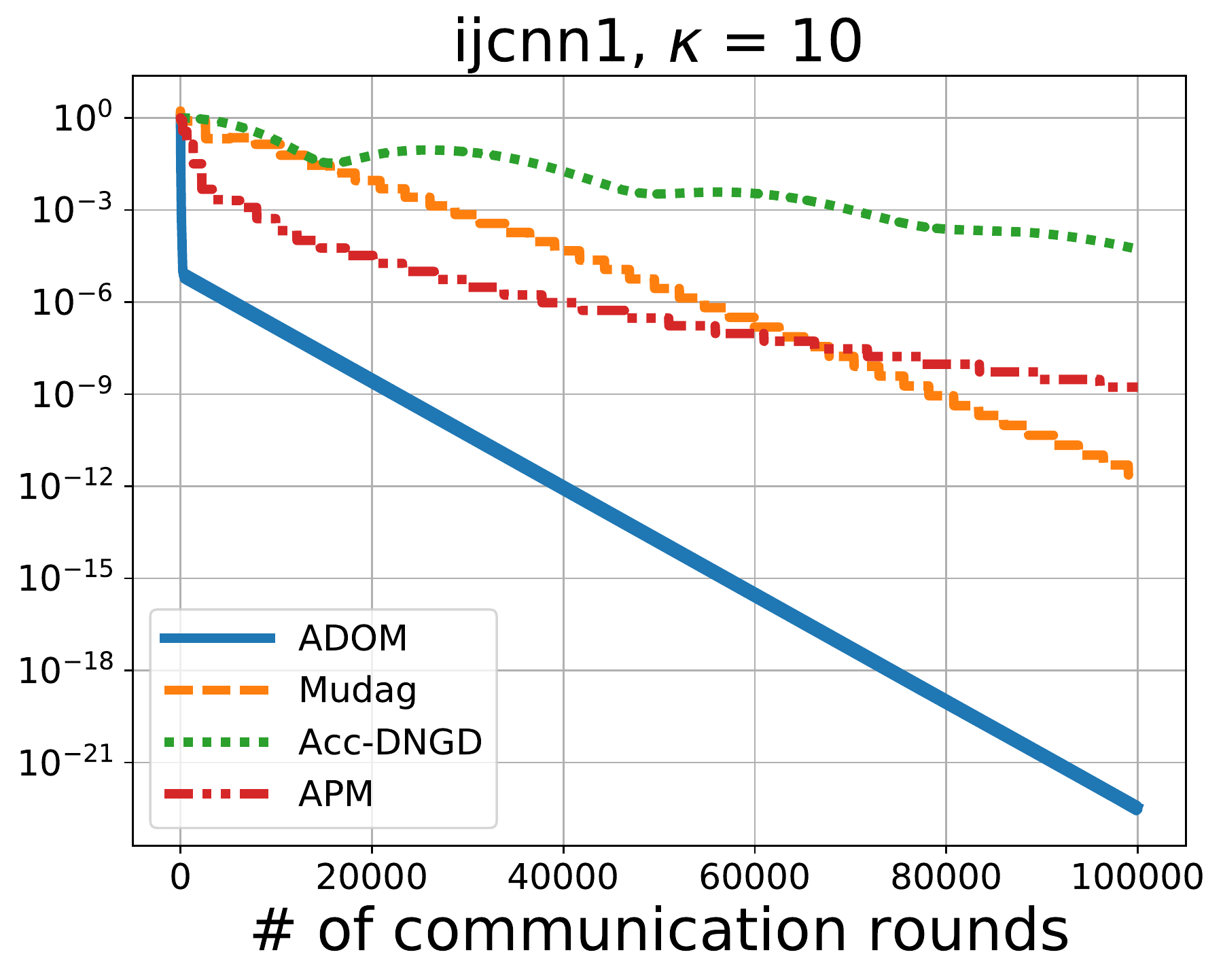}
	
	\includegraphics[width=0.24\linewidth]{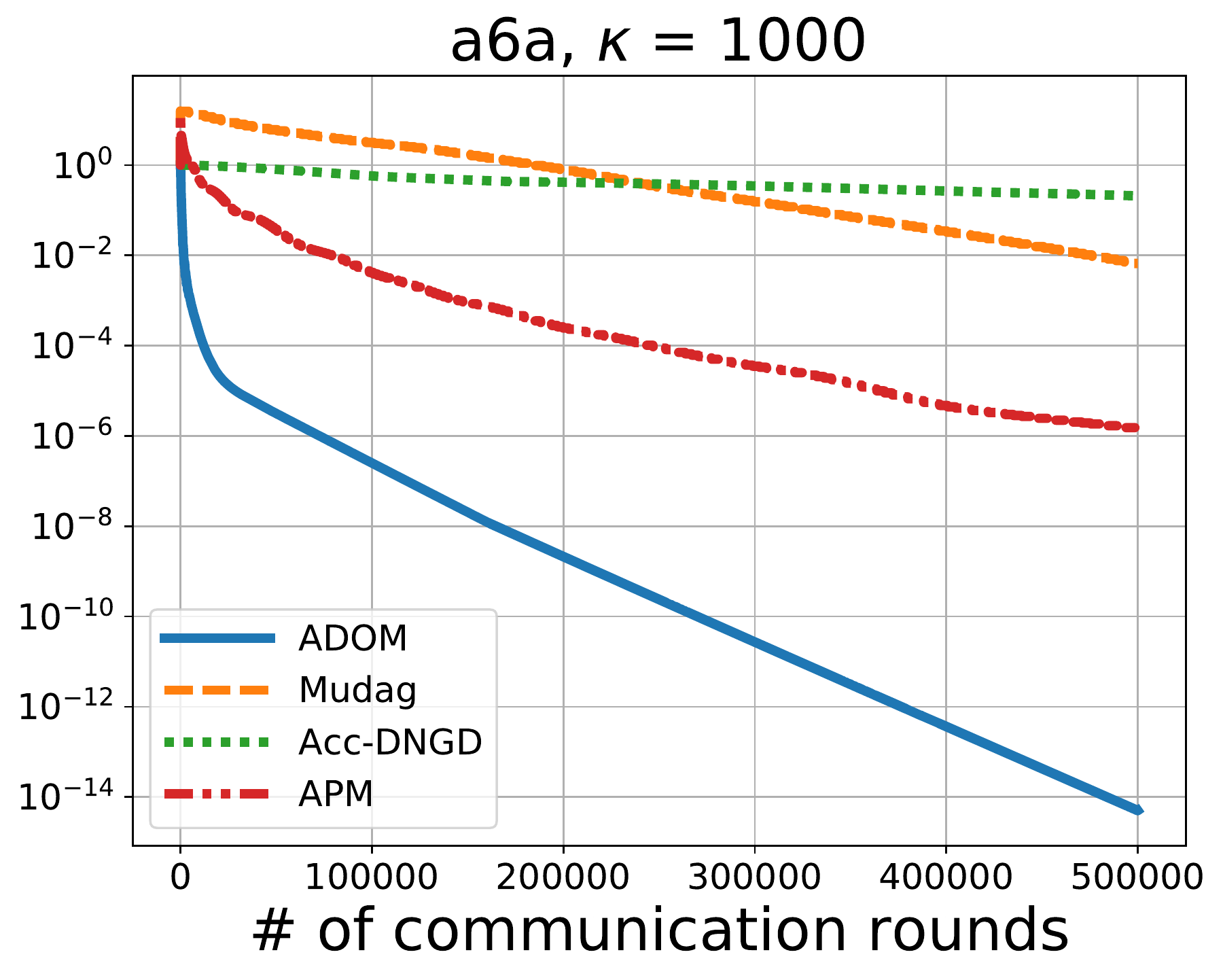}
	\includegraphics[width=0.24\linewidth]{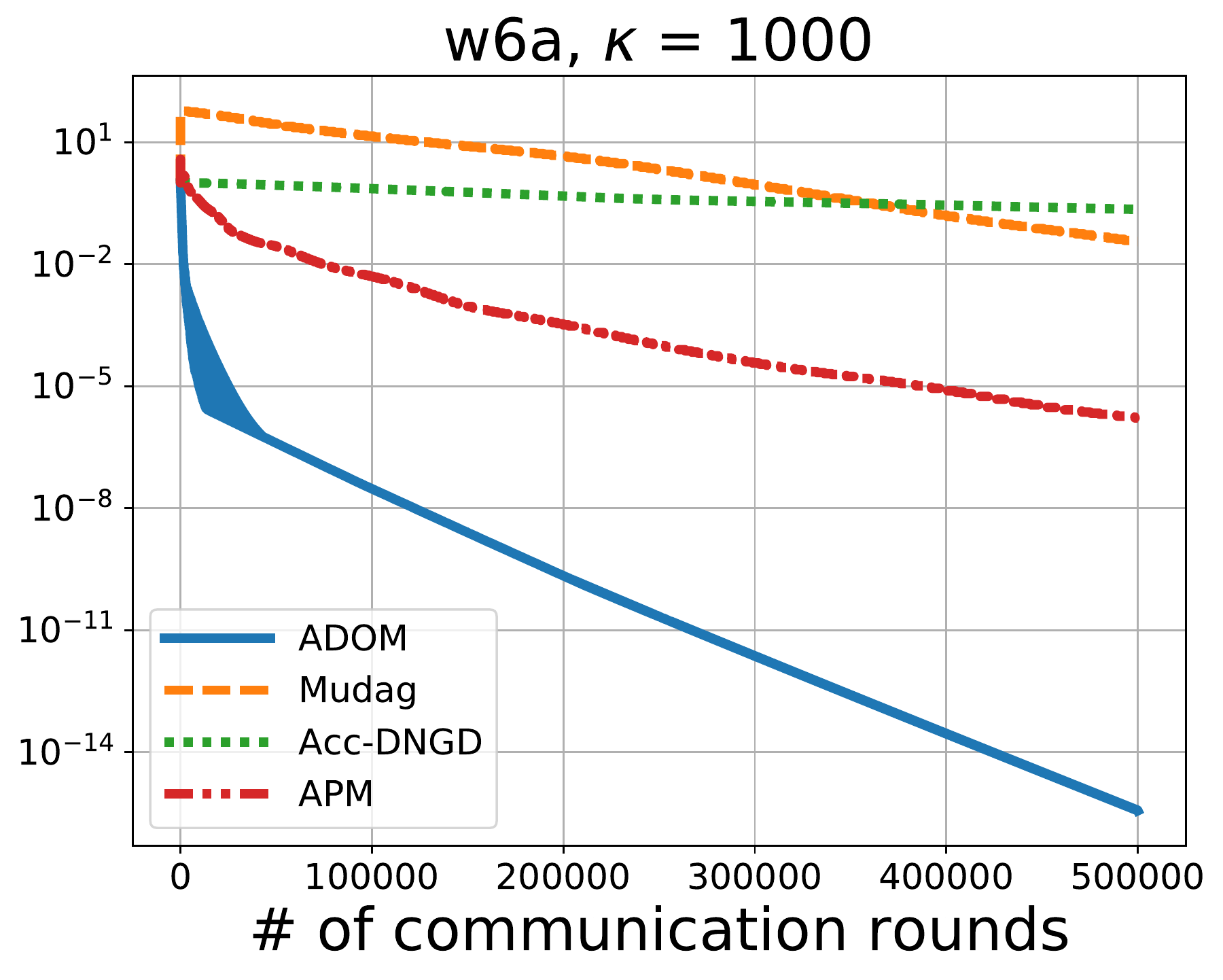}
	\includegraphics[width=0.24\linewidth]{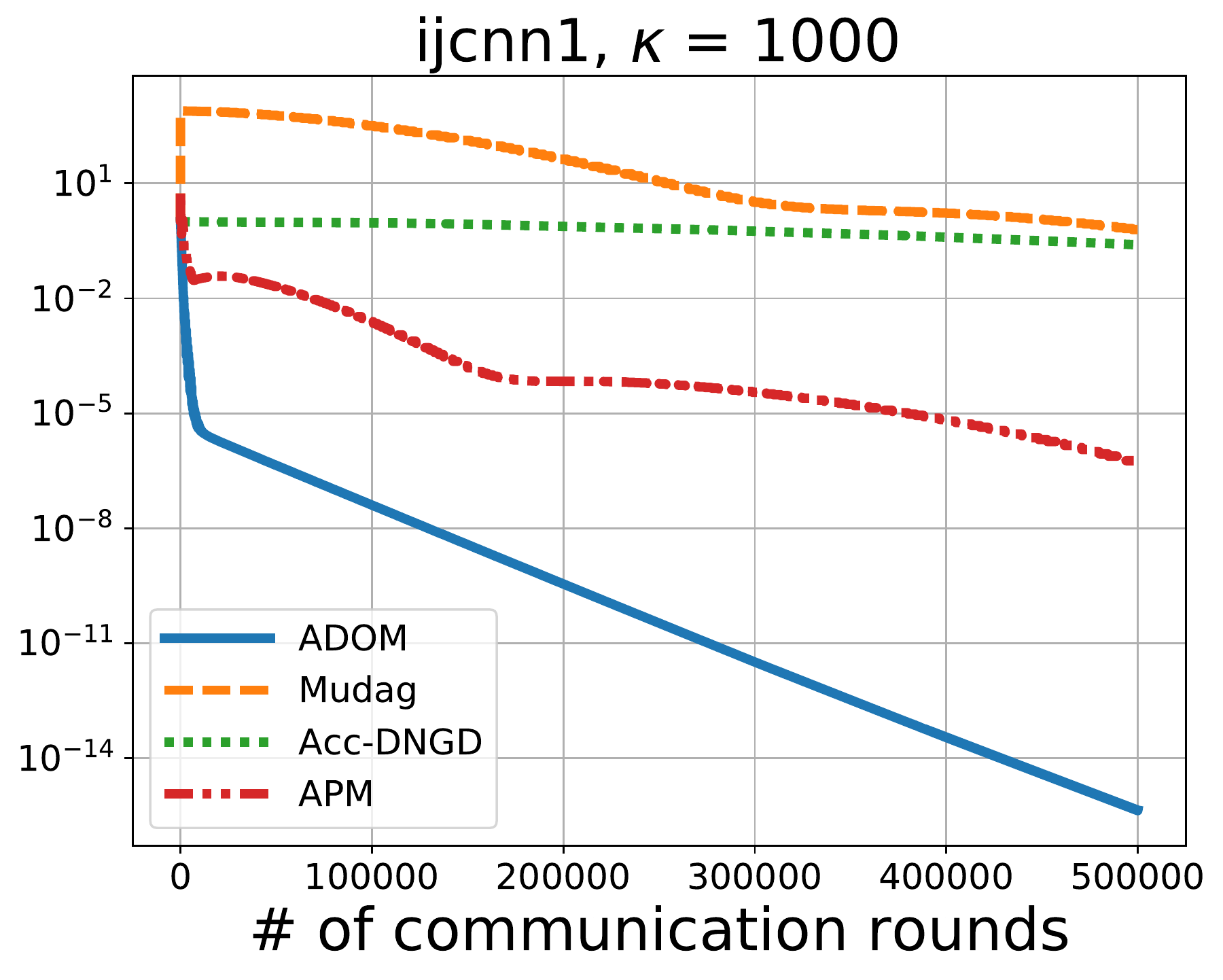}
			
	\caption{Comparison of Mudag, Acc-DNGD, APM and {\sf ADOM} on temporal graph dataset \emph{infectious\_ct1} with $\chi \approx 232$ and $\kappa \in \{10, 10^4\}$.}
	\label{fig:real_networks}	
\end{figure*}

\end{document}

%% file: commands.tex
\usepackage{mathtools}
\mathtoolsset{showmanualtags}
\usepackage{amsthm}
\usepackage{xcolor}
\usepackage{tabularx}
\usepackage{bbm}
\usepackage[colorinlistoftodos,bordercolor=orange,backgroundcolor=orange!20,linecolor=orange,textsize=scriptsize]{todonotes}
\usepackage{amsmath}
\usepackage{amssymb}
\usepackage{makecell}
\usepackage{tikz}

\allowdisplaybreaks[1]

\newtheorem{theorem}{Theorem}
\newtheorem{lemma}{Lemma}

\newcommand{\proj}{\mathrm{proj}}
\newcommand{\range}{\mathrm{range}}

\newcommand{\lmax}{\lambda_{\max}}

\newcommand{\lminp}{\lambda_{\min}^+}

\newcommand{\g}{\nabla}
\newcommand{\ones}{\mathbf{1}}

\newcommand{\R}{\mathbb{R}}

\def\<#1,#2>{\langle #1,#2\rangle}

\newcommand{\norm}[1]{\|#1\|}
\newcommand{\sqn}[1]{\norm{#1}^2}

\newcommand{\cC}{\mathcal{C}}
\newcommand{\cN}{\mathcal{N}}

\newcommand{\cG}{\mathcal{G}}

\newcommand{\cV}{\mathcal{V}}
\newcommand{\cE}{\mathcal{E}}
\newcommand{\cL}{\mathcal{L}}
\newcommand{\cQ}{\mathcal{Q}}

\newcommand{\cO}{\mathcal{O}}

\newcommand{\mW}{\mathbf{W}}

\newcommand{\mI}{\mathbf{I}}
\newcommand{\mP}{\mathbf{P}}

\newcommand{\eqdef}{\coloneqq}
\DeclareMathOperator*{\argmin}{arg\,min}